\documentclass[a4paper,11pt,reqno]{amsart}
\usepackage{amsmath}
\usepackage{amstext}
\usepackage{amsbsy}
\usepackage{amsopn}
\usepackage{upref}
\usepackage{amsthm}
\usepackage{amsfonts}
\usepackage{amssymb}
\usepackage{mathrsfs}
\usepackage{times}
\allowdisplaybreaks
 \newtheorem{theorem}{Theorem}

 \newtheorem{lemma}[theorem]{Lemma}
\theoremstyle{definition}
 
\theoremstyle{remark}
 
\newcommand{\ep}{\varepsilon}
\newcommand{\p}{\partial}

\newcommand{\re}{\operatorname{Re}}
\newcommand{\im}{\operatorname{Im}}
\begin{document}
\title[Semilinear Schr\"odinger Equations]{Gain of analyticity for semilinear Schr\"odinger equations}
\author[H.~Chihara]{Hiroyuki Chihara}
\thanks{Supported by JSPS Grant-in-Aid for Scientific Research \#20540151.}
\address{Mathematical Institute,  
         Tohoku University, 
         Sendai 980-8578, Japan}
\email{chihara@math.tohoku.ac.jp}
\begin{abstract}
We discuss gain of analyticity phenomenon of solutions to 
the initial value problem for semilinear Schr\"odinger equations 
with gauge invariant nonlinearity. 
We prove that if the initial data decays exponentially, 
then the solution becomes real-analytic in the space variable 
and a Gevrey function of order $2$ in the time variable 
except in the initial plane. 
Our proof is based on the energy estimates developed in our previous work 
and on fine summation formulae concerned with a matrix norm. 
\end{abstract}
\maketitle
\section{Introduction}
\label{section:introduction}
In this paper we study the gain of regularity phenomenon of solutions to 
the initial value problem for semilinear Schr\"odinger equations 
of the form: 
\begin{alignat}{2}
   \p_tu-i\Delta u 
 & = 
   f(u,\p{u}) 
&
   \quad \text{in} \quad 
 & (-T,T)\times\mathbb{R}^n, 
\label{equation:pde1}
\\ 
   u(0,x) 
 & = 
   u_0(x) 
& 
   \quad \text{in} \quad 
 & \mathbb{R}^n,  
\label{equation:data1}
\end{alignat}
where $u(t,x)$ is a complex-valued unknown function of 
$(t,x)=(t,x_1,\dotsc,x_n)\in[-T,T]\times\mathbb{R}^n$, 
$T>0$, 
$i=\sqrt{-1}$, 
$\p_t=\p/\p{t}$, 
$\p_j=\p/\p{x_j}$ ($j=1,\dotsc,n$), 
$\p=(\p_1,\dotsc,\p_n)$,   
$\Delta=\p_1^2+\dotsb+\p_n^2$ 
and 
$n$ is the space dimension. 
Throughout this paper, we assume that 
the nonlinearity $f(u,v)$ is a real-analytic function 
on $\mathbb{R}^{2+2n}$ having a holomorphic extension on
$\mathbb{C}^{2+2n}$, and that $f(u,v)$ satisfies 
$$
f(u,v)=O(\lvert{u}\rvert^3+\lvert{v}\rvert^3)
\quad\text{near}\quad
(u,v)=0, 
$$
\begin{equation}
f(e^{i\theta}u,e^{i\theta}v)
=
e^{i\theta}f(u,v)
\quad\text{for}\quad
(u,v)\in\mathbb{C}^{1+n}, 
\theta\in\mathbb{R}. 
\label{equation:gauge}
\end{equation}
For $z=(u,v)\in\mathbb{C}^{1+n}$ and  any multi-index 
$\alpha=(\alpha_0,\dotsc,\alpha_n)\in(\mathbb{N}\cup\{0\})^{1+n}$, 
we denote 
$$
\lvert\alpha\rvert=\alpha_0+\dotsb+\alpha_n, 
\quad
z^\alpha=u^{\alpha_0}v_1^{\alpha_1}\dotsb{v}_n^{\alpha_n}.
$$ 
It follows from our hypothesis on the nonlinearity 
that $f(z)$ is given by  
\begin{equation}
f(z)
=
\sum_{p=1}^\infty
\sum_{\substack{\lvert\alpha\rvert=p+1
                \\
                \rvert\beta\rvert=p}}
f_{\alpha\beta}z^\alpha\bar{z}^\beta, 
\quad
f_{\alpha\beta}\in\mathbb{C}, 
\label{equation:nonlinearity}
\end{equation}
and that for any $R>0$ there exists $C_R>0$ such that 
$$
A_p
\equiv
\sum_{\substack{\lvert\alpha\rvert=p+1
                \\
                \rvert\beta\rvert=p}}
\lvert{f_{\alpha\beta}}\rvert
\leqslant
C_RR^{-(2p+1)}, 
\quad
p=1,2,3\dotsc.
$$
\par
Here we introduce notation. 
Let $\theta$ and $l$ be real numbers. 
$H^{\theta,l}$ is the set of 
all tempered distributions on $\mathbb{R}^n$ satisfying 
$$
\lVert{u}\rVert_{\theta,l}^2
=
\int_{\mathbb{R}^n}
\lvert
\langle{x}\rangle^l
\langle{D}\rangle^\theta
u(x)
\rvert^2
dx
<+\infty, 
$$
where 
$\langle{x}\rangle=\sqrt{1+\lvert{x}\rvert^2}$, 
$\lvert{x}\rvert=\sqrt{x_1^2+\dotsb+x_n^2}$ 
and 
$\langle{D}\rangle=(1-\Delta)^{1/2}$. 
In particular, set 
$H^\theta=H^{\theta,0}$, 
$\lVert\cdot\rVert_\theta=\lVert\cdot\rVert_{\theta,0}$, 
$L^2=H^0$ for short. 
$\lVert\cdot\rVert$ and $(\cdot,\cdot)$ denote 
the $L^2$-norm and the $L^2$-inner product respectively. 
In this paper we treat not only scalar-valued functions 
but also vector-valued ones. 
The $(L^2)^m$-norm and the $(L^2)^m$-inner product is denoted 
by the same notation:
$$
(u,v)
=
\int_{\mathbb{R}^n}
\sum_{j=1}^mu_j(x)\bar{v}_j(x)dx, 
\quad
\lVert{u}\rVert
=
\sqrt{(u,u)}
$$ 
for 
$u={}^t[u_1,\dotsc,u_m]$ 
and 
$v={}^t[v_1,\dotsc,v_m]$. 
Let $X$ be a Banach space, 
and let $k$ be a nonnegative integer. 
$C^k(I;X)$ denotes the set of all $X$-valued $C^k$-functions 
on the interval $I$. 
In particular, set $C(I;X)=C^0(I;X)$ for short. 
For any real number $s$, 
$[s]$ is the largest integer not greater than $s$. 
\par
In the previous paper \cite{chihara} the author studied 
the finite gain of regularity of solutions to 
\eqref{equation:pde1}-\eqref{equation:data1}. 
Loosely speaking, 
if $u_0(x)=o(\lvert{x}\rvert^{-l})$ 
as $\lvert{x}\rvert\rightarrow\infty$ 
with some positive integer $l$, 
then 
the solution $u$ gains spatial smoothness of order $l$ 
locally in $x$ when $t\ne0$. 
More precisely, we proved the following. 
\begin{theorem}[\cite{chihara}]
\label{theorem:previous}
Let $\theta>n/2+3$, 
and let $l$ be a nonnegative integer. 
Then for any $u_0\in{H}^{\theta,l}$, 
there exists $T>0$ depending only on $\lVert{u_0}\rVert_\theta$ such that 
{\rm \eqref{equation:pde1}-\eqref{equation:data1}} 
has a unique solution 
$u{\in}C([-T,T];H^\theta)$ 
satisfying 
$$
\langle{x}\rangle^{-\lvert\alpha\rvert}\p^\alpha{u}
\in
C([-T,T]\setminus\{0\};H^\theta)
\quad\text{for}\quad
\lvert\alpha\rvert\leqslant{l}.
$$
\end{theorem}
This type of properties of dispersive equations 
have been investigated in the last two decades. 
See, e.g., the references in \cite{chihara}. 
For local existence theorems for more general 
semilinear Schr\"odinger-type equations, 
see \cite{kpv}, \cite{rolvung}, \cite{t'joen}. 
More recently, in \cite{hnp} Hayashi, Naumikin and Pipolo 
studied the infinite version of Theorem~\ref{theorem:previous} 
for one-dimensional equations with small initial data. 
Roughly speaking, they proved that if 
$u_0$ is small and 
$u_0(x)=o(e^{-\lvert{x}\rvert})$ 
as $\lvert{x}\rvert\rightarrow\infty$, 
then the solution $u$ becomes real-analytic in $x$ for $t\ne0$.  
The purpose of this paper is to prove the infinite version of 
Theorem~\ref{theorem:previous} 
without smallness condition on the initial data
and the restriction on the space dimension. 
Our main results are the following. 
\begin{theorem}
\label{theorem:main}
Let $\theta$ and $s$ be positive numbers satisfying 
$\theta>n/2+3$ and $s\geqslant1$ respectively, 
and let $\ep$ be an arbitrary positive number. 
For any $u_0$ satisfying 
$\exp(\ep\langle{x}\rangle^{1/s})u_0\in{H^\theta}$, 
there exist a positive time $T$ depending only on 
$\lVert{u_0}\rVert_\theta$, 
and a unique solution 
$u{\in}C([-T,T];H^\theta)$ 
to {\rm \eqref{equation:pde1}-\eqref{equation:data1}}. 
Moreover there exist positive constants 
$M$ and $\rho$ 
such that 
\begin{equation}
\lVert
\langle{x}\rangle^{-2m-\lvert\alpha\rvert}
\p_t^m\p^\alpha{u(t)}
\rVert_\theta
\leqslant
M(\rho{t})^{-(2m+\lvert\alpha\rvert)}
m!^{2s}\alpha!^s
\label{equation:gevrey}
\end{equation}
for 
$t\in[-T,T]\setminus\{0\}$, 
$m\in\mathbb{N}\cup\{0\}$, 
$\alpha\in(\mathbb{N}\cup\{0\})^n$.
\end{theorem}
Our condition on the Gevrey exponent does not seem to be optimal. 
Indeed, in \cite{hk} Hayashi and Kato studied the case $s=1/2$ 
for the equation of the form 
$$
\p_tu-i\Delta{u}=f(u), 
$$ 
and proved that the solution becomes real-analytic in 
$([-T,T]\setminus\{0\})\times\mathbb{R}^n$. 
Moreover, it is interesting that if 
$\exp(\ep\langle{x}\rangle^{1/s})u_0\in{H^\theta}$ 
for $s\geqslant1/2$, 
then $e^{it\Delta}u_0$ satisfies \eqref{equation:gevrey}.  
Using this fact, we can construct nonlinear equations 
whose solutions have the same regularity property. 
See Section~\ref{section:remark} for the detail. 
For more information about gain of regularity phenomenon of 
dispersive equations, see 
\cite{doi2}, \cite{doi3}, \cite{kajitani}, \cite{kw}, 
\cite{mns}, \cite{mrz}, \cite{rz1}, \cite{rz2}, \cite{szeftel}, 
\cite{takuwa} and references therein. 
\par
Our method of proof of Theorem~\ref{theorem:main} 
is basically due to the energy method developed in \cite{chihara}. 
We shall show the uniform bound of $\{w_l\}_{l=0,1,2,\dotsc}$, 
where  
$$
w_l
=
{}^t
[
r^{\lvert\alpha\rvert}\alpha!^{-s}
\langle{D}\rangle^\theta
J^\alpha{u},
\overline{
r^{\lvert\alpha\rvert}\alpha!^{-s}
\langle{D}\rangle^\theta
J^\alpha{u}
}
]_{\lvert\alpha\rvert\leqslant{l}}, 
$$
$$
J_k=x_k+2it\p_k, 
\quad
J=(J_1,\dotsc,J_n), 
$$
and $r$ is a positive constant. 
\eqref{equation:gevrey} immediately follows 
from the uniform bound of $\{w_l\}$ and the equation \eqref{equation:pde1}. 
\par
This paper is organized as follows. 
In Section~\ref{section:psdo} we present the elementary facts on 
pseudodifferential operators associated with nonlinearities. 
Section~\ref{section:preliminaries} is devoted to studying 
fine summation properties used in the uniform estimates later. 
In Section~\ref{section:linear} we refine the energy method 
for some linear systems in \cite{chihara}. 
Section~\ref{section:nonlinear} is devoted to the estimates of
nonlinearity in Gevrey classes. 
In Section~\ref{section:energy} we we obtain the uniform energy estimates. 
In Section~\ref{section:proof} we complete 
the proof of Theorem~\ref{theorem:main}. 
Finally, in Section~\ref{section:remark} we give an interesting example 
of semilinear Schr\"odinger equations related with the exponent $s=1/2$. 

\section{Pseudodifferential operators associated with nonlinear PDEs}
\label{section:psdo}
In this section we recall the Kato-Ponce commutator estimates 
established in \cite{kp}, and pseudodifferential calculus 
developed in \cite{chihara}. 
In addition we present some rough estimates associated with 
the Leibniz formula for pseudodifferential operators 
with constant coefficients. 
One can refer to \cite{cm} and \cite{taylor2} 
for the infromation related to this section. 
\par
Let $m$ be a real number. 
$S^m$ denotes the set of all smooth functions 
on $\mathbb{R}^n\times\mathbb{R}^n$ satisfying 
$$
\lvert\p_x^\beta\p_\xi^\alpha{p}(x,\xi)\rvert
\leqslant
C_{\alpha\beta}
\langle\xi\rangle^{m-\lvert\alpha\rvert}
$$
for any multi-indices $\alpha$ and $\beta$. 
For a symbol $p(x,\xi)$, 
a pseudodifferential operator $p(x,D)$ is defined by 
$$
p(x,D)u(x)
=
(2\pi)^{-n}
\iint_{\mathbb{R}^{2n}}
e^{i(x-y)\cdot\xi}p(x,\xi)u(y)dyd\xi,
$$
where $x\cdot\xi=x_1\xi_1+\dotsb+x_n\xi_n$. 
See \cite{hoermander}, \cite{kumanogo} and \cite{taylor} for the detail. 
We first recall pseudodifferential operators with nonsmooth coefficients 
and their properties needed later. 
Let $\sigma\geqslant0$. 
$\mathscr{B}^\sigma$ is the set of 
all $C^{[\sigma]}$-functions on $\mathbb{R}^n$ satisfying 
$$
\lVert{f}\rVert_{\mathscr{B}^\sigma}
=
\begin{cases}
\displaystyle\sup_{x\in\mathbb{R}^n}
\displaystyle\sum_{\lvert\alpha\rvert\leqslant\sigma}
\lvert\p^\alpha{f(x)}\rvert
<+\infty
&
\text{if}\quad
\sigma=0,1,2,\dotsc,
\\
\displaystyle\sup_{x{\in}X}
\displaystyle\sum_{\lvert\alpha\rvert\leqslant[\sigma]}
\lvert\p^\alpha{f(x)}\rvert
&
\\
\quad
+
\displaystyle\sup_{\substack{x,y{\in}X\\{x\ne{y}}}}
\displaystyle\sum_{\lvert\alpha\rvert=[\sigma]}
\frac{\lvert\p^\alpha{f(x)}-\p^\alpha{f(y)}\rvert}
     {\lvert{x-y}\rvert^{\sigma-[\sigma]}}
<+\infty
&
\text{otherwise.}
\end{cases}
$$
Similarly, $\mathscr{B}^\sigma{S^m}$ denotes 
the set of all functions on $\mathbb{R}^n\times\mathbb{R}^n$ satisfying 
$$
\lVert{p}\rVert_{\mathscr{B}^\sigma{S^m},l}
=
\sup_{\substack{\xi\in\mathbb{R}^n
                \\
                \lvert\alpha\rvert\leqslant{l}}}
\lVert
\langle\xi\rangle^{\lvert\alpha\rvert-m}
\p_\xi^\alpha{p(\cdot,\xi)}
\rVert_{\mathscr{B}^\sigma}
<+\infty
$$
for $l=0,1,2,\dotsc$. 
$\mathscr{S}$ denotes the set of Schwartz functions on $\mathbb{R}^n$, 
and $L^p$ denotes the usual Lebesgue space for all $p\in[1,\infty]$. 
In \cite{nagase} Nagase introduced larger classes of symbols, 
and proved the $L^p$ boundedness theorem 
by his symbol smoothing technique. 
We make full use of $L^2$-version of them. 
\begin{theorem}[Nagase {\cite[Theorem~A]{nagase}}]
\label{theorem:nagase} 
Let $q(x,\xi)$ be a function on $\mathbb{R}^n\times\mathbb{R}^n$. 
Suppose that there exist $\tau$ and $\lambda$ satisfying 
$0\leqslant\tau<\lambda\leqslant1$ 
such that 
\begin{align*}
  \lvert\p_\xi^\alpha{q(x,\xi)}\rvert
& \leqslant
  C_\alpha
  \langle\xi\rangle^{-\lvert\alpha\rvert},
\\
  \lvert\p_\xi^\alpha{q(x,\xi)}-\p_\xi^\alpha{q(y,\xi)}\rvert
& \leqslant
  C_\alpha
  \langle\xi\rangle^{-\lvert\alpha\rvert+\tau}
  \lvert{x-y}\rvert^\lambda 
\end{align*} 
for $\lvert\alpha\rvert\leqslant{n+1}$. 
Then 
$$
\lVert{q(x,D)u}\rVert_{L^2(\mathbb{R}^n)}
\leqslant
A(q)
\lVert{u}\rVert_{L^2(\mathbb{R}^n)}
$$
for any $u{\in}L^2(\mathbb{R}^n)$, 
where $A(q)$ depends only on 
\begin{align*}
& \sum_{\lvert\alpha\rvert\leqslant{n+1}}
  \sup_{x,\xi\in\mathbb{R}^n}
  (\langle\xi\rangle^{\lvert\alpha\rvert}
   \lvert\p_\xi^\alpha{q(x,\xi)}\rvert)
\\
& \quad
  +
  \sum_{\lvert\alpha\rvert\leqslant{n+1}}
  \sup_{\substack{x,y,\xi\in\mathbb{R}^n\\{x{\ne}y}}}
  \left(
  \langle\xi\rangle^{\lvert\alpha\rvert-\tau}
  \frac{\lvert\p_\xi^\alpha{q(x,\xi)}-\p_\xi^\alpha{q(y,\xi)}\rvert}
       {\lvert{x-y}\rvert^\lambda}
  \right).   
\end{align*}
\end{theorem}
Combining Nagase's idea and results, and 
well-known facts on smooth symbols, 
one can obtain the fundamental theorem for algebra 
and the sharp G{\aa}rding inequality. 
\begin{lemma}[Chihara {\cite[Lemma~2]{chihara}}]
\label{theorem:algebra} 
Let $\sigma>1$. 
If $p_j(x,\xi)\in\mathscr{B}^\sigma{S^j}$ for $j=0,1$, 
then 
\begin{equation}
p_0(x,D)p_1(x,D)
\equiv
p_1(x,D)p_0(x,D)
\equiv
q(x,D),
\label{equation:product} 
\end{equation}
\begin{equation}
p_1(x,D)^\ast
\equiv
r(x,D) 
\label{equation:adjoint}
\end{equation}
modulo $L^2$-bounded operators, 
where 
$q(x,\xi)=p_0(x,\xi)p_1(x,\xi)$ 
and 
$r(x,\xi)=\bar{p}_1(x,\xi)$. 
More precisely, there exist a positive integer $\nu$ and $C>0$ such that 
for any $u{\in}L^2$ 
\begin{align*}
  \lVert(p_0(x,D)p_1(x,D)-q(x,D))u\rVert
& \leqslant
  C
  \lVert{p_0}\rVert_{\mathscr{B}^\sigma{S^0},\nu}
  \lVert{p_1}\rVert_{\mathscr{B}^\sigma{S^1},\nu}
  \lVert{u}\rVert,
\\
  \lVert(p_1(x,D)p_0(x,D)-q(x,D))u\rVert
& \leqslant
  C
  \lVert{p_0}\rVert_{\mathscr{B}^\sigma{S^0},\nu}
  \lVert{p_1}\rVert_{\mathscr{B}^\sigma{S^1},\nu}
  \lVert{u}\rVert,
\\
  \lVert(p_1(x,D)^\ast-r(x,D))u\rVert
& \leqslant
  C
  \lVert{p_1}\rVert_{\mathscr{B}^\sigma{S^1},\nu}
  \lVert{u}\rVert.
\end{align*}
\end{lemma}
\begin{lemma}[Chihara {\cite[Lemma~3]{chihara}}]
\label{theorem:garding} 
Suppose that 
$p(x,\xi)=[p_{ij}(x,\xi)]_{i,j=1,\dotsc,l}$ 
is an $l{\times}l$ matrix whose entries belong to 
$\mathscr{B}^2S^1$, and that 
$$
p(x,\xi)+p(x,\xi)^\ast\geqslant0
$$
for $\lvert\xi\rvert\geqslant{R}$ with some $R>0$. 
Then there exists $C_1>0$ which is independent of $l$, such that 
for any $u\in(\mathscr{S})^l$ 
$$
\re(p(x,D)u,u)
\geqslant
-C_1A(p)\lVert{u}\rVert^2, 
$$
where 
$$
A(p)=\sup_{\substack{X\in\mathbb{C}^l
                     \\
                     \lvert{X}\rvert=1}}
     \lvert{}^tXPX\rvert,
\quad
P
=
[\lVert{p_{ij}}\rVert_{\mathscr{B}^2S^1,\nu}]_{i,j=1,\dotsc,l},
$$
and $\nu$ is some positive integer. 
\end{lemma}
Loosely speaking, 
Theorem~\ref{theorem:nagase}, 
Lemma~\ref{theorem:algebra} 
and 
Lemma~\ref{theorem:garding} 
allow us to deal with 
$p(x,\xi)\in\mathscr{B}^2S^m$ ($m=0,1$) 
as if it belonged to $S^m$. 
Now, let us consider commutator estimates of 
pseudodifferential operators with constant coefficients. 
First we recall the Kato-Ponce commutator estimates. 
\begin{theorem}[Kato and Ponce {\cite[Lemma~X1]{kp}}]
\label{theorem:kato-ponce} 
If $\theta>0$, then for any $f,g\in\mathscr{S}$ 
\begin{equation}
\lVert
\langle{D}\rangle^\theta(fg)
-
f\langle{D}\rangle^\theta{g}
\rVert
\leqslant
C
(
\lVert\p{f}\rVert_{L^\infty}
\lVert{g}\rVert_{\theta-1}
+
\lVert{f}\rVert_\theta
\lVert{g}\rVert_{L^\infty}
).
\label{equation:kato-ponce}
\end{equation}
\end{theorem}  
Here we remark that Kato and Ponce actually proved 
$L^p$-version of \eqref{equation:kato-ponce}. 
Next we give the Leibniz formula for Fourier multipliers. 
\begin{lemma}
\label{theorem:leibniz}
Let $k=2,3,4,\dotsc$, and let $m\geqslant1$ and $\theta>n/2+1$. 
If $p(\xi)\in{S^m}$, then there exists a positive constant 
$C_{m,\theta}$ which is independent of $k$, such that 
for any $f_1,\dotsc,f_k\in\mathscr{S}$ 
\begin{equation}
\lVert{p(D)}(\prod_{j=1}^kf_j)\rVert
\leqslant
C_{m,\theta}^k
\sum_{\nu=1}^k
\lVert{f_\nu}\rVert_{m}
\prod_{\substack{j=1\\ j\ne\nu}}^k
\lVert{f_j}\rVert_{\theta-1},
\label{equation:chain} 
\end{equation}
\begin{equation}
\lVert
p(D)(\prod_{j=1}^kf_j)
-
\sum_{\nu=1}^k
\prod_{\substack{j=1\\ j\ne\nu}}^kf_j
p(D)f_\nu
\rVert
\leqslant
C_{m,\theta}^k
\sum_{\nu=1}^k
\lVert{f_\nu}\rVert_{m-1}
\prod_{\substack{j=1\\ j\ne\nu}}^k
\lVert{f_j}\rVert_\theta.
\label{equation:leibniz}
\end{equation}
\end{lemma}
\begin{proof}
First we show \eqref{equation:chain}. 
We denote the Fourier transform of $f$ by $\hat{f}$ or $\mathscr{F}[f]$, 
and the convolution of functions on $\mathbb{R}^n$ by $*$ respectively. 
Using the Plancherel-Parseval formula and the Sobolev embedding, 
we deduce 
\begin{align*}
& \lVert{p(D)}(\prod_{j=1}^kf_j)\rVert
\\
& \quad
  =
  \left(
  \int_{\mathbb{R}^n}
  \lvert
  p(\xi)
  \hat{f}_1*\dotsb*\hat{f}_k(\xi)
  \rvert^2
  d\xi
  \right)^{1/2}
\\
& \quad
  \leqslant
  C
  \sum_{\nu=1}^k
  \left(
  \int_{\mathbb{R}^n}
  \lvert
  \hat{f}_1*\dotsb*\hat{f}_{\nu-1}*
  \mathscr{F}[\langle{D}\rangle^mf_\nu]
  *\hat{f}_{\nu+1}*\dotsb*\hat{f}_k(\xi)
  \rvert^2
  d\xi
  \right)^{1/2}
\\
& \quad
  =
  C
  \sum_{\nu=1}^k
  \lVert
  \prod_{\substack{j=1\\ j\ne\nu}}^kf_j
  \langle{D}\rangle^mf_\nu
  \rVert 
\\
& \quad
  \leqslant
  C
  \sum_{\nu=1}^k
  \prod_{\substack{j=1\\ j\ne\nu}}^k
  \lVert{f_j}\rVert_{L^\infty}
  \lVert{f_\nu}\rVert_{m}
\\
& \quad
  \leqslant
  CC_0^{k-1}
  \sum_{\nu=1}^k
  \prod_{\substack{j=1\\ j\ne\nu}}^k
  \lVert{f_j}\rVert_{\theta-1}
  \lVert{f_\nu}\rVert_{m}
\\
& \quad
  \leqslant
  C_1^k
  \sum_{\nu=1}^k
  \prod_{\substack{j=1\\ j\ne\nu}}^k
  \lVert{f_j}\rVert_{\theta-1}
  \lVert{f_\nu}\rVert_{m},
\end{align*}
where $C_1=\max\{C,C_0\}$. 
\par
Next we show \eqref{equation:leibniz}. 
Set $\sigma(\xi,\eta)=p(\xi+\eta)-p(\xi)-p(\eta)$ for short. 
Here we claim 
\begin{equation}
\lvert\sigma(\xi,\eta)\rvert
\leqslant
C\langle\xi\rangle^{m-1}\langle\eta\rangle
\quad\text{for}\quad
\lvert\xi\rvert\geqslant\lvert\eta\rvert.
\label{equation:sigma} 
\end{equation}
Indeed, the mean value theorem implies 
$$
\sigma(\xi,\eta)
=
\sum_{j=1}^n
\eta_j
\int_0^1
\frac{\p{p}}{\p\xi_j}(\xi+\rho\eta)
d\rho
-
p(\eta). 
$$
Then we have 
$$
\lvert\sigma(\xi,\eta)\rvert
\leqslant
C
\lvert\eta\rvert
\int_0^1
\langle\xi+\rho\eta\rangle^{m-1}
d\rho
+
C
\langle\eta\rangle^m.
$$
Since $\lvert\xi\rvert\geqslant\lvert\eta\rvert$ and $m-1\geqslant0$, 
we get 
$$
\lvert\sigma(\xi,\eta)\rvert
\leqslant
C\langle\xi\rangle^{m-1}\langle\eta\rangle
+
C\langle\eta\rangle^m,
$$
which is \eqref{equation:sigma}. 
\par
Now we show \eqref{equation:leibniz} for $k=2$. 
The Plancherel-Parseval formula gives 
\begin{align*}
& \lVert{p(D)(fg)-gp(D)f-fp(D)g}\rVert
\\
& \quad
  =
  \left(
  \int_{\mathbb{R}^n}
  \left\lvert
  \int_{\mathbb{R}^n}
  \sigma(\xi-\eta,\eta)
  \hat{f}(\xi-\eta)\hat{g}(\eta)
  d\eta
  \right\rvert^2
  d\xi
  \right)^{1/2}. 
\end{align*}
We split the above integration in $\eta$ into two pieces:
$$
\int_{\mathbb{R}^n}\dotsb d\eta
=
\int_{\lvert\xi-\eta\rvert\geqslant\lvert\eta\rvert}\dotsb d\eta
+
\int_{\lvert\xi-\eta\rvert<\lvert\eta\rvert}\dotsb d\eta.
$$
Then we have 
$$
\lVert{p(D)(fg)-gp(D)f-fp(D)g}\rVert
\leqslant
\text{I}+\text{II}, 
$$
\begin{align*}
  \text{I}
& =
  \left(
  \int_{\mathbb{R}^n}
  \left\lvert
  \int_{\lvert\xi-\eta\rvert\geqslant\lvert\eta\rvert}
  \sigma(\xi-\eta,\eta)
  \hat{f}(\xi-\eta)\hat{g}(\eta)
  d\eta
  \right\rvert^2
  d\xi
  \right)^{1/2},
\\
  \text{II}
& =
  \left(
  \int_{\mathbb{R}^n}
  \left\lvert
  \int_{\lvert\xi-\eta\rvert<\lvert\eta\rvert}
  \sigma(\xi-\eta,\eta)
  \hat{f}(\xi-\eta)\hat{g}(\eta)
  d\eta
  \right\rvert^2
  d\xi
  \right)^{1/2}.
\end{align*}
On one hand, 
applying \eqref{equation:sigma}, 
the Young and the Schwarz inequalities 
in order of precedence, 
we deduce 
\begin{align}
  \text{I}
& \leqslant
  C
  \left(
  \int_{\mathbb{R}^n}
  \left\lvert
  \int_{\mathbb{R}^n}
  \langle\xi-\eta\rangle^{m-1}\lvert\hat{f}(\xi-\eta)\rvert
  \langle\eta\rangle\lvert\hat{g}(\eta)\rvert
  d\eta
  \right\rvert^2
  d\xi
  \right)^{1/2}
\nonumber
\\
& \leqslant
  C
  \lVert{f}\rVert_{m-1}
  \int_{\mathbb{R}^n}
  \langle\eta\rangle\lvert\hat{g}(\eta)\rvert
  d\eta
\nonumber
\\ 
& =
  C
  \lVert{f}\rVert_{m-1}
  \int_{\mathbb{R}^n}
  \langle\eta\rangle^{-(\theta-1)}
   \lvert\langle\eta\rangle^\theta\hat{g}(\eta)\rvert
   d\eta
\nonumber
\\ 
& \leqslant
  C
  \left(
  \int_{\mathbb{R}^n}
  \langle\eta\rangle^{-2(\theta-1)}
  d\eta
  \right)^{1/2}
  \lVert{f}\rVert_{m-1}\lVert{g}\rVert_\theta
\nonumber
\\ 
& =
  C^\prime
  \lVert{f}\rVert_{m-1}\lVert{g}\rVert_\theta.
\label{equation:half1}  
\end{align}
Here we used $\theta-1>n/2$.  
On the other hand, 
using \eqref{equation:sigma} again, and 
changing variable by $\eta\mapsto\zeta=\xi-\eta$, 
we have 
$$
\text{II}
\leqslant
C
\left(
\int_{\mathbb{R}^n}
\left\lvert
\int_{\mathbb{R}^n}
\langle\zeta\rangle\lvert\hat{f}(\zeta)\rvert
\langle\xi-\zeta\rangle\lvert\hat{g}(\xi-\zeta)\rvert
d\zeta
\right\rvert^2
d\xi
\right)^{1/2}, 
$$
which is reduced to $\text{I}$. 
Then we get \eqref{equation:leibniz} for $k=2$:
\begin{equation}
\lVert{p(D)(fg)-gp(D)f-fp(D)g}\rVert
\leqslant
C_2
(
\lVert{f}\rVert_{m-1}\lVert{g}\rVert_\theta
+
\lVert{f}\rVert_\theta\lVert{g}\rVert_{m-1}
).
\label{equation:leibniz2}
\end{equation}
\par
Lastly, we prove \eqref{equation:leibniz} for $k\geqslant3$. 
Set $\prod_{j=1}^0=1$. 
Applying \eqref{equation:chain} and \eqref{equation:leibniz2} 
to the identity 
\begin{align*}
& p(D)(\prod_{j=1}^kf_j)
  -
  \sum_{\nu=1}^k
  \prod_{\substack{j=1\\ j\ne\nu}}^kf_j
  p(D)f_\nu
\\
& \quad
  =
  \sum_{\nu=1}^{k-1}
  \prod_{l=1}^{\nu-1}f_l
  \{
  p(D)(\prod_{j=\nu}^kf_j)
  -
  f_{\nu}p(D)(\prod_{j=\nu+1}^kf_j)
  -
  \prod_{j=\nu+1}^kf_jp(D)f_\nu
  \}, 
\end{align*}
we deduce 
\begin{align*}
& \lVert
  p(D)(\prod_{j=1}^kf_j)
  -
  \sum_{\nu=1}^k
  \prod_{\substack{j=1\\ j\ne\nu}}f_j
  p(D)f_\nu
  \rVert
\\
& \quad
  \leqslant
  \sum_{\nu=1}^{k-1}
  \prod_{l=1}^{\nu-1}\lVert{f_l}\rVert_{L^\infty}
  \lVert
  p(D)(\prod_{j=\nu}^kf_j)
  -
  f_{\nu}p(D)(\prod_{j=\nu+1}^kf_j)
  -
  \prod_{j=\nu+1}^kf_jp(D)f_\nu
  \rVert
\\
& \quad
  \leqslant
  C_2
  \sum_{\nu=1}^{k-1}
  \prod_{l=1}^{\nu-1}\lVert{f_l}\rVert_{L^\infty}
  (
  \lVert{f_\nu}\rVert_{m-1}
  \lVert\prod_{j=\nu+1}^kf_j\rVert_\theta
  +
  \lVert{f_\nu}\rVert_\theta
  \lVert\prod_{j=\nu+1}^kf_j\rVert_{m-1}
  )
\\
& \quad
  \leqslant
  C_2
  \sum_{\nu=1}^{k-1}C_3^{\nu-1}
  \prod_{l=1}^{\nu-1}\lVert{f_l}\rVert_\theta
  (
  \lVert{f_\nu}\rVert_{m-1}
  \lVert\prod_{j=\nu+1}^kf_j\rVert_\theta
  +
  \lVert{f_\nu}\rVert_\theta
  \lVert\prod_{j=\nu+1}^kf_j\rVert_{m-1}
  )
\\
& \quad
  \leqslant
  C_2
  \sum_{\nu=1}^{k-1}C_3^{\nu-1}
  \prod_{l=1}^{\nu-1}\lVert{f_l}\rVert_\theta
\\
& \qquad
  \times
  \{
  (k-\nu)C_3^{k-\nu}
  \lVert{f_\nu}\rVert_{m-1}
  \prod_{j=\nu+1}^k\lVert{f_j}\rVert_\theta
  +
  C_4^{k-\nu-1}
  \sum_{p=\nu+1}^k
  \prod_{\substack{j=\nu\\ j\ne{p}}}^k
  \lVert{f_j}\rVert_\theta
  \lVert{f_p}\rVert_{m-1}
  \}
\\
& \quad
  \leqslant
  C_5^kk
  \sum_{\nu=1}^k
  \lVert{f_\nu}\rVert_{m-1}
  \prod_{\substack{j=1\\ j\ne\nu}}^k
  \lVert{f_j}\rVert_\theta 
\\
& \quad
  \leqslant
  2^kC_5^k
  \sum_{\nu=1}^k
  \lVert{f_\nu}\rVert_{m-1}
  \prod_{\substack{j=1\\ j\ne\nu}}^k
  \lVert{f_j}\rVert_\theta, 
\end{align*}
where $C_5=\max\{C_2,C_3,C_4\}$. 
This completes the proof.
\end{proof}

\section{Summation properties}
\label{section:preliminaries}
This section consists of miscellaneous lemmas needed later. 
In particular, we obtain some fine summmation properties 
related with Gevrey estimates.   
We start by giving the properties of exponentially decaying functions. 
\begin{lemma}
\label{theorem:decay}
Let $s>0$, $\ep>0$ and $\theta\in\mathbb{R}$. 
If $\exp(\ep\langle{x}\rangle^{1/s})u_0{\in}H^\theta$, 
then there exists $q=q(n,\theta,\ep,s)>0$ such that 
for any multi-index $\alpha$ 
\begin{equation}
\lVert{x^\alpha}u_0\rVert_\theta
\leqslant
\lVert\exp(\ep\langle{x}\rangle^{1/s})u_0\rVert_\theta
q^{\lvert\alpha\rvert+1}\alpha!^s. 
\label{equation:decay}
\end{equation}
\end{lemma}
\begin{proof}
By the $L^2$-boundedness theorem 
for pseudodifferential operators of order zero, 
we deduce 
\begin{align}
  \lVert{x^\alpha}u_0\rVert_\theta
& =
  \lVert\langle{D}\rangle^\theta{x^\alpha}u_0\rVert
\nonumber
\\
& =
  \lVert
  \langle{D}\rangle^\theta
  (x^\alpha{e}^{-\ep\langle{x}\rangle^{1/s}})
  \langle{D}\rangle^{-\theta}
  \langle{D}\rangle^\theta
  e^{\ep\langle{x}\rangle^{1/s}}
  u_0
  \rVert
\nonumber
\\
& \leqslant
  C
  \lVert{x^\alpha{e}^{-\ep\langle{x}\rangle^{1/s}}}\rVert_{\mathscr{B}^\nu}
  \lVert{e^{\ep\langle{x}\rangle^{1/s}}u_0}\rVert_\theta, 
\label{equation:301}
\end{align}
where $\nu$ is a positive integer satisfying $\nu>\lvert\theta\rvert$. 
Set $\rho=\max\{0,1/s-1\}$.  
Using the Leibniz formula for $\lvert\beta\rvert\leqslant\nu$, we have 
\begin{align*}
& \lvert\p^\beta(x^\alpha{e}^{-\ep\langle{x}\rangle^{1/s}})\rvert
\\
& \quad
  =
  \left\lvert
  \sum_{\gamma\leqslant\beta}
  \frac{\beta!}{\gamma!(\beta-\gamma)!}
  \p^\gamma{x^\alpha}
  \p^{\beta-\gamma}
  e^{-\ep\langle{x}\rangle^{1/s}}
  \right\rvert
\\
& \quad
  \leqslant
  \sum_{\gamma\leqslant\beta,\alpha}
  \frac{\beta!}{\gamma!(\beta-\gamma)!}
  \frac{\alpha!}{\gamma!(\alpha-\gamma)!}
  \gamma!\lvert{x}^{\alpha-\gamma}\rvert
  \lvert{\p^{\beta-\gamma}e^{-\ep\langle{x}\rangle^{1/s}}}\rvert
\\
& \quad
  \leqslant
  C_\nu
  \sum_{\gamma\leqslant\beta,\alpha}
  \frac{\beta!}{\gamma!(\beta-\gamma)!}
  \frac{\alpha!}{\gamma!(\alpha-\gamma)!}
  \gamma!
  \langle{x}\rangle^{\lvert\alpha-\gamma\rvert
                     +(1/s-1)\lvert\beta-\gamma\rvert}
  e^{-\ep\langle{x}\rangle^{1/s}}
\\
& \quad
  \leqslant
  C_\nu2^{\lvert\alpha\rvert+\nu}\nu!
  \langle{x}\rangle^{\lvert\alpha\rvert+\rho\nu}
  e^{-\ep\langle{x}\rangle^{1/s}}
\\
& \quad
  \leqslant
  C_\nu2^{(1+\rho)\nu}\nu!4^{\lvert\alpha\rvert}
  \sup_{\tau>0}\tau^{\rho\nu}
  e^{-\ep\tau^{1/s}}
  \sup_{t>0}
  t^{\lvert\alpha\rvert}
  e^{-\ep{t}^{1/s}}
\\
& \quad
  =
  C_{\nu,s}^\prime
  4^{\lvert\alpha\rvert}
  \sup_{t>0}
  t^{\lvert\alpha\rvert}
  e^{-\ep{t}^{1/s}}
\\
& \quad
  =
  C_{\nu,s}^\prime
  4^{\lvert\alpha\rvert}
  t^{\lvert\alpha\rvert}
  e^{-\ep{t}^{1/s}}
  \bigg\vert_{t=(s\lvert\alpha\rvert/\ep)^s}
\\
& \quad
  =
  C_{\nu,s}^\prime
  (4s^se^{-s})^{\lvert\alpha\rvert}
  ({\lvert\alpha\rvert}^{\lvert\alpha\rvert}e^{-\lvert\alpha\rvert})^s
\\
& \quad
  \leqslant
  C_{\nu,s}^\prime
  (4s^se^{-s})^{\lvert\alpha\rvert}
  \lvert\alpha\rvert!^s.
\end{align*}  
Then there exists $q>0$ which is independent $\alpha$, such that 
\begin{equation}
\lVert{x^\alpha{e}^{-\ep\langle{x}\rangle^{1/s}}}\rVert_{\mathscr{B}^\nu}
\leqslant
q^{\lvert\alpha\rvert+1}\alpha!^s.
\label{equation:302}
\end{equation} 
The substitution of \eqref{equation:302} into \eqref{equation:301} 
gives \eqref{equation:decay}. 
\end{proof}
\par
Next we present a lemma concerned with factorials. 
\begin{lemma}
\label{theorem:factorial}
For any multi-indices $\alpha$, $\alpha^1,\dotsc,\alpha^p$ 
satisfying $\alpha=\alpha^1+\dotsc+\alpha^p$, 
\begin{equation}
\frac{\lvert\alpha^1\rvert!\dotsb\lvert\alpha^p\rvert!}
     {\alpha^1!\dotsb\alpha^p!}
\leqslant
\frac{\lvert\alpha\rvert!}{\alpha!}. 
\label{equation:factorial}
\end{equation}
\end{lemma}
\begin{proof}
Let $n$ be the dimension of $\alpha$. 
Since 
$$
(x_1+\dotsb+x_n)^{\lvert\alpha\rvert}
=
\prod_{j=1}^p(x_1+\dotsb+x_n)^{\lvert\alpha^j\rvert}, 
\quad
x=(x_1,\dotsc,x_n)\in\mathbb{R}^n,
$$ 
the multinomial theorem gives 
$$
\sum_{\lvert\gamma\rvert=\lvert\alpha\rvert}
\frac{\lvert\alpha\rvert!}{\gamma!}x^\gamma
=
\sum_{\lvert\gamma^1\rvert=\lvert\alpha^1\rvert}
\dotsb
\sum_{\lvert\gamma^p\rvert=\lvert\alpha^p\rvert}
\frac{\lvert\alpha^1\rvert!\dotsb\lvert\alpha^p\rvert!}
     {\gamma^1!\dotsb\gamma^p!}
     x^{\gamma^1+\dotsb+\gamma^p}.
$$
Operating $\p^\alpha/\alpha!$ on the both sides of the above identity, 
we have 
$$
\frac{\lvert\alpha\rvert!}{\alpha!}
=
\sum_{\substack{\gamma^1+\dotsb+\gamma^p=\alpha
                \\
                \lvert\gamma^1\rvert=\lvert\alpha^1\rvert
                \\
                \dotsb
                \\
                \lvert\gamma^p\rvert=\lvert\alpha^p\rvert}}
\frac{\lvert\alpha^1\rvert!\dotsb\lvert\alpha^p\rvert!}
     {\gamma^1!\dotsb\gamma^p!},
$$
which implies \eqref{equation:factorial}. 
\end{proof}
\par
Now we present a lemma concerned with the nonlinearity and multi-indices. 
This plays a crucial role in the estimate of nonlinearity. 
Let $\alpha=(\alpha_1,\dotsc,\alpha_n)$ be a multi-index. 
Set 
$$
\alpha_\ast
=
(\max\{0,\alpha_1-1\},\dotsb,\max\{0,\alpha_n-1\}),
$$
$$
\alpha^\dagger
=
\prod_{j=1}^n\alpha_j^\dagger, 
\quad
\alpha_j^\dagger
=
\max\{1,\alpha_j\},
\quad
\alpha_\ast^\dagger=(\alpha_\ast)^\dagger.
$$ 
\begin{lemma}
\label{theorem:summation1}
Let $l$, $p$ and $q$ be integers satisfying 
$l\geqslant0$ and $p,q\geqslant2$ respectively. 
Set 
$$
N
=
\sum_{k=0}^l
\frac{(k+n-1)!}{k!(n-1)!}, 
$$
which is the number of $n$-dimensional multi-indices satisfying 
$\lvert\alpha\rvert\leqslant{l}$. 
For any vector 
$(X(\alpha))_{\lvert\alpha\rvert\leqslant{l}}\in[0,\infty)^N$, 
\begin{align}
& \sum_{\lvert\alpha\rvert\leqslant{l}}
  \sum_{\substack{\alpha(1)+\dotsb+\alpha(p)=\alpha
                  \\
                  \beta(1)+\dotsb+\beta(q)=\alpha}}
  \dfrac{(\alpha_\ast^\dagger)^2}
        {\prod_{j=1}^p\alpha(j)_\ast^\dagger
         \prod_{k=1}^q\beta(k)_\ast^\dagger}
  \prod_{j=1}^pX(\alpha(j))
  \prod_{k=1}^qX(\beta(k))
\nonumber
\\
& \quad
  \leqslant
  a^{n(p+q-2)}p^{2n}q^{2n}
  \left(
  \sum_{\lvert\alpha\rvert\leqslant{l}}
  X(\alpha)^2
  \right)^{(p+q)/2}, 
\label{equation:summation1}
\end{align}
where
$$
a
=
\left(
1+\sum_{j=1}^\infty\frac{1}{j^2}
\right)^{1/2}.
$$
\end{lemma}
\begin{proof}
When $n=1$, there exist $j$ and $k$ such that 
$
\alpha_\ast^\dagger
\leqslant
p\alpha(j)_\ast^\dagger,
q\beta(k)_\ast^\dagger
$. 
Without loss of generality, we can assume $j=p$ and $k=q$. 
Using the Schwarz inequality for the finite sum again and again, we have 
\begin{align}
& \sum_{\lvert\alpha\rvert\leqslant{l}}
  \sum_{\substack{\alpha(1)+\dotsb+\alpha(p)=\alpha
                  \\
                  \beta(1)+\dotsb+\beta(q)=\alpha}}
  \dfrac{(\alpha_\ast^\dagger)^2}
        {\prod_{j=1}^p\alpha(j)_\ast^\dagger
         \prod_{k=1}^q\beta(k)_\ast^\dagger}
  \prod_{j=1}^pX(\alpha(j))
  \prod_{k=1}^qX(\beta(k))
\nonumber
\\
& \quad
  \leqslant
  p^2q^2
  \sum_{\substack{\alpha(1)+\dotsb+\alpha(p)\leqslant{l}
                  \\
                  \beta(1)+\dotsb+\beta(q)
                  =
                  \alpha(1)+\dotsb+\alpha(p)}}
  \prod_{j=1}^{p-1}\frac{X(\alpha(j))}{\alpha(j)_\ast^\dagger}
  \prod_{k=1}^{q-1}\frac{X(\beta(k))}{\beta(k)_\ast^\dagger}
  X(\alpha(p))X(\beta(q))
\nonumber
\\
& \quad
  =
  p^2q^2
  \sum_{\alpha(1)=0}^l
  \frac{X(\alpha(1))}{\alpha(1)_\ast^\dagger}
  \dotsb
  \sum_{\alpha(p-1)=0}^{l-\alpha(1)-\dotsb-\alpha(p-2)}
  \frac{X(\alpha(p-1))}{\alpha(p-1)_\ast^\dagger}
\nonumber
\\
& \quad
  \times
  \sum_{\alpha(p)=0}^{l-\alpha(1)-\dotsb-\alpha(p-1)}
  X(\alpha(p))
\nonumber
\\
& \quad
  \times
  \sum_{\beta(1)=0}^{\alpha(1)+\dotsb+\alpha(p)}
  \frac{X(\beta(1))}{\beta(1)_\ast^\dagger}
  \dotsb
  \sum_{\beta(q-1)=0}^{\alpha(1)+\dotsb+\alpha(p)-\beta(1)-\dotsb-\beta(q-2)}
  \frac{X(\beta(q-1))}{\beta(q-1)_\ast^\dagger}
\nonumber
\\
& \quad
  \times
  X(\alpha(1)+\dotsb+\alpha(p)-\beta(1)-\dotsb-\beta(q-1))
\nonumber
\\
& \quad
  \leqslant
  p^2q^2
  \sum_{\alpha(1)=0}^l
  \frac{X(\alpha(1))}{\alpha(1)_\ast^\dagger}
  \dotsb
  \sum_{\alpha(p-1)=0}^{l-\alpha(1)-\dotsb-\alpha(p-2)}
  \frac{X(\alpha(p-1))}{\alpha(p-1)_\ast^\dagger}
\nonumber
\\
& \quad
  \times
  \sum_{\beta(1)=0}^{l}
  \frac{X(\beta(1))}{\beta(1)_\ast^\dagger}
  \dotsb
  \sum_{\beta(q-1)=0}^{l-\beta(1)-\dotsb-\beta(q-2)}
  \frac{X(\beta(q-1))}{\beta(q-1)_\ast^\dagger}
\nonumber
\\
& \quad
  \times
  \sum_{\alpha(p)=\gamma\geqslant0}^l
  X(\alpha(p))
  X(\alpha(1)+\dotsb+\alpha(p)-\beta(1)-\dotsb-\beta(q-1))
\nonumber
\\
& \quad
  \leqslant
  p^2q^2
  \sum_{\alpha(1)=0}^l
  \frac{X(\alpha(1))}{\alpha(1)_\ast^\dagger}
  \dotsb
  \sum_{\alpha(p-1)=0}^{l-\alpha(1)-\dotsb-\alpha(p-2)}
  \frac{X(\alpha(p-1))}{\alpha(p-1)_\ast^\dagger}
\nonumber
\\
& \quad
  \times
  \sum_{\beta(1)=0}^{l}
  \frac{X(\beta(1))}{\beta(1)_\ast^\dagger}
  \dotsb
  \sum_{\beta(q-1)=0}^{l-\beta(1)-\dotsb-\beta(q-2)}
  \frac{X(\beta(q-1))}{\beta(q-1)_\ast^\dagger}
\nonumber
\\
& \quad
  \times
  \left(
  \sum_{\alpha(p)=\gamma\geqslant0}^l
  X(\alpha(p))^2
  \right)^{1/2}
\nonumber
\\
& \quad
  \times
  \left(
  \sum_{\alpha(p)=\gamma\geqslant0}^l
  X(\alpha(1)+\dotsb+\alpha(p)-\beta(1)-\dotsb-\beta(q-1))^2
  \right)^{1/2}
\nonumber
\\
& \quad
  \leqslant
  p^2q^2
  \left(
  \sum_{\alpha=0}^l
  \frac{X(\alpha)}{\alpha_\ast^\dagger}
  \right)^{p+q-2}
  \sum_{\alpha=0}^l
  X(\alpha)^2
\nonumber
\\
& \quad
  \leqslant
  p^2q^2a^{p+q-2}
  \left(
  \sum_{\alpha=0}^l
  X(\alpha)^2
  \right)^{(p+q)/2},
\label{equation:303}
\end{align}
where we denote  
$\gamma=\beta(1)+\dotsb+\beta(q-1)-\alpha(1)-\dotsb-\alpha(p-1)$. 
\par
When $n=2$, using \eqref{equation:303} twice, we deduce 
\begin{align*}
& \sum_{\lvert\alpha\rvert\leqslant{l}}
  \sum_{\substack{\alpha(1)+\dotsb+\alpha(p)=\alpha
                  \\
                  \beta(1)+\dotsb+\beta(q)=\alpha}}
  \dfrac{(\alpha_\ast^\dagger)^2}
        {\prod_{j=1}^p\alpha(j)_\ast^\dagger
         \prod_{k=1}^q\beta(k)_\ast^\dagger}
  \prod_{j=1}^pX(\alpha(j))
  \prod_{k=1}^qX(\beta(k))
\\
& \quad
  \leqslant
  a^{p+q-2}p^2q^2
  \sum_{\alpha_2=0}^l
  \sum_{\substack{\alpha_2(1)+\dotsb+\alpha_2(p)=\alpha_2
                  \\
                  \beta_2(1)+\dotsb+\beta_2(q)=\alpha_2}}
  \dfrac{\{(\alpha_2)_\ast^\dagger\}^2}
        {\prod_{j=1}^p\alpha_2(j)_\ast^\dagger
         \prod_{k=1}^q\beta_2(k)_\ast^\dagger}
\\
& \qquad
  \times
  \prod_{j=1}^p
  \left(
  \sum_{\alpha_1(j)\leqslant{l}-\alpha_2}
  X(\alpha_1(j),\alpha_2(j))^2
  \right)^{1/2}
\\
& \qquad\quad
  \times
  \prod_{k=1}^q
  \left(
  \sum_{\beta_1(k)\leqslant{l}-\alpha_2}
  X(\beta_1(k),\beta_2(k))^2
  \right)^{1/2}
\\
& \leqslant
  a^{2(p+q-2)}p^4q^4
  \left(
  \sum_{\lvert\alpha\rvert\leqslant{l}}
  X(\alpha)^2
  \right)^{(p+q)/2}.
\end{align*}
In the same way, we can obtain \eqref{equation:summation1} 
for any $n\geqslant3$. We omit the detail. 
\end{proof}
\par
We conclude this section by giving an estimate of matrices 
used for some linear systems later. 
\begin{lemma}
\label{theorem:summation2}
Let $l$ be a positive integer, 
and let $N$ be the same integer as in Lemma~\ref{theorem:summation1}. 
Suppose that 
$B=[b_{\alpha,\beta}]_{\lvert\alpha\rvert\leqslant{l},\beta\leqslant\alpha}$ 
is an $N{\times}N$ lower triangular matrix 
whose entries are suffixed by multi-indices. 
Note that $b_{\alpha,\beta}=0$ unless $\beta\leqslant\alpha$. 
For any 
$X=[X(\alpha)]_{\lvert\alpha\rvert\leqslant{l}}\in\mathbb{C}^N$, 
\begin{align}
& \lvert{}^tXB\bar{X}\rvert
\nonumber
\\
& \quad
  \leqslant
  2^nn!
  \lvert{X}\rvert^2
  \max_{\substack{\sigma\in{S_n}
        \\
        \nu=0,1,\dotsc,n
        \\
        {\sigma(1)}<\dotsb<{\sigma(\nu)}
        \\
        {\sigma(\nu+1)}<\dotsb<{\sigma(n)}}}
\nonumber
\\
& \qquad
  \times
  \left(
  \sum_{\substack{\alpha_{\sigma(1)}+\dotsb+\alpha_{\sigma(\nu)}
                  \leqslant
                  l-\alpha_{\sigma(\nu+1)}-\dotsb-\alpha_{\sigma(n)}
                  \\
                  \beta_{\sigma(1)}
                  =
                  \alpha_{\sigma(1)},\dotsc,2\alpha_{\sigma(1)}-1
                  \\
                  \dotsb
                  \\
                  \beta_{\sigma(\nu)}
                  =
                  \alpha_{\sigma(\nu)},\dotsc,2\alpha_{\sigma(\nu)}-1}}
  \max_{\substack{\beta_{\sigma(\nu+1)}
                  =
                  2\alpha_\sigma(\nu+1),\dotsc
                  \\
                  \dotsb
                  \\
                  \beta_{\sigma(n)}
                  =
                  2\alpha_{\sigma(n)},\dotsc}}
  \lvert{b_{\beta,\beta-\alpha}}\rvert^2
  \right)^{1/2}, 
\label{equation:summation2}
\end{align}
where $S_n$ is the $n$-dimensional symmetric group. 
\end{lemma}  
\begin{proof}
We split ${}^tXB\bar{X}$ into several pieces according 
to the index of the entries of $B$: 
\begin{align}
& \lvert{}^tXB\bar{X}\rvert
\nonumber
\\
& \quad
  =
  \left\lvert
  \sum_{\lvert\alpha\rvert\leqslant{l}}
  \sum_{\beta\leqslant\alpha}
  b_{\alpha,\beta}
  X(\alpha)\bar{X}(\beta)
  \right\rvert
\nonumber
\\
& \quad
  =
  \left\lvert
  \sum_{\lvert\alpha\rvert\leqslant{l}}
  \sum_{\substack{\sigma\in{S_n}
                  \\
                  \nu=0,1,\dotsc,n
                  \\
                  {\sigma(1)}<\dotsb<{\sigma(\nu)}
                  \\
                  {\sigma(\nu+1)}<\dotsb<{\sigma(n)}}}
  \sum_{\substack{\beta_{\sigma(1)}<\alpha_{\sigma(1)}
                  \\
                  \dotsb
                  \\
                  \beta_{\sigma(\nu)}<\alpha_{\sigma(\nu)}
                  \\
                  \beta_{\sigma(\nu+1)}\geqslant\alpha_{\sigma(\nu+1)}/2
                  \\
                  \dotsb
                  \\
                  \beta_{\sigma(n)}\geqslant\alpha_{\sigma(n)}/2}}
  b_{\alpha,\beta}X(\alpha)\bar{X}(\beta)  
  \right\rvert
\nonumber
\\
& \quad
  \leqslant
  2^nn!
  \max_{\substack{\sigma\in{S_n}
                  \\
                  \nu=0,1,\dotsc,n
                  \\
                  {\sigma(1)}<\dotsb<{\sigma(\nu)}
                  \\
                  {\sigma(\nu+1)}<\dotsb<{\sigma(n)}}}
  \sum_{\lvert\alpha\rvert\leqslant{l}}
  \sum_{\substack{\beta_{\sigma(1)}<\alpha_{\sigma(1)}
                  \\
                  \dotsb
                  \\
                  \beta_{\sigma(\nu)}<\alpha_{\sigma(\nu)}
                  \\
                  \beta_{\sigma(\nu+1)}\geqslant\alpha_{\sigma(\nu+1)}/2
                  \\
                  \dotsb
                  \\
                  \beta_{\sigma(n)}\geqslant\alpha_{\sigma(n)}/2}}
  \lvert{b_{\alpha,\beta}}\rvert
  \lvert{X(\alpha)}\rvert
  \lvert{X(\beta)}\rvert  
\nonumber
\\
& \quad
  =
  2^nn!
  \max_{\substack{\sigma\in{S_n}
                  \\
                  \nu=0,1,\dotsc,n
                  \\
                  {\sigma(1)}<\dotsb<{\sigma(\nu)}
                  \\
                  {\sigma(\nu+1)}<\dotsb<{\sigma(n)}}}
  \sum_{\lvert\alpha\rvert\leqslant{l}}
  \sum_{\substack{\lvert\beta\rvert\leqslant{l}
                  \\
                  \beta_{\sigma(1)}
                  =
                  \alpha_{\sigma(1)},\dotsc,2\alpha_{\sigma(1)}-1
                  \\
                  \dotsb
                  \\
                  \beta_{\sigma(\nu)}
                  =
                  \alpha_{\sigma(\nu)},\dotsc,2\alpha_{\sigma(\nu)}-1
                  \\
                  \beta_{\sigma(\nu+1)}
                  =
                  2\alpha_{\sigma(\nu+1)},\dotsb
                  \\
                  \dotsb
                  \\
                  \beta_{\sigma(n)}
                  =
                  2\alpha_{\sigma(n)},\dotsb}}
  \lvert{b_{\beta,\beta-\alpha}}\rvert
  \lvert{X(\beta)}\rvert
  \lvert{X(\beta-\alpha)}\rvert.  
\label{equation:311}
\end{align}
By the Schwarz inequality to the summation on 
$\beta_{\sigma(\nu+1)},\dotsc,\beta_{\sigma(n)}$, 
\eqref{equation:311} becomes
\begin{align}
  \lvert{}^tXB\bar{X}\rvert
& \leqslant
  2^nn!
  \max_{\substack{\sigma\in{S_n}
                  \\
                  \nu=0,1,\dotsc,n
                  \\
                  {\sigma(1)}<\dotsb<{\sigma(\nu)}
                  \\
                  {\sigma(\nu+1)}<\dotsb<{\sigma(n)}}}
  \sum_{\lvert\alpha\rvert\leqslant{l}}
  \sum_{\substack{\lvert\beta\rvert\leqslant{l}
                  \\
                  \beta_{\sigma(1)}
                  =
                  \alpha_{\sigma(1)},\dotsc,2\alpha_{\sigma(1)}-1
                  \\
                  \dotsb
                  \\
                  \beta_{\sigma(\nu)}
                  =
                  \alpha_{\sigma(\nu)},\dotsc,2\alpha_{\sigma(\nu)}-1}}
\nonumber
\\
& \quad
  \times
  \max_{\substack{\beta_{\sigma(\nu+1)}
                  =
                  2\alpha_{\sigma(\nu+1)},\dotsb
                  \\
                  \dotsb
                  \\
                  \beta_{\sigma(n)}
                  =
                  2\alpha_{\sigma(n)}/2,\dotsb}}
  \lvert{b_{\beta,\beta-\alpha}}\rvert
\nonumber
\\
& \qquad
  \times
  \left(
  \sum_{\substack{\beta_{\sigma(\nu+1)}
                  =
                  2\alpha_{\sigma(\nu+1)},\dotsb
                  \\
                  \dotsb
                  \\
                  \beta_{\sigma(n)}
                  =
                  2\alpha_{\sigma(n)},\dotsb}}
  \lvert{X(\beta)}\rvert^2
  \right)^{1/2}
\nonumber
\\
& \qquad\quad
  \times
  \left(
  \sum_{\substack{\beta_{\sigma(\nu+1)}
                  =
                  2\alpha_{\sigma(\nu+1)},\dotsb
                  \\
                  \dotsb
                  \\
                  \beta_{\sigma(n)}
                  =
                  2\alpha_{\sigma(n)},\dotsb}}
  \lvert{X(\beta-\alpha)}\rvert^2
  \right)^{1/2}.
\label{equation:312}
\end{align}
If we apply the Schwarz inequality to the summation on 
$\alpha_{\sigma(1)},\dotsb,\alpha_{\sigma(\nu)}$ 
and 
$\beta_{\sigma(1)},\dotsb,\beta_{\sigma(\nu)}$, 
then \eqref{equation:312} becomes

\begin{align*}
  \lvert{}^tXB\bar{X}\rvert
& \leqslant
  2^nn!
  \max_{\substack{\sigma\in{S_n}
                  \\
                  \nu=0,1,\dotsc,n
                  \\
                  \sigma(1)<\dotsb<\sigma(\nu)
                  \\
                  \sigma(\nu+1)<\dotsb<\sigma(n)}}
  \sum_{\alpha_{\sigma(\nu+1)}+\dotsb+\alpha_{\sigma(n)}\leqslant{l}}
\\
& \quad
  \times
  \left(
  \sum_{\substack{\alpha_{\sigma(1)}+\dotsb+\alpha_{\sigma(\nu)}
                  \leqslant
                  l-\alpha_{\sigma(\nu+1)}-\dotsb-\alpha_{\sigma(n)}
                  \\
                  \beta_{\sigma(1)}
                  =
                  \alpha_{\sigma(1)},\dotsc,2\alpha_{\sigma(1)}-1
                  \\
                  \dotsb
                  \\
                  \beta_{\sigma(\nu)}
                  =
                  \alpha_{\sigma(\nu)},\dotsc,2\alpha_{\sigma(\nu)}-1}}
  \max_{\substack{\beta_{\sigma(\nu+1)}
                  =
                  2\alpha_\sigma(\nu1),\dotsc
                  \\
                  \dotsb
                  \\
                  \beta_{\sigma(n)}
                  =
                  2\alpha_{\sigma(n)},\dotsc}}
  \lvert{b_{\beta,\beta-\alpha}}\rvert^2
  \right)^{1/2}
\\
& \qquad
  \times
  \left(
  \sum_{\substack{\beta_{\sigma(1)}
                  =
                  \alpha_{\sigma(1)},\dotsc,2\alpha_{\sigma(1)}-1
                  \\
                  \dotsb
                  \\
                  \beta_{\sigma(\nu)}
                  =
                  \alpha_{\sigma(\nu)},\dotsc,2\alpha_{\sigma(\nu)}-1
                  \\
                  \beta_{\sigma(\nu+1)}
                  =
                  2\alpha_{\sigma(\nu+1)},\dotsb
                  \\
                  \dotsb
                  \\
                  \beta_{\sigma(n)}
                  =
                  2\alpha_{\sigma(n)}/2,\dotsb}}
  \lvert{X(\beta)}\rvert^2
  \right)^{1/2}
\\
& \qquad\quad
  \times
  \left(
  \sum_{\substack{\alpha_{\sigma(1)}+\dotsb+\alpha_{\sigma(\nu)}
                  \leqslant
                  l-\alpha_{\sigma(\nu+1)}-\dotsb-\alpha_{\sigma(n)}
                  \\
                  \beta_{\sigma(\nu+1)}
                  =
                  2\alpha_{\sigma(\nu+1)},\dotsb
                  \\
                  \dotsb
                  \\
                  \beta_{\sigma(n)}
                  =
                  2\alpha_{\sigma(n)}/2,\dotsb}}
  \lvert{X(\beta-\alpha)}\rvert^2
  \right)^{1/2}
\\
& \leqslant
  2^nn!\lvert{X}\rvert^2
  \max_{\substack{\sigma\in{S_n}
                  \\
                  \nu=0,1,\dotsc,n
                  \\
                  \sigma(1)<\dotsb<\sigma(\nu)
                  \\
                  \sigma(\nu+1)<\dotsb<\sigma(n)}}
  \sum_{\alpha_{\sigma(\nu+1)}+\dotsb+\alpha_{\sigma(n)}\leqslant{l}}
\\
& \quad
  \times
  \left(
  \sum_{\substack{\alpha_{\sigma(1)}+\dotsb+\alpha_{\sigma(\nu)}
                  \leqslant
                  l-\alpha_{\sigma(\nu+1)}-\dotsb-\alpha_{\sigma(n)}
                  \\
                  \beta_{\sigma(1)}
                  =
                  \alpha_{\sigma(1)},\dotsc,2\alpha_{\sigma(1)}-1
                  \\
                  \dotsb
                  \\
                  \beta_{\sigma(\nu)}
                  =
                  \alpha_{\sigma(\nu)},\dotsc,2\alpha_{\sigma(\nu)}-1}}
  \max_{\substack{\beta_{\sigma(\nu+1)}
                  =
                  2\alpha_\sigma(\nu1),\dotsc
                  \\
                  \dotsb
                  \\
                  \beta_{\sigma(n)}
                  =
                  2\alpha_{\sigma(n)},\dotsc}}
  \lvert{b_{\beta,\beta-\alpha}}\rvert^2
  \right)^{1/2}.
\end{align*}
This completes the proof. 
\end{proof}

\section{Linear systems}
\label{section:linear}
In this section we recall the $L^2$-well-posedness for some systems 
developed in \cite{chihara}. 
Consider the initial value problem of the form
\begin{alignat}{2}
   (I_{2N}\p_t-iE_{2N}\Delta+\sum_{k=1}^nB^k(t,x)\p_k)w 
 & = 
   g(t,x) 
&
   \quad \text{in} \quad 
 & (0,T)\times\mathbb{R}^n, 
\label{equation:pde2}
\\ 
   w(0,x) 
 & = 
   w_0(x) 
& 
   \quad \text{in} \quad 
 & \mathbb{R}^n,  
\label{equation:data2}
\end{alignat}
where 
$w(t,x)$ is a $\mathbb{C}^{2N}$-valued unknown function of 
$(t,x)\in[0,T]\times\mathbb{R}^n$, 
$I_p$ is $p{\times}p$ identity matrix, 
$$
E_{2N}=[I_N]\oplus[-I_N], 
\quad
N
=
\sum_{j=0}^l
\frac{(j+n-1)!}{j!(n-1)!}
$$
which is the number of kinds of multi-indices of order at most $l$, and 
$$
B^k(t,x)
=
\begin{bmatrix}
B^{k,1}(t,x) & B^{k,2}(t,x)
\\ 
B^{k,3}(t,x) & B^{k,4}(t,x)
\end{bmatrix},
$$
$$
\quad
B^{k,m}(t,x)
=
[b_{\alpha,\beta}^{k,m}(t,x)
]_{\lvert\alpha\rvert,\lvert\beta\rvert\leqslant{l}}, 
\quad
b_{\alpha,\beta}^{k,m}(t,x)=0
\quad\text{unless}\quad
\beta\leqslant\alpha. 
$$
We here assume that the Doi-type conditions, that is, 
there exists a nonnegative function 
$\phi(t,y)$ on $[0,T]\times\mathbb{R}$ such that 
$\phi(t,y){\in}C([0,T];\mathscr{B}^2(\mathbb{R}^2))$, 
\begin{equation}
\sup_{t\in[0,T]}
\int_{-\infty}^{+\infty}
\phi(t,y)dy
+
\sup_{t\in[0,T]}
\left\lvert
\int_{-\infty}^{+\infty}
\p_t\phi(t,y)dy
\right\rvert
<+\infty, 
\label{equation:doi1}
\end{equation}
\begin{align}
& 2^nn!
  \sum_{\substack{m=1,4
        \\
        k=1,\dotsc,n}}
  \max_{\substack{\sigma\in{S_n}
        \\
        \nu=0,1,\dotsc,n
        \\
        {\sigma(1)}<\dotsb<{\sigma(\nu)}
        \\
        {\sigma(\nu+1)}<\dotsb<{\sigma(n)}}}
\nonumber
\\
& \quad
  \times
  \left(
  \sum_{\substack{\alpha_{\sigma(1)}+\dotsb+\alpha_{\sigma(\nu)}
                  \leqslant
                  l-\alpha_{\sigma(\nu+1)}-\dotsb-\alpha_{\sigma(n)}
                  \\
                  \beta_{\sigma(1)}
                  =
                  \alpha_{\sigma(1)},\dotsc,2\alpha_{\sigma(1)}-1
                  \\
                  \dotsb
                  \\
                  \beta_{\sigma(n)}
                  =
                  \alpha_{\sigma(\nu)},\dotsc,2\alpha_{\sigma(\nu)}-1}}
  \max_{\substack{\beta_{\sigma(\nu+1)}
                  =
                  2\alpha_\sigma(\nu1),\dotsc
                  \\
                  \dotsb
                  \\
                  \beta_{\sigma(n)}
                  =
                  2\alpha_{\sigma(n)},\dotsc}}
  \lvert{b_{\beta,\beta-\alpha}^{k,m}(t,x)}\rvert^2
  \right)^{1/2}
\nonumber
\\
& \leqslant
  \phi(t,x_j)
\label{equation:doi2}
\end{align}
for 
$(t,x)=(t,x_1,\dotsc,x_n)\in[0,T]\times\mathbb{R}^n$, 
$j=1,\dotsc,n$. 
One can prove that 
the initial value problem 
\eqref{equation:pde2}-\eqref{equation:data2} 
is $L^2$-well-posed by using the block-diagonalization technique 
in \cite{chihara}, 
and Doi's transformation in \cite{doi1}. 
See \cite{chihara} for the detail. 
To state the energy inequality needed later, 
we here introduce some pseudodifferential operators as follows: 
\begin{align*}
  \Lambda(t)
& =
  I_{2N}
  +
  \frac{1}{2}
  E_{2N}
  \sum_{k=1}^n 
  \begin{bmatrix}
  0 & B^{k,2}(t,x)
  \\
  B^{k,3}(t,x) & 0 
  \end{bmatrix}
  \p_k
  (\mu^2-\Delta)^{-1},
\\
  K(t)
& =
  [I_Nk_1(t,x,D)]\oplus[I_Nk_1^\prime(t,x,D)], 
\\
  k_1(t,x,\xi)
& =
  e^{-p(t,x,\xi)}, 
  \quad
  k_1^\prime(t,x,\xi)
  =
  e^{p(t,x,\xi)}, 
\\
  p(t,x,\xi)
& =
  \sum_{k=1}^n
  \int_0^{x_k}\phi(t,y)dy
  \xi_k(\mu^2+\lvert\xi\rvert^2)^{-1/2}. 
\end{align*}
It is easy to see that 
$K(t)\Lambda(t)$ is automorphic on $(L^2)^{2N}$ 
provided that $\mu>0$ is sufficiently large. 
More precisely, 
there exists a positive constants 
$M$ and $\mu$ depending only on 
\begin{align*}
& \sup_{t\in[0,T]}
  \lVert\phi(t,\cdot)\rVert_{\mathscr{B}^2}
\\
  +
& \sup_{t\in[0,T]}
  \int_{-\infty}^{+\infty}
  \phi(t,y)dy
  +
  \sup_{t\in[0,T]}
  \left\lvert
  \int_{-\infty}^{+\infty}
  \p_t\phi(t,y)dy
  \right\rvert
\\
  +
& 2^nn!
  \sum_{\substack{m=1,4
        \\
        k=1,\dotsc,n}}
  \max_{\substack{\sigma\in{S_n}
        \\
        \nu=0,1,\dotsc,n
        \\
        {\sigma(1)}<\dotsb<{\sigma(\nu)}
        \\
        {\sigma(\nu+1)}<\dotsb<{\sigma(n)}}}
\\
  \times
& \left(
  \sum_{\substack{\alpha_{\sigma(1)}+\dotsb+\alpha_{\sigma(\nu)}
                  \leqslant
                  l-\alpha_{\sigma(\nu+1)}-\dotsb-\alpha_{\sigma(n)}
                  \\
                  \beta_{\sigma(1)}
                  =
                  \alpha_{\sigma(1)},\dotsc,2\alpha_{\sigma(1)}-1
                  \\
                  \dotsb
                  \\
                  \beta_{\sigma(n)}
                  =
                  \alpha_{\sigma(\nu)},\dotsc,2\alpha_{\sigma(\nu)}-1}}
  \max_{\substack{\beta_{\sigma(\nu+1)}
                  =
                  2\alpha_\sigma(\nu1),\dotsc
                  \\
                  \dotsb
                  \\
                  \beta_{\sigma(n)}
                  =
                  2\alpha_{\sigma(n)},\dotsc}}
  \lVert
  b_{\beta,\beta-\alpha}^{k,m}(t,\cdot)
  \rVert^2_{\mathscr{B}^2}
  \right)^{1/2}
\end{align*}
such that 
$$
M^{-1}\lVert{w}\rVert
\leqslant
\lVert{K(t)\Lambda(t)w}\rVert
\leqslant
M\lVert{w}\Vert.
$$
Now we state $L^2$-well-posedness. 
\begin{lemma}
\label{theorem:doi} 
Assume \eqref{equation:doi1} and \eqref{equation:doi2}. 
Then, the initial value problem 
{\rm \eqref{equation:pde2}}-{\rm \eqref{equation:data2}} 
is $L^2$-well-posed, that is, 
for any $w_0\in(L^2)^{2N}$ 
and 
$g{\in}L^1(0,T;(L^2)^{2N})$, 
{\rm \eqref{equation:pde2}}-{\rm \eqref{equation:data2}} 
has a unique solution $w$ belonging to 
$C([0,T];(L^2)^{2N})$. 
Moreover, $w$ satisfies 
\begin{align}
\lVert{K(t)\Lambda(t)w}\rVert^2
& =
  \lVert{K(t)\Lambda(t)w_0}\rVert^2
\nonumber
\\
& +
  \int_0^t
  2\re
  (Q(\tau)K(\tau)\Lambda(\tau)w(\tau),
   K(\tau)\Lambda(\tau)w(\tau))
  d\tau
\nonumber
\\
& +
  \int_0^t
  2\re
  (R(\tau)K(\tau)\Lambda(\tau)w(\tau),
   K(\tau)\Lambda(\tau)w(\tau))
  d\tau
\nonumber
\\
& +
  \int_0^t
  2\re
  (K(\tau)\Lambda(\tau)w(\tau),
   K(\tau)\Lambda(\tau)g(\tau))
  d\tau, 
\label{equation:identity}
\end{align}
for all $t\in[0,T]$, where 
$$  
\sigma(Q(t))(x,\xi)
=
\sum_{j=1}^n
2\phi(t,x_j)\xi_j^2(\mu^2+\lvert\xi\rvert^2)^{-1/2}
+
i\sum_{j=1}^n
B^{k,\text{diag}}(t,x)\xi_j,    
$$
$$
B^{k,\text{diag}}(t,x)
=
\begin{bmatrix}
B^{k,1}(t,x) & 0
\\
0 & B^{k,4}(t,x)
\end{bmatrix},
$$
$$
\sup_{t\in[0,T]}
\lVert{R(t)w}\rVert
\leqslant
CM\lVert{w}\rVert.
$$
\end{lemma}
\begin{proof}
The proof of Lemma~\ref{theorem:doi} is basically same as that of 
\cite[Lemma~6]{chihara}. 
In particular, 
we make use of Lemma~\ref{theorem:summation2} 
to evaluate the matrices of coefficients. 
We here omit the detail. 
\end{proof}

\section{Nonlinear estimates}
\label{section:nonlinear}
This section is devoted to estimating nonlinearity. 
For the sake of convenience, we use the following notation. 
$$
X^l_{\theta,s,r}
=
\left(
\sum_{\lvert\alpha\rvert\leqslant{l}}
\frac{r^{2\lvert\alpha\rvert}}{\alpha_\ast!^{2s}}
\lVert{J^\alpha}u\rVert_\theta^2
\right)^{1/2}, 
$$
$$
r>0,
\quad 
\alpha_\ast!
=
\prod_{j=1}^n
\max\{\alpha_j-1,0\}!, 
\quad 
(\alpha)
=
\prod_{j=1}^n
\max\{\alpha_j,1\}. 
$$
\par
First, we obtain an estimate related to the commutator 
$[J^\alpha,\p_j]$. 
\begin{lemma}
\label{theorem:commutation}
For $u\in\mathscr{S}$ and $j=1,\dotsc,n$, 
\begin{equation}
\left(
\sum_{\lvert\alpha\rvert\leqslant{l}}
\frac{r^{2\lvert\alpha\rvert}}{\alpha_\ast!^{2s}}
\lVert{J^\alpha}\p_ju\rVert_{\theta-1}^2
\right)^{1/2} 
\leqslant
\sqrt{2}(1+2r)X^l_{\theta,s,r}. 
\label{equation:commutation}
\end{equation}
\end{lemma}
\begin{proof}
A simple computation gives 
$$
[J^\alpha,\p_j]
=
\begin{cases}
-\alpha_jJ^{\alpha-e_j} & \alpha_j\ne0
\\
0 & \alpha_j=0 
\end{cases}
$$ 
On one hand, when $\alpha_j=0$, we have 
\begin{equation}
\lVert{J^\alpha\p_ju}\rVert_{\theta-1}
=
\lVert{\p_jJ^\alpha}u\rVert_{\theta-1}
\leqslant
\lVert{J^\alpha}u\rVert_{\theta}.
\label{equation:zerocase} 
\end{equation}
On the other hand, when $\alpha\ne0$, we have 
\begin{align}
  \lVert{J^\alpha\p_ju}\rVert_{\theta-1}
& \leqslant
  \lVert{\p_jJ^\alpha}u\rVert_{\theta-1}
  +
  \alpha_j
  \lVert{J^{\alpha-e_j}u}\rVert_{\theta-1}
\nonumber
\\
& \leqslant
  \lVert{J^\alpha}u\rVert_{\theta}
  +
  \alpha_j
  \lVert{J^{\alpha-e_j}u}\rVert_{\theta}.
\label{equation:nonzerocase} 
\end{align}
Substituting \eqref{equation:zerocase} and \eqref{equation:nonzerocase} 
into the left hand side of \eqref{equation:commutation}, 
we deduce 
\begin{align*}
& \left(
  \sum_{\lvert\alpha\rvert\leqslant{l}}
  \frac{r^{2\lvert\alpha\rvert}}{\alpha_\ast!^{2s}}
  \lVert{J^\alpha\p_ju}\rVert_{\theta-1}^2
  \right)^{1/2} 
\\
& \leqslant
  \sqrt{2}
  \left(
  \sum_{\lvert\alpha\rvert\leqslant{l}}
  \frac{r^{2\lvert\alpha\rvert}}{\alpha_\ast!^{2s}}
  \lVert{J^\alpha}u\rVert_{\theta}^2
  \right)^{1/2} 
  +
  \sqrt{2}
  \left(
  \sum_{\lvert\alpha\rvert\leqslant{l}}
  \frac{r^{2\lvert\alpha\rvert}\alpha_j^2}{\alpha_\ast!^{2s}}
  \lVert{J^{\alpha-e_j}u}\rVert_{\theta}^2
  \right)^{1/2} 
\\
& =
  \sqrt{2}
  \left(
  \sum_{\lvert\alpha\rvert\leqslant{l}}
  \frac{r^{2\lvert\alpha\rvert}}{\alpha_\ast!^{2s}}
  \lVert{J^\alpha}u\rVert_{\theta}^2
  \right)^{1/2} 
  +
  \sqrt{2}r
  \left(
  \sum_{\lvert\beta\rvert\leqslant{l-1}}
  \frac{r^{2\lvert\beta\rvert}(\beta_j+1)^2}{(\beta+e_j)_\ast!^{2s}}
  \lVert{J^{\beta}u}\rVert_{\theta}^2
  \right)^{1/2} 
\\
& \leqslant
  \sqrt{2}(1+2r)X^l_{\theta,s,r}.
\end{align*}
Here we used 
$(\beta_j+1)^2/\max\{\beta_j^2,1\}\leqslant4$. 
\end{proof}
\par
Secondly, we show the lower order estimates of the nonlinearity. 
\begin{lemma}
\label{theorem:lower}
Let $\theta>n/2+2$. Set $\psi=\lvert{x}\rvert^2/4t$ and 
\begin{align*}
  f_{\theta,\alpha}
& =
  \langle{D}\rangle^\theta
  J^\alpha
  f(u,\p{u})
\\
& -
  \sum_{j=1}^n
  \sum_{\alpha^\prime\leqslant\alpha}
  \frac{\alpha!}{\alpha^\prime!(\alpha-\alpha^\prime)!}
\\
& \quad
  \times
  \biggl\{
  (2it)^{\lvert\alpha^\prime\rvert}
  \p^{\alpha^\prime}
  \frac{\p{f}}{\p{v_j}}
  (e^{-i\psi}u,\p{e^{-i\psi}\p{u}})
  \p_j\langle{D}\rangle^\theta{J^{\alpha-\alpha^\prime}}u
\\
& \quad\qquad
  +
  (-1)^{\lvert\alpha-\alpha^\prime\rvert}
  \p^{\alpha^\prime}
  \left(
  e^{2i\psi}
  \p^{\alpha^\prime}
  \frac{\p{f}}{\p\bar{v}_j}
  (e^{-i\psi}u,\p{e^{-i\psi}\p{u}})
  \right)
  \p_j
  \overline{\langle{D}\rangle^\theta{J^{\alpha-\alpha^\prime}}u}
  \biggr\}.    
\end{align*}
Then, there exist a positive constant 
$C_{\theta,n}$ 
such that for any 
$u\in\mathscr{S}$ and $l\in\mathbb{N}$, 
$$
\left(
\sum_{\lvert\alpha\rvert\leqslant{l}}
\frac{r^{2\lvert\alpha\rvert}}{\alpha_\ast!^{2s}}
\lVert{f_{\theta,\alpha}}\rVert^2
\right)^{1/2} 
\leqslant
C_{\theta,n}
\sum_{p=1}^\infty
A_p
(C_{\theta,n}X^{l-1}_{\theta,s,r})^{2p}
(C_{\theta,n}X^l_{\theta,s,r}).
$$
\end{lemma}
\begin{proof}
For any multi-indices 
$\beta,\bar{\beta}\in(\mathbb{N}\cup\{0\})^{n+1}$ 
satisfying 
$\lvert\beta\rvert=p+1$ 
and 
$\lvert\bar{\beta}\rvert=p$, 
and for $q=0,1,\dotsc,2p$, set 
$$
\p_{q,\beta\bar{\beta}}
=
\begin{cases}
1 
& 
(q\leqslant\beta_0)
\\
1 
& 
(2p+1\leqslant{q}\leqslant2p+\beta_0)
\\
\p_j 
& 
(\beta_0+\dotsb+\beta_{j-1}\leqslant{q}\leqslant\beta_0+\dotsb+\beta_j-1)
\\
\p_j 
& 
(p+\bar{\beta}_0+\dotsb+\bar{\beta}_{j-1}+1\leqslant{q}\leqslant{p}+\bar{\beta}_0+\dotsb+\bar{\beta}_j) 
\end{cases}
$$
where 
$\beta=(\beta_0,\beta_1,\dotsc,\beta_n)$ 
and 
$\bar{\beta}=(\bar{\beta}_0,\bar{\beta}_1,\dotsc,\bar{\beta}_n)$. 
We split $f_{\theta,\alpha}$ into two parts: 
$f_{\theta,\alpha}=g_{\theta,\alpha}+h_{\theta,\alpha}$, 
\begin{align*}
  g_{\theta,\alpha}
& =
  \sum_{p=1}^\infty
  \sum_{\substack{{\lvert\beta\rvert=p+1} \\ {\lvert\bar{\beta}}\rvert=p}}
  f_{\beta\bar\beta}
  \sum_{\alpha^0+\dotsb+\alpha^{2p}=\alpha}
  \frac{\alpha!}{\alpha^0!\dotsb\alpha^{2p}!}
  (-1)^{\lvert\alpha^{p+1}+\dotsb+\alpha^{2p}\rvert}
\\
& \quad
  \times 
  \Biggl\{
  \langle{D}\rangle^\theta
  \left(
  \prod_{q=0}^pJ^{\alpha^q}\p_{q,\beta\bar{\beta}}u
  \prod_{q^\prime=p+1}^{2p}
  \overline{J^{\alpha^{q^\prime}}\p_{q^\prime,\beta\bar\beta}u}
  \right)
\\
& \qquad 
  -
  \sum_{q_1=\beta_0}^p
  \langle{D}\rangle^\theta  
  J^{\alpha^{q_1}}\p_{q_1,\beta\bar{\beta}}u
  \prod_{\substack{{q=0} \\ {q{\ne}q_1}}}^p
  J^{\alpha^q}\p_{q,\beta\bar{\beta}}u
  \prod_{q^\prime=p+1}^{2p}
  \overline{J^{\alpha^{q^\prime}}\p_{q^\prime,\beta\bar{\beta}}u}
\\
& \quad\qquad 
  -
  \prod_{q=0}^p
  J^{\alpha^q}\p_{q,\beta\bar{\beta}}u
  \sum_{q_1=p+1+\bar{\beta}_0}^{2p}
  \overline{\langle{D}\rangle^\theta 
            J^{\alpha^{q_1}}\p_{q_1,\beta\bar{\beta}}u} 
  \prod_{\substack{{q^\prime=p+1} \\ {q^\prime{\ne}q_1}}}^p
  \overline{J^{\alpha^{q^\prime}}\p_{q^\prime,\beta\bar{\beta}}u}  
  \Biggr\}, 
\\
  h_{\theta,\alpha}
& =
  \sum_{p=1}^\infty
  \sum_{\substack{{\lvert\beta\rvert=p+1} \\ {\lvert\bar{\beta}}\rvert=p}}
  f_{\beta\bar\beta}
  \sum_{\alpha^0+\dotsb+\alpha^{2p}=\alpha}
  \frac{\alpha!}{\alpha^0!\dotsb\alpha^{2p}!}
  (-1)^{\lvert\alpha^{p+1}+\dotsb+\alpha^{2p}\rvert}
\\
& \quad
  \times 
  \Biggl\{
  \sum_{q_1=\beta_0}^p
  \langle{D}\rangle^\theta  
  [J^{\alpha^{q_1}},\p_{q_1,\beta\bar{\beta}}]u
  \prod_{\substack{{q=0} \\ {q{\ne}q_1}}}^p
  J^{\alpha^q}\p_{q,\beta\bar{\beta}}u
  \prod_{q^\prime=p+1}^{2p}
  \overline{J^{\alpha^{q^\prime}}\p_{q^\prime,\beta\bar{\beta}}u}
\\
& \quad\qquad 
  +
  \prod_{q=0}^p
  J^{\alpha^q}\p_{q,\beta\bar{\beta}}u
  \sum_{q_1=p+1+\bar{\beta}_0}^{2p}
  \overline{\langle{D}\rangle^\theta 
            [J^{\alpha^{q_1}},\p_{q_1,\beta\bar{\beta}}]u} 
  \prod_{\substack{{q^\prime=p+1} \\ {q^\prime{\ne}q_1}}}^p
  \overline{J^{\alpha^{q^\prime}}\p_{q^\prime,\beta\bar{\beta}}u}  
  \Biggr\}.
\end{align*}
Using Theorem~\ref{theorem:kato-ponce} and Lemma~\ref{theorem:leibniz}, we deduce 
\begin{align*}
& \Biggl\lVert
  \langle{D}\rangle^\theta
  \left(
  \prod_{q=0}^pJ^{\alpha^q}\p_{q,\beta\bar{\beta}}u
  \prod_{q^\prime=p+1}^{2p}
  \overline{J^{\alpha^{q^\prime}}\p_{q^\prime,\beta\bar\beta}u}
  \right)
\\
& \quad 
  -
  \sum_{q_1=\beta_0}^p
  \langle{D}\rangle^\theta  
  J^{\alpha^{q_1}}\p_{q_1,\beta\bar{\beta}}u
  \prod_{\substack{{q=0} \\ {q{\ne}q_1}}}^p
  J^{\alpha^q}\p_{q,\beta\bar{\beta}}u
  \prod_{q^\prime=p+1}^{2p}
  \overline{J^{\alpha^{q^\prime}}\p_{q^\prime,\beta\bar{\beta}}u}
\\
& \qquad 
  -
  \prod_{q=0}^p
  J^{\alpha^q}\p_{q,\beta\bar{\beta}}u
  \sum_{q_1=p+1+\bar{\beta}_0}^{2p}
  \overline{\langle{D}\rangle^\theta 
            J^{\alpha^{q_1}}\p_{q_1,\beta\bar{\beta}}u} 
  \prod_{\substack{{q^\prime=p+1} \\ {q^\prime{\ne}q_1}}}^p
  \overline{J^{\alpha^{q^\prime}}\p_{q^\prime,\beta\bar{\beta}}u}  
  \Biggr\rVert 
\\
& \leqslant
  C_\theta^{2p+1-(\beta_0+\bar{\beta}_0)}
  (2p+1-(\beta_0+\bar{\beta}_0))
  \prod_{q=0}^{2p}
  \lVert
  J^{\alpha^q}\p_{q,\beta\bar{\beta}}u  
  \rVert_{\theta-1}
\\
& \leqslant
  C_\theta^{2p+1}
  (2p+1)
  \prod_{q=0}^{2p}
  \lVert
  J^{\alpha^q}\p_{q,\beta\bar{\beta}}u  
  \rVert_{\theta-1}
\end{align*}
This estimate and the Minkowski inequality show that 
\begin{align*}
& \sum_{\lvert\alpha\rvert\leqslant{l}}
\frac{r^{2\lvert\alpha}\rvert}{\alpha_\ast!^{2s}}
\lVert{g_{\theta,\alpha}}\rVert^2
\\
& \leqslant
  \sum_{\lvert\alpha\rvert\leqslant{l}}
  \frac{r^{2\lvert\alpha\rvert}}{\alpha_\ast!^{2s}}
  \Biggl\{
  \sum_{p=1}^\infty
  A_p(2p+1)C_\theta^{2p+1}
  \sum_{\alpha^0+\alpha^{2p}=\alpha}
  \frac{\alpha!}{\alpha^0!\dotsb\alpha^{2p}!}
  \prod_{q=0}^{2p}
  \lVert{J^{\alpha^q}\p_{q,\beta\bar{\beta}}u}\rVert_{\theta-1}
  \Biggr\}^2
\\
& \leqslant
  \biggl[
  \sum_{p=1}^\infty
  A_p(2p+1)C_\theta^{2p+1}
  \biggl\{
  \sum_{\lvert\alpha\rvert\leqslant{l}}
\\
& \times
  \biggl(
  \sum_{\alpha^0+\alpha^{2p}=\alpha}
  \frac{\alpha!}{\alpha^0!\dotsb\alpha^{2p}!}
  \frac{r^{2\lvert\alpha\rvert}}{\alpha_\ast!^{2s}}
  \prod_{q=0}^{2p}
  \lVert{J^{\alpha^q}\p_{q,\beta\bar{\beta}}u}\rVert_{\theta-1}
  \biggr)^2
  \biggr\}^{1/2}
  \biggr]^2
\\
& \leqslant 
  \biggl[
  \sum_{p=1}^\infty
  A_p(2p+1)C_\theta^{2p+1}
  \biggl\{
  \sum_{\lvert\alpha\rvert\leqslant{l}}
\\
& \times
  \biggl(
  \sum_{\alpha^0+\alpha^{2p}=\alpha}
  \frac{\alpha!}{\alpha^0!\dotsb\alpha^{2p}!}
  \left(
  \frac{\alpha_\ast^0!\dotsb\alpha_\ast^{2p}!}{\alpha_\ast!}
  \right)^{s-1}
  \prod_{q=0}^{2p}
  \frac{r^{2\lvert\alpha\rvert}}{\alpha_\ast!^{2s}}
  \lVert{J^{\alpha^q}\p_{q,\beta\bar{\beta}}u}\rVert_{\theta-1}
  \biggr)^2
  \biggr\}^{1/2}
  \biggr]^2
\\
& \leqslant 
  \biggl[
  \sum_{p=1}^\infty
  A_p(2p+1)C_\theta^{2p+1}
  \biggl\{
  \sum_{\lvert\alpha\rvert\leqslant{l}}
  \biggl(
  \sum_{\alpha^0+\alpha^{2p}=\alpha}
  \frac{(\alpha)}{(\alpha^0)\dotsb(\alpha^{2p})}
  \prod_{q=0}^{2p}
  A(\alpha^q)
  \biggr)^2
  \biggr\}^{1/2}
  \biggr]^2,
\end{align*}
where 
$$
A(\alpha)
=
\frac{r^{2\lvert\alpha}\rvert}{\alpha_\ast!^{2s}}
\max
\Bigl\{
\lVert{J^\alpha}u\rVert_{\theta-1}, 
\lVert{J^\alpha}\p_1u\rVert_{\theta-1},
\dotsc, 
\lVert{J^\alpha}\p_nu\rVert_{\theta-1}
\Bigr\}.
$$
Using Lemma~\ref{theorem:summation1} and \eqref{equation:commutation}, 
we deduce 
\begin{align}
& \sum_{\lvert\alpha\rvert\leqslant{l}}
  \frac{r^{2\lvert\alpha}\rvert}{\alpha_\ast!^{2s}}
  \lVert{g_{\theta,\alpha}}\rVert^2
\nonumber
\\
& \leqslant
  \biggl[
  \sum_{p=1}^\infty
  A_p(2p+1)C_\theta^{2p+1}(a^n)^{2p}
  \left(\sum_{\lvert\alpha\rvert\leqslant{l-1}}A(\alpha)^2\right)^p
  \left(\sum_{\lvert\alpha\rvert\leqslant{l}}A(\alpha)^2\right)^{1/2}
  \biggr]^2
\nonumber
\\
& \leqslant
  \biggl[
  \sum_{p=1}^\infty
  A_p(2p+1)C_\theta^{2p+1}(a^n)^{2p}C_r^{2p+1}
  (X_{\theta,s,r}^{l-1})^{2p}
  X_{\theta,s,r}^l
  \biggr]^2. 
\label{equation:estofg}
\end{align}
In the same way, we can get 
\begin{equation}
\sum_{\lvert\alpha\rvert\leqslant{l}}
\frac{r^{2\lvert\alpha}\rvert}{\alpha_\ast!^{2s}}
\lVert{h_{\theta,\alpha}}\rVert^2
\leqslant
\biggl[
\sum_{p=1}^\infty
A_pC_{\theta,n}^{2p+1}
(X_{\theta,s,r}^{l-1})^{2p}
X_{\theta,s,r}^l
\biggr]^2. 
\label{equation:estofh} 
\end{equation}
Combining \eqref{equation:estofg} and \eqref{equation:estofh}, 
we obtain Lemma~\ref{theorem:lower}. 
\end{proof}
To use the linear estimates obtained in Section~\ref{section:linear}, 
we need the estimates of coefficient matrices of the system for 
${}^t\left[
r^{\lvert\alpha\rvert}\langle{D}\rangle^\theta J^\alpha u/\alpha_\ast!^s,
r^{\lvert\alpha\rvert}\langle{D}\rangle^\theta \overline{J^\alpha u}/\alpha_\ast!^s
\right]_{\lvert\alpha\rvert\leqslant{l}}$. 
For this purpose, we here define some matrices appearing in the system as follows. 
For $j=1,\dotsc,n$ and $l\in\mathbb{N}$, we set 
$$
B_j^l
=
\begin{bmatrix}
C_{j,1}^l & C_{j,2}^l
\\
\overline{C_{j,2}^l} & \overline{C_{j,1}^l}
\end{bmatrix} 
\quad
C_{j,1}^l
=
\Bigl[b_{j,1,\alpha\beta}^l\Bigr]_{\lvert\alpha\rvert,\lvert\beta\rvert\leqslant{l}},
\quad
C_{j,2}^l
=
\Bigl[b_{j,2,\alpha\beta}^l\Bigr]_{\lvert\alpha\rvert,\lvert\beta\rvert\leqslant{l}},
$$
$$
b_{j,1,\alpha\beta}^l
=
\begin{cases}
\dfrac{\beta_\ast!^sr^{\lvert\alpha-\beta\rvert}}{\alpha_\ast!^s}
\displaystyle\sum_{p=1}^\infty
\displaystyle\sum_{\vert\gamma\rvert-1=\lvert\bar{\gamma}\rvert=p}
f_{\gamma\bar{\gamma}}\gamma_j
&
\\
\quad\times
\displaystyle\sum_{j,1}
\frac{\alpha!}{\alpha^0!\dotsb\alpha^{2p}!\beta!}
(-1)^{\lvert\alpha^{p+1}+\dotsb+\alpha^{2p}\rvert}
& 
\\
\qquad\times
\displaystyle\prod_{\substack{q=0 \\ q\ne\gamma_1+\dotsb+\gamma_{j-1}}}^p
J^{\alpha^q}\p_{q,\gamma\bar{\gamma}}u
\displaystyle\prod_{q^\prime=p+1}^{2p}
\overline{J^{\alpha^{q^\prime}}\p_{q^\prime,\gamma\bar{\gamma}}u}
& 
\ \text{if}\quad \beta\leqslant\alpha,
\\
0 
&
\ \text{otherwise},
\end{cases}
$$
$$
\sum_{j,1}
=
\sum_{\alpha^0+\dotsb+\alpha^{m-1}+\alpha^{m+1}\dotsb+\alpha^{2p}=\alpha-\beta},
\quad
m=\gamma_0+\dotsb+\gamma_{j-1},
$$
$$
b_{j,2,\alpha\beta}^l
=
\begin{cases}
\dfrac{\beta_\ast!^sr^{\lvert\alpha-\beta\rvert}}{\alpha_\ast!^s}
\displaystyle\sum_{p=1}^\infty
\displaystyle\sum_{\vert\gamma\rvert-1=\lvert\bar{\gamma}\rvert=p}
f_{\gamma\bar{\gamma}}\bar{\gamma}_j
&
\\
\quad\times
\displaystyle\sum_{j,2}
\frac{\alpha!}{\alpha^0!\dotsb\alpha^{2p}!\beta!}
(-1)^{\lvert\alpha^{p+1}+\dotsb+\alpha^{2p}\rvert}
& 
\\
\qquad\times
\displaystyle\prod_{q=0}^p
J^{\alpha^q}\p_{q,\gamma\bar{\gamma}}u
\displaystyle\prod_{\substack{q^\prime=p+1 \\ q^\prime\ne p+\bar{\gamma}_0+\dotsb+\bar{\gamma}_{j-1}+1}}^{2p}
\overline{J^{\alpha^{q^\prime}}\p_{q^\prime,\gamma\bar{\gamma}}u}
& 
\ \text{if}\quad \beta\leqslant\alpha,
\\
0 
&
\ \text{otherwise},
\end{cases}
$$
$$
\sum_{j,2}
=
\sum_{\alpha^0+\dotsb+\alpha^{m-1}+\alpha^{m+1}\dotsb+\alpha^{2p}=\alpha-\beta},
\quad
m=p+\bar{\gamma}_0+\dotsb+\bar{\gamma}_{j-1}.
$$
We need the estimates of the above matrices later. 
\begin{lemma}
\label{theorem:coefficients}
Let $\theta>n/2+3$. 
Then, there exists a positive constant $C_{\theta,n}$ which is 
independent of $l\in\mathbb{N}$, such that 
for $j=1,\dotsc,n$ and 
for any $u{\in}C^1([0,T];\mathscr{S})$ solving \eqref{equation:pde1}, 
\begin{equation}
\lVert{B_j^l}(t)\rVert_{\mathscr{B}^2}
\leqslant
C_{\theta,n}
\sum_{p=1}^\infty
A_p
(C_{\theta,n}X_{\theta,s,r}^{l-1}(t))^{2p-1} 
(C_{\theta,n}X_{\theta,s,r}^l(t)),
\label{equation:flat2}
\end{equation}
\begin{align}
  \lVert\p_tB_l^l(t)\rVert_{\mathscr{B}^0}
& \leqslant
  C_{\theta,n}
  \sum_{p=1}^\infty
  A_p
  (C_{\theta,n}X_{\theta,s,r}^{l-1}(t))^{2p-1}
  (C_{\theta,n}X_{\theta,s,r}^l(t))
\nonumber
\\
& \times
  \left(
  1
  +
  \sum_{q=1}^\infty(C_{\theta,n}X_{\theta,s,r}^{l-1}(t))^{2q}
  \right),
\label{equation:flatt}
\end{align}
\end{lemma}
\begin{proof}
Simple computation shows that 
\begin{align}
  \lVert{B_j^l(t)}\rVert_{\mathscr{B}^2}
& \leqslant
  2\sum_{k=1,2}  
  \lVert{C_{j,k}^l(t)}\rVert_{\mathscr{B}^2}
\nonumber
\\
& \leqslant
  2\sum_{k=1,2}
  \left\lvert
  \Bigl[
  \lVert
  b_{j,k,\alpha\beta}^l(t)
  \rVert_{\mathscr{B}^2}
  \Bigr]_{\lvert\alpha\rvert, \lvert\beta\rvert \leqslant l}
  \right\rvert
\nonumber
\\
& =
  \sum_{k=1,2}I_k(t). 
\label{eqution:sobolev}
\end{align}
We show that $I_1(t)$ bounded by the right hand side of \eqref{equation:flat2}. 
For the sake of convenience, set 
$$
A(\alpha)
=
\begin{cases}
\max
\Bigl\{
\lVert{J^\alpha{u}}\rVert, 
\lVert{J^\alpha}\p_1u\rVert,
\dotsc,
\lVert{J^\alpha}\p_nu\rVert
\Bigr\} 
& 
\ \text{if}\quad \lvert\alpha\rvert\leqslant{l-1},
\\
0 
& 
\ \text{if}\quad \lvert\alpha\rvert=l.
\end{cases}
$$
Using \eqref{equation:chain}, we have for any $\beta\leqslant\alpha$ 
\begin{align*}
  \lVert{b_{j,1,\alpha(\alpha-\beta)}}\rVert_{\theta-1}
& \leqslant
  \frac{(\alpha-\beta)_\ast!^sr^{\lvert\beta\rvert}}{\alpha_\ast!^s}
  C_{\theta,n}
  \sum_{p=1}^\infty
  A_pp^2C_{\theta,n}^{2p}
\\
& \quad \times
  \sum_{\alpha^1+\dotsb+\alpha^{2p}=\beta}
  \frac{\alpha!}{\alpha^1!\dotsb\alpha^{2p}!}
  \prod_{q=1}^{2p}A(\alpha^q)
\\
& \leqslant
  C_{\theta,n}
  \sum_{p=1}^\infty
  A_pp^2C_{\theta,n}^{2p}
\\
& \quad \times
  \sum_{\alpha^1+\dotsb+\alpha^{2p}=\beta}
  \frac{(\alpha-\beta)_\ast!^{s-1}\alpha_\ast^1!^{s-1}\dotsb\alpha_\ast^{2p}!^{s-1}}{\alpha_\ast!^{s-1}}
\\
& \qquad \times
  \frac{(\alpha)}{(\alpha^1)\dotsb(\alpha^{2p})((\alpha-\beta))}
  \prod_{q=1}^{2p}
  A(\alpha^q)
  \frac{r^{\lvert\alpha^q\rvert}}{\alpha_\ast^q!^s} 
\\
& \leqslant
  C_{\theta,n}
  \frac{(\alpha)}{(\beta)((\alpha-\beta))}
  \sum_{p=1}^\infty
  A_pp^2C_{\theta,n}^{2p}
  \frac{(\beta)}{(\alpha^1)\dotsb(\alpha^{2p})}
  \prod_{q=1}^{2p}
  A(\alpha^q)
  \frac{r^{\lvert\alpha^q\rvert}}{\alpha_\ast^q!^s}. 
\end{align*}
Using \eqref{equation:summation2} and the above estimates, we deduce 
\begin{align}
  I_1(t) 
& \leqslant
  2^nn!
  \max_{\substack{\sigma{\in}S_n
                  \\
                  \nu=0,1,\dotsc,n-1
                  \\
                  \sigma(1)<\dotsb<\sigma(\nu)
                  \\
                  \sigma(\nu+1)<\dotsb<\sigma(n)}}
  \sum_{\beta_{\sigma(\nu+1)}+\dotsb+\beta_{\sigma(n)}\leqslant l}
\nonumber
\\
& \times
  \Biggl\{
  \sum_{\substack{\beta_{\sigma(1)}
                  \\
                  \vdots
                  \\
                  \beta_{\sigma(\nu)}}}
  \sum_{\substack{\alpha_{\sigma(1)}<2\beta_{\sigma(1)}
                  \\
                  \vdots
                  \\
                  \alpha_{\sigma(\nu)}<2\beta_{\sigma(\nu)}}}
  \max_{\substack{\alpha_{\sigma(\nu+1)}\geqslant2\beta_{\sigma(\nu1)}
                  \\
                  \vdots
                  \\
                  \alpha_{\sigma(n)}\geqslant2\beta_{\sigma(n)}}} 
  \frac{(\alpha)^2}{(\beta)^2((\alpha-\beta))^2}
\nonumber
\\
& \times
  \left(
  \sum_{p=1}^\infty
  A_pp^2C_{\theta,n}^{2p}
  \sum_{\alpha^1+\dotsb+\alpha^{2p}=\beta}
  \frac{(\beta)}{(\alpha^1)\dotsm(\alpha^{2p})}
    \prod_{q=1}^{2p}
  A(\alpha^q)
  \frac{r^{\lvert\alpha^q\rvert}}{\alpha_\ast^q!^s}
  \right)^2
  \Biggr\}^{1/2}.
\label{equation:step51}
\end{align}
We here remark that if 
$$
\alpha_{\sigma(1)}<2\beta_{\sigma(1)},
\dotsc,
\alpha_{\sigma(\nu)}<2\beta_{\sigma(\nu)},
\quad
\alpha_{\sigma(\nu+1)}\geqslant2\beta_{\sigma(\nu+1)},
\dotsc,
\alpha_{\sigma(n)}\geqslant2\beta_{\sigma(n)},
$$
then
\begin{align*}
  \frac{(\alpha)}{(\beta)((\alpha-\beta))}
& =
  \frac{(\alpha_{\sigma(1)})}{(\beta_{\sigma(1)})}
  \dotsm
  \frac{(\alpha_{\sigma(\nu)})}{(\beta_{\sigma(\nu)})}
  \frac{(\alpha_{\sigma(\nu+1)})}{((\alpha-\beta)_{\sigma(\nu+1)})}
  \dotsm
  \frac{(\alpha_{\sigma(n)})}{((\alpha-\beta)_{\sigma(n)})}
\\
& \quad\times
  \frac{1}{((\alpha-\beta)_{\sigma(1)})\dotsm((\alpha-\beta)_{\sigma(\nu)})}
  \frac{1}{(\beta_{\sigma(\nu+1)})\dotsm(\beta_{\sigma(n)})}
\\
& \leqslant
  \frac{2^n}{((\alpha-\beta)_{\sigma(1)})\dotsm((\alpha-\beta)_{\sigma(\nu)})
             (\beta_{\sigma(\nu+1)})\dotsm(\beta_{\sigma(n)})}.
\end{align*}
Substituting this into \eqref{equation:step51} and 
using the Schwarz inequality to the summation on 
$\beta_{\sigma(\nu+1)},\dotsc,\beta_{\sigma(n)}$, 
we deduce 
\begin{align*}
  I_1(t) 
& \leqslant
  2^{2n}n!^2
  \max_{\substack{\sigma{\in}S_n
                  \\
                  \nu=0,1,\dotsc,n-1
                  \\
                  \sigma(1)<\dotsb<\sigma(\nu)
                  \\
                  \sigma(\nu+1)<\dotsb<\sigma(n)}}
  2^\nu
  \sum_{\beta_{\sigma(\nu+1)}+\dotsb+\beta_{\sigma(n)}\leqslant l}
  \frac{1}{(\beta_{\sigma(\nu+1)})\dotsm(\beta_{\sigma(n)})}
\\
& \times
  \Biggl\{
  \sum_{\substack{\beta_{\sigma(1)}
                  \\
                  \vdots
                  \\
                  \beta_{\sigma(\nu)}}}
   \sum_{\substack{\alpha_{\sigma(1)}<2\beta_{\sigma(1)}
                  \\
                  \vdots
                  \\
                  \alpha_{\sigma(\nu)}<2\beta_{\sigma(\nu)}}}
\\
& \times
  \biggl(
  \sum_{p=1}^\infty
  A_pp^2C_{\theta,n}^{2p}
  \sum_{\alpha^1+\dotsb+\alpha^{2p}=\beta}
  \frac{(\beta)}{(\alpha^1)\dotsm(\alpha^{2p})}
  \prod_{q=1}^{2p}
  A(\alpha^q)
  \frac{r^{\lvert\alpha^q\rvert}}{\alpha_\ast^q!^s}
  \biggr)^2
  \Biggr\}^{1/2}.
\\
& \leqslant
  2^{2n}n!^2a^n
  \Biggl\{
  \sum_{\lvert\beta\rvert\leqslant{l}}
\\
& \times
  \biggl(
  \sum_{p=1}^\infty
  A_pp^2C_{\theta,n}^{2p}
  \sum_{\alpha^1+\dotsb+\alpha^{2p}=\beta}
  \frac{(\beta)}{(\alpha^1)\dotsm(\alpha^{2p})}
  \prod_{q=1}^{2p}
  A(\alpha^q)
  \frac{r^{\lvert\alpha^q\rvert}}{\alpha_\ast^q!^s}
  \biggr)^2
  \Biggr\}^{1/2}.
\end{align*}
The Minkowski inequality shows that 
\begin{align*}
  I_1(t)
& \leqslant
  2^{2n}n!^2a^n
  \sum_{p=1}^\infty
  A_pp^2C_{\theta,n}^{2p}
\\
& \times
  \left\{
  \sum_{\lvert\beta\rvert\leqslant{l}}
  \left(
  \sum_{\alpha^1+\dotsb+\alpha^{2p}=\beta}
  \frac{(\beta)}{(\alpha^1)\dotsm(\alpha^{2p})}
  \prod_{q=1}^{2p}
  A(\alpha^q)
  \frac{r^{\lvert\alpha^q\rvert}}{\alpha_\ast^q!^s}
  \right)^2
  \right\}^{1/2}.
\end{align*}
Applying \eqref{equation:summation1} to this, we have 
\begin{align}
  I_1(t)
& \leqslant
  2^{2n}n!^2a^n
  \sum_{p=1}^\infty
  A_pp^{2n+2}C_{\theta,n}^{2p}(a^n)^{2p}
\nonumber
\\
& \times
  \left(
  \sum_{\lvert\alpha\rvert\leqslant{l-1}}
  A(\alpha)^2
  \frac{r^{\lvert\alpha^q\rvert}}{\alpha_\ast^q!^s}
  \right)^{(2p-1)/2}
  \left(
  \sum_{\lvert\alpha\rvert\leqslant{l}}
  A(\alpha)^2
  \frac{r^{\lvert\alpha^q\rvert}}{\alpha_\ast^q!^s}
  \right)^{1/2}.
\label{equation:step52} 
\end{align}
Since $A(\alpha)=0$ for $\lvert\alpha\rvert=l$, 
and $p^{2n+2}\leqslant e^p(2n+2)!$, 
\eqref{equation:step52} is bounded by 
$$
I_1(t)
\leqslant
e^22^{2n}n!^2a^n(2n+2)!
\sum_{p=1}^n
A_p
\left\{
eC_{\theta,n}a^n
\left(
\sum_{\lvert\alpha\rvert\leqslant{l-1}}
A(\alpha)^2
\frac{r^{\lvert\alpha^q\rvert}}{\alpha_\ast^q!^s}
\right)^{1/2}
\right\}^{2p}. 
$$
Using \eqref{equation:commutation} for $r\leqslant1$, we have 
$$
I_1(t)
\leqslant
e^22^{2n}n!^2a^n(2n+2)!
\sum_{p=1}^n
A_p
\Bigl(eC_{\theta,n}a^nX_{\theta,s,r}^{l-1}(t)\Bigr)^{2p}. 
$$
The estimates of $I_2(t)$, 
\eqref{equation:flat2} and \eqref{equation:flatt} 
can be obtained similarly. 
If $u$ solves \eqref{equation:pde1}, then 
$$
\p_tJ^\alpha\p_ju
=
i\Delta{J^\alpha}\p_ju
+
J^\alpha\p_jf(u,\p{u}).
$$
Applying this formula to the time-derivatives of the matrices, 
we can show \eqref{equation:flatt} 
in the same way as \eqref{equation:flat2}. 
We here omit the detail. 
\end{proof}
%
%
\section{Uniform energy estimates}
\label{section:energy}
In this section we show that 
$\{X_{\theta,s,r}^l(t)\}_{l=0,1,2,\dotsc}$ 
is bounded in $C[-T,T]$. 
If this is true, then there exists a constant $C_0>0$ such that 
\begin{equation}
X_{\theta,s,r}^\infty(t)
=
\left(
\sum_{\alpha}
\frac{r^{2\lvert\alpha\rvert}}{\alpha_\ast!^s}
\lVert{J^\alpha{u(t)}}\rVert_\theta^2
\right)^{1/2}
\leqslant
C_0 
\label{equation:convergenceatt}
\end{equation}
for $t\in[-T,T]$. 
Let $u{\in}C([-T,T];H^\theta)$ be a solution to 
\eqref{equation:pde1}-\eqref{equation:data1} 
with $e^{\ep\langle{x}\rangle^{1/s}}u_0{\in}H^\theta$. 
Theorem~\ref{theorem:previous} shows that 
$X_{\theta,s,r}^l(t)$ is well-defined for any $l=0,1,2,\dotsc$. 
Lemma~\ref{theorem:decay} implies that 
there exist positive constants $M$ and $r$ such that 
\begin{equation}
X_{\theta,s,r}^\infty(0)
=
\left(
\sum_{\alpha}
\frac{r^{2\lvert\alpha\rvert}}{\alpha_\ast!^s}
\lVert{x^\alpha{u_0}}\rVert_\theta^2
\right)^{1/2}
\leqslant
M. 
\label{equation:convergenceat0}
\end{equation}
Without loss of generality, we may assume $r\leqslant1$. 
Since the finite sum $X_{\theta,s,r}^l(t)$ is well-defined, 
it suffices to prove \eqref{equation:convergenceatt} for small $T>0$. 
In order to make use of the energy estimates in Section~\ref{section:linear}, 
we here define functions and pseudodifferential operators:
$$
w^l
=
{}^t
\left[
\frac{r^{\lvert\alpha\rvert}}{\alpha_\ast!^s}
\langle{D}\rangle^\theta{J^\alpha}u,
\frac{r^{\lvert\alpha\rvert}}{\alpha_\ast!^s}
\overline{\langle{D}\rangle^\theta{J^\alpha}u}
\right]_{\lvert\alpha\rvert\leqslant{l}}, 
\quad
g^l
=
{}^t
\left[
\frac{r^{\lvert\alpha\rvert}}{\alpha_\ast!^s}
f_{\theta,\alpha},
\frac{r^{\lvert\alpha\rvert}}{\alpha_\ast!^s}
\overline{f_{\theta,\alpha}}
\right]_{\lvert\alpha\rvert\leqslant{l}},  
$$
$$
k_1^l(t,x,\xi)
=
\exp
\left(
A
\sum_{j=1}^n
\xi_j(\nu^2+\lvert\xi\rvert^2)^{-1/2}
\int_{-\infty}^{x_j}
\phi^l(t,s)ds
\right),
$$
\begin{align*}
  \phi^l(t,s)
& =
  \sum_{j=1}^n
  \sum_{\lvert\alpha\rvert\leqslant{l}}
  \frac{r^{\lvert\alpha\rvert}}{\alpha_\ast!^s}
  \int\dotsm\int_{\mathbb{R}^{n-1}}
\\
& \quad\times
  \lvert
  \langle{D}\rangle^\delta
  J^\alpha
  u(t,x_1,\dotsc,x_{j-1},s,x_{j+1},\dotsc,x_n)
  \rvert^2
\\
& \quad\times
  dx_1{\dotsm}dx_{j-1}dx_{j+1}{\dotsm}dx_n,
\\
  \delta
& =
  \frac{\theta}{2}+\frac{n}{4}-1
  >
  \frac{n+1}{2},  
\end{align*}
where $A$ and $\nu$ are positive constans determined later, 
$$
N(l)
=
\sum_{m=0}^l\frac{(l+n-1)!}{l!(n-1)!},
$$
\begin{align*}
  k^l(t,x,\xi) 
& =
  [I_{N(l)}k_1^l(t,x,\xi)]\oplus[I_{N(l)}k_1^l(t,x,\xi)^{-1}],
\\
  k^l_{\text{inv}}(t,x,\xi)
& =
  k^l(t,x,\xi)^{-1},
\end{align*}
$$
K^l(t)=k^l(t,x,D),
\quad
K^l_{\text{inv}}(t)=k^l_{\text{inv}}(t,x,D),
$$
$$
\tilde{\Lambda}^l(t)
=
\frac{1}{2}
\sum_{j=1}^n
E_{2N(l)}
\Bigl(
B_j^l(t,x)-B_j^{l,\text{diag}}(t,x)
\Bigr)
\p_j(\nu^2-\Delta)^{-1},
$$
$$
\Lambda^l(t)
=
I_{2N(l)}-i\tilde{\Lambda}^l(t),
\quad
\Lambda^l_{\text{inv}}(t)
=
I_{2N(l)}+i\tilde{\Lambda}^l(t).
$$
\par
First we determine $A$ and $\nu$. 
On one hand, 
in the same way as the proof of Lemma~\ref{theorem:coefficients}, 
we have 
$$
\sum_{j=1}^n
\lvert{B_j^l(t,x)}\rvert
\leqslant
C_{\theta,n}^2
\phi^l(t,x_m)
\sum_{p=1}^\infty
A_p
\Bigl(C_{\theta,n}X_{\theta-2,s,r}^{l-1}(t)\Bigr)^{2p-1}
$$
for $(t,x)\in[-T,T]\times\mathbb{R}^n$ and $m=1,\dotsc,n$. 
Hence, if $X_{\theta-2,s,r}^l(t)\leqslant4M$, then 
there exists a positive constant 
A depending only on $M$, $\theta$, $n$ and $\{A_p\}_{p=1,2,3\dotsc}$ 
such that 
$\lvert{B_j^l(t,x)}\rvert\leqslant{A}\phi^l(t,x_m)$ 
for $(t,x)\in[-T,T]\times\mathbb{R}^n$ and $m=1,\dotsc,n$. 
On the other hand, it is easy to see that 
\begin{align*}
  K^l(t)\Lambda^l(t)\Lambda^l_{\text{inv}}(t)K^l_{\text{inv}}(t) 
& =
  I_{N(l)}+R^l_1(t),
\\
  \Lambda^l_{\text{inv}}(t)K^l_{\text{inv}}(t)K^l(t)\Lambda^l(t) 
& =
  I_{N(l)}+R^l_2(t).
\end{align*}
We here remark that 
$R^l_1(t)$ and $R^l_2(t)$ are pseudodifferential operators of order $-1$ 
and 
$$
\lVert{R^l_1(t))}\rVert,
\lVert{R^l_2(t)}\rVert
=
\mathcal{O}(\nu^{-1}). 
$$
Nagase's theorem shows that 
if $X_{\theta-1,s,r}^l(t)\leqslant2M$, 
then there exist $\nu_0\geqslant1$ and $C_M>0$ which are independent of $l$, 
such that 
$K^l(t)\Lambda^l(t)$ is invertible, and 
\begin{equation}
\lVert{K^l(t)\Lambda^l(t)}\rVert, 
\lVert{(K^l(t)\Lambda^l(t))^{-1}}\rVert
\leqslant
C_M
\label{equation:normestimates}
\end{equation}
for $\nu\geqslant\nu_0$. 
Set $\nu=\nu_0$ below. 
\par
Now we begin the proof of \eqref{equation:convergenceatt} for some small $T>0$. 
Without loss of generality, we may assume that 
$$
\lVert{K^l(0)\Lambda^l(0)w^l(0)}\rVert
\leqslant
M
$$
for $l=0,1,2,\dotsc$. 
It suffices to consider only the forward direction in time. 
Let $T_l$ be a positive time defined by 
$$
T_l
=
\sup
\Bigl\{
T>0
\ \Big\vert \ 
X_{\theta-1,s,r}^l(t)^2
+
\lVert{K^l(t)\Lambda^l(t)w^l(t)}\rVert^2
\leqslant
4M^2 
\quad \text{for}\quad 
t\in[0,T]
\Bigr\}.
$$
We remark that \eqref{equation:normestimates} is valid for $t\in[0,T_l]$. 
Using the Schwarz inequality, we have 
\begin{align}
  \frac{d}{dt}
  X_{\theta-1,s,r}^l(t)^2 
& =
  2\re
  \sum_{\lvert\alpha\rvert\leqslant{l}}
  \left(
  \frac{r^{\lvert\alpha\rvert}}{\alpha_\ast!^s}
  \p_t\langle{D}\rangle^{\theta-1}J^\alpha{u},
  \frac{r^{\lvert\alpha\rvert}}{\alpha_\ast!^s}
  \langle{D}\rangle^{\theta-1}J^\alpha{u},
  \right)
\nonumber
\\
& =
  2\re
  \sum_{\lvert\alpha\rvert\leqslant{l}}
  \left(
  \frac{r^{\lvert\alpha\rvert}}{\alpha_\ast!^s}
  i\Delta\langle{D}\rangle^{\theta-1}J^\alpha{u},
  \frac{r^{\lvert\alpha\rvert}}{\alpha_\ast!^s}
  \langle{D}\rangle^{\theta-1}J^\alpha{u},
  \right)
\nonumber
\\
& +
  2\re
  \sum_{\lvert\alpha\rvert\leqslant{l}}
  \left(
  \frac{r^{\lvert\alpha\rvert}}{\alpha_\ast!^s}
  \langle{D}\rangle^{\theta-1}J^\alpha{f},
  \frac{r^{\lvert\alpha\rvert}}{\alpha_\ast!^s}
  \langle{D}\rangle^{\theta-1}J^\alpha{u},
  \right)
\nonumber
\\
& =
  2\re
  \sum_{\lvert\alpha\rvert\leqslant{l}}
  \left(
  \frac{r^{\lvert\alpha\rvert}}{\alpha_\ast!^s}
  \langle{D}\rangle^{\theta-1}J^\alpha{f},
  \frac{r^{\lvert\alpha\rvert}}{\alpha_\ast!^s}
  \langle{D}\rangle^{\theta-1}J^\alpha{u},
  \right)
\nonumber
\\
& \leqslant
  2
  \sum_{\lvert\alpha\rvert\leqslant{l}}
  \frac{r^{2\lvert\alpha\rvert}}{\alpha_\ast!^{2s}}
  \lVert\langle{D}^{\theta-1}\rangle^{\theta-1}J^\alpha{f}\rVert
  \lVert{J^\alpha}u\rVert_{\theta-1}
\nonumber
\\
& \leqslant
  2
  \left(
  \sum_{\lvert\alpha\rvert\leqslant{l}}
  \frac{r^{2\lvert\alpha\rvert}}{\alpha_\ast!^{2s}}
  \lVert\langle{D}\rangle^{\theta-1}J^\alpha{f(t)}\rVert^2
  \right)^{1/2}
  X_{\theta,s,r,}^l(t).
\label{equation:xl}
\end{align}
In the same way as Lemma~\ref{theorem:lower}, we can get 
\begin{align*}
& \left(
  \sum_{\lvert\alpha\rvert\leqslant{l}}
  \frac{r^{2\lvert\alpha\rvert}}{\alpha_\ast!^{2s}}
  \lVert\langle{D}\rangle^{\theta-1}J^\alpha{f(t)}\rVert^2
  \right)^{1/2} 
  X_{\theta,s,r,}^l(t)
\\
& \leqslant
  C_{\theta,n}
  \sum_{p=1}^\infty
  A_pC_{\theta,n}^{2p+1}
  \Bigl(X_{\theta,s,r}^l(t)\Bigr)^{2p+2}
\\
& \leqslant
  C_{\theta,n}
  \sum_{p=1}^\infty
  A_pC_{\theta,n}^{2p+1}C_M^{2p+2}
  \lVert{K^l(t)\Lambda^l(t)w^l(t)}\rVert^{2p+2}
\\
& \leqslant
  C_{\theta,n}
  \sum_{p=1}^\infty
  A_pC_{\theta,n}^{2p+1}C_M^{2p+2}M^2p
  \lVert{K^l(t)\Lambda^l(t)w^l(t)}\rVert^2.
\end{align*}
For $R=2C_{\theta,n}C_MM$, there exists a positive constant $C_R$ such that 
$$
\left(
\sum_{\lvert\alpha\rvert\leqslant{l}}
\frac{r^{2\lvert\alpha\rvert}}{\alpha_\ast!^{2s}}
\lVert\langle{D}\rangle^{\theta-1}J^\alpha{f(t)}\rVert^2
\right)^{1/2}
X_{\theta,s,r,}^l(t) 
\leqslant
C_{\theta,n}C_R
\lVert{K^l(t)\Lambda^l(t)w^l(t)}\rVert^2.
$$
Substituting this into \eqref{equation:xl}, we obtain 
\begin{equation}
\frac{d}{dt}
X_{\theta-1,s,r}^l(t)^2
\leqslant
2C_{\theta,n}C_R
\lVert{K^l(t)\Lambda^l(t)w^l(t)}\rVert^2.
\label{equation:energy1}
\end{equation}
\par
On the other hand, $w^l$ solves 
$$
\left(
I_{2N(l)}\p_t
-
i\Delta{E_{2N(l)}}
+
\sum_{j=1}^n
B_j^l(t,x)\p_j
\right)
w^l
=
g^l.
$$
By using Lemma~\ref{theorem:doi}, we have 
\begin{align*}
  \frac{d}{dt}
  \lVert{K^l(t)\Lambda^l(t)w^l(t)}\rVert^2
& \leqslant
  2\mathcal{C}^l(t)
  \lVert{K^l(t)\Lambda^l(t)w^l(t)}\rVert^2
\\
& +
  2
  \lVert{K^l(t)\Lambda^l(t)w^l(t)}\rVert
  \lVert{K^l(t)\Lambda^l(t)g^l(t)}\rVert,
\end{align*}
where $\mathcal{C}^l(t)$ depends only on 
$$
\int_{\mathbb{R}}
\phi^l(t,s)
ds,
\quad
\lVert\phi^l(t)\rVert_{\mathscr{B}^2},
\quad
\sup_{z\in\mathbb{R}}
\left\lvert
\int_{-\infty}^z
\p_t\phi^l(t,x)
ds
\right\rvert,
$$
$$
\sum_{j=1}^n
\Bigl(
\lVert{B_j^l(t)}\rVert_{\mathscr{B}^2}
+
\lVert{\p_tB_j^l(t)}\rVert_{\mathscr{B}^0}
\Bigr).
$$
It is easy to see that 
\begin{equation}
\int_{\mathbb{R}}
\phi^l(t,s)
ds,
\quad
\lVert\phi^l(t)\rVert_{\mathscr{B}^2}
\leqslant
X_{\theta,s,r}^l(t).
\label{equation:coef1}
\end{equation}
Using 
$$
\p_t\langle{D}\rangle^\delta{J^\alpha}u
=
i\Delta\langle{D}\rangle^\delta{J^\alpha}u
+
\langle{D}\rangle^\delta{J^\alpha}f,
$$
and the integration by parts, 
we deduce 
\begin{align}
  \sup_{z\in\mathbb{R}}
  \left\lvert
  \int_{-\infty}^z
  \p_t\phi^l(t,x)
  ds
  \right\rvert
& \leqslant
  C\{\theta,n\}X_{\theta,s,r}^l(t)
\nonumber
\\
& \qquad\times
  \left(
  1
  +
  \sum_{p=1}^\infty
  A_pC_{\theta,n}^{2p}
  \Bigl(X_{\theta,s,r}^l(t)\Bigr)^{2p}
  \right).
\label{equation:coef2}
\end{align}
\eqref{equation:coef1}, 
\eqref{equation:coef2} 
and 
Lemma~\ref{theorem:coefficients} 
show that 
there exists a positive constant $D_M$ 
which depends only on $M$, $\theta$ and $n$, and is independent of $l$, 
such that $\mathcal{C}^l(t)\leqslant{D_M}$ for $t\in[0,T_l]$. 
On the other hand, Lemma~\ref{theorem:lower} shows that 
\begin{align*}
  \lVert{K^l(t)\Lambda^l(t)g^l(t)}\rVert
& \leqslant
  C_M
  \left(
  \sum_{\lvert\alpha\rvert\leqslant{l}}
  \frac{r^{2\lvert\alpha\rvert}}{\alpha_\ast!^{2s}}
  \lVert{f_{\theta,\alpha}(t)}\rVert^2
  \right)^{1/2}
\\
& \leqslant
  C_MC_{\theta,n}
  \sum_{p=1}^\infty
  A_pC_{\theta,n}^{2p+1}
  \Bigl(X_{\theta,s,r}^l(t)\Bigr)^{2p+1}
\\
& \leqslant
  C_MC_{\theta,n}C_R
  \lVert{K^l(t)\Lambda^l(t)w^l(t)}\rVert.
\end{align*}
Hence, we have 
\begin{equation}
\frac{d}{dt}
\lVert{K^l(t)\Lambda^l(t)w^l(t)}\rVert^2
\leqslant
2(D_M+C_MC_{\theta,n}C_R)
\lVert{K^l(t)\Lambda^l(t)w^l(t)}\rVert^2.
\label{equation:energy2}
\end{equation}
Combining \eqref{equation:energy1} and \eqref{equation:energy2}, 
we obtain 
\begin{align}
& \frac{d}{dt}
  \Bigl\{
  X_{\theta-1,s,r}^l(t)^2
  +
  \lVert{K^l(t)\Lambda^l(t)w^l(t)}\rVert^2
  \Bigr\}
\nonumber
\\
  \leqslant
& 2C_1
  \Bigl\{
  X_{\theta-1,s,r}^l(t)^2
  +
  \lVert{K^l(t)\Lambda^l(t)w^l(t)}\rVert^2
  \Bigr\},
\label{equation:energy3}
\end{align}
where $C_1$ depends only on $M$, $\theta$ and $n$, 
and is independent of $l$. 
Integrating \eqref{equation:energy3} over $[0,T_l]$, 
we obtain $4M^2\leqslant2M^2\exp(2C_1T_l)$, 
which implies that $T_l\geqslant\log2/2C_1>0$. 
Set $T^\ast=\log2/2C_1$ for short. 
For all $l$, we obtain 
$$
X_{\theta-1,s,r}^l(t), 
\lVert{K^l(t)\Lambda^l(t)w^l(t)}\rVert 
\leqslant
2M
$$
for $t\in[0,T^\ast]$. 
Hence, \eqref{equation:normestimates} shows that 
\begin{align*}
  X_{\theta,s,r}^l(t)
& =
  \frac{1}{2}
  \lVert{w^l(t)}\rVert
\\
& =
  \frac{1}{2}
  \lVert{(K^l(t)\Lambda^l(t))^{-1}K^l(t)\Lambda^l(t)w^l(t)}\rVert
\\
& \leqslant
  \frac{C_M}{2}
  \lVert{K^l(t)\Lambda^l(t)w^l(t)}\rVert
\\
& \leqslant
  C_MM
\end{align*}
for $t\in[0,T^\ast]$. 
Thus we obtain 
$X_{\theta,s,r,}^\infty(t)\leqslant{C_MM}$
for $t\in[0,T^\ast]$. 
This completes the proof of the uniform energy estimates. 

\section{Gevrey estimates of solutions}
\label{section:proof}
In this section we complete the proof of Theorem~\ref{theorem:main}. 
For $k\in\mathbb{N}\cup\{0\}$ and a multi-index
$\alpha=(\alpha_1,\dotsc,\alpha_1)$, 
set $k_+=\max\{k,1\}$ and
$\alpha_+=((\alpha_1)_+,\dotsc,(\alpha_n)_+)$. 
In the previous section, we have proved that 
$$
\sum_\alpha
\frac{r^{2\lvert\alpha\rvert}}{\alpha_\ast!^{2s}}
\lVert{J^\alpha}u(t)\rVert_\theta^2
\leqslant
M^2
$$
for $t\in[-T,T]$ with some positive constants $r$, $M$ and $s\geqslant1$. 
The Schwartz inequality shows that 
\begin{align}
  \sum_\alpha
  \frac{r^{\lvert\alpha\rvert}}{\alpha!^s}
  \lVert{J^\alpha}u(t)\rVert_\theta
& =
  \sum_\alpha
  \frac{1}{\alpha_+^s}
  \frac{r^{\lvert\alpha\rvert}}{\alpha_\ast!^s}
  \lVert{J^\alpha}u(t)\rVert_\theta
\nonumber
\\
& \leqslant
  \left(
  \sum_\alpha
  \frac{1}{\alpha_+^{2s}}
  \right)^{1/2}
  \left(
  \sum_\alpha
  \frac{r^{2\lvert\alpha\rvert}}{\alpha_\ast!^{2s}}
  \lVert{J^\alpha}u(t)\rVert_\theta^2
  \right)^{1/2}
\nonumber
\\
& \leqslant
  a^nM
\label{equation:gev71}
\end{align}
for $t\in[-T,T]$. 
In order to obtain \eqref{equation:gevrey} from \eqref{equation:gev71}, 
we need two lemmas. 
\begin{lemma}
\label{theorem:hermite} 
Suppose that $s\geqslant1/2$ and 
$$
\sum_\alpha
\frac{r^{\lvert\alpha\rvert}}{\alpha!^s}
\lVert{J^\alpha}u(t)\rVert_\theta
\leqslant
M_0
$$
for $t\in[-T,T]\setminus\{0\}$. 
Then there here exist positive constants $\rho$ and $M_1$ such that 
$$
\sum_\alpha
\frac{1}{\lvert\alpha\rvert!^s}
\left(\frac{\lvert{t}\rvert}{\rho}\right)^{\lvert\alpha\rvert}
\lVert\langle{x}\rangle^{-\lvert\alpha\rvert}\p^\alpha{u(t)}\rVert_\theta
\leqslant
M_1
$$
for $t\in[-T,T]\setminus\{0\}$. 
\end{lemma}
\begin{lemma}
\label{theorem:weight}
For any smooth function $u$ of $(t,x)$, 
\begin{align}
& \lVert
  \langle{x}\rangle^{-\lvert\alpha\rvert-2m}
  \p_t^m\p^{\alpha+e_j}
  u(t)
  \rVert_{\theta-1}
\nonumber
\\
& \qquad
  \leqslant
  C_0(\lvert\alpha\rvert+2m)_+
  \lVert
  \langle{x}\rangle^{-\lvert\alpha\rvert-2m}
  \p_t^m\p^\alpha
  u(t)
  \rVert_{\theta},
\label{equation:weight1} 
\\
& \lVert
  \langle{x}\rangle^{-2}\p_j
  \{
  \langle{x}\rangle^{-\lvert\alpha\rvert-2m}
  \p_t^m\p^{\alpha+e_k}
  u(t)
  \}
  \rVert_{\theta-1}
\nonumber
\\
& \qquad
  \leqslant
  \lVert
  \langle{x}\rangle^{-\lvert\alpha+e_j+e_k\rvert-2m}
  \p_t^m\p^{\alpha+e_j+e_k}
  u(t)
  \rVert_{\theta-1}
\nonumber
\\
& \qquad
  +
  C_0(\lvert\alpha\rvert+2m)_+
  \lVert
  \langle{x}\rangle^{-\lvert\alpha+e_k\rvert-2m}
  \p_t^m\p^{\alpha+e_k}
  u(t)
  \rVert_{\theta-1},
\label{equation:weight2}
\end{align}
where $C_0>0$ is independent of $\alpha$ and $m$. 
\end{lemma}
\begin{proof}[Proof of Lemma~\ref{theorem:hermite}] 
Recall the explicit formula of the hermitian polynomial 
$$
e^{-a\tau^2}
\left(\frac{d}{d\tau}\right)^\mu
e^{a\tau^2}
=
\sum_{\nu\leqslant\mu/2}
\frac{\mu!}{\nu!(\mu-2\nu)!}
a^{\mu-\nu}
(2\tau)^{\mu-2\nu}
$$
for $a,\tau\in\mathbb{R}$ and $\mu\in\mathbb{N}$. 
Applying this to the definition of the operator $J$, we deduce 
\begin{align*}
  \p^\alpha{u} 
& =
  \p^\alpha
  \Bigl\{
  e^{i\lvert{x}\rvert^2/4t}
  (e^{-i\lvert{x}\rvert^2/4t}u)
  \Bigr\}
\\
& =
  \sum_{\beta\leqslant\alpha}
  \frac{\alpha!}{\beta!(\alpha-\beta)!}
  \Bigl(
  e^{-i\lvert{x}\rvert^2/4t}
  \p^\beta
  e^{i\lvert{x}\rvert^2/4t}
  \Bigr)
  \Bigl\{
  e^{i\lvert{x}\rvert^2/4t}
  \p^{\alpha-\beta}
  \Bigl(e^{-i\lvert{x}\rvert^2/4t}u)\Bigr)
  \Bigr\}
\\
& =
  \sum_{\beta\leqslant\alpha}
  \frac{\alpha!}{\beta!(\alpha-\beta)!}
  \left(\frac{1}{2it}\right)^{\lvert\alpha-\beta\rvert}
  J^{\alpha-\beta}u
\\
& \qquad
  \times
  \sum_{\gamma\leqslant\beta/2}
  \frac{\beta!}{\gamma!(\beta-2\gamma)!}
  \left(\frac{i}{4t}\right)^{\lvert\beta-\gamma\rvert}
  (2x)^{\beta-2\gamma}. 
\end{align*}
Set $\Theta=[\theta]+1$ and 
$\rho_0=\max\{4,2(1+T),1/r^2\}$.  
Here we remark that $1-s\leqslant1/2$ since $s\geqslant1/2$. 
We deduce
\begin{align}
& \frac{1}{(\lvert\alpha\rvert+2\Theta)!^s}
  \left(\frac{\lvert{t}\rvert}{\rho_0}\right)^{\lvert\alpha\rvert}
  \lVert\langle{x}\rangle^{-\lvert\alpha\rvert}\p^\alpha{u}\rVert_\theta
\nonumber
\\
& \leqslant
  \frac{1}{\alpha!^s\lvert\alpha\rvert^{2\Theta}\rho_0^{\lvert\alpha\rvert}}
  2^{-\lvert\alpha\rvert}
  \sum_{\beta\leqslant\alpha}
  \sum_{\gamma\leqslant\beta/2}
\nonumber
\\
& \quad
  \times
  \lvert{t}\rvert^{\lvert\gamma\rvert}
  \frac{\alpha!}{\beta!(\alpha-\beta)!}
  \frac{\beta!}{\gamma!(\beta-2\gamma)!}
  \left\lVert
  \frac{x^{\beta-2\gamma}}{\langle{x}\rangle^{\lvert\alpha\rvert}}
  J^{\alpha-\beta}u
  \right\rVert_\theta
\nonumber
\\
& = 
  \frac{2^{-\lvert\alpha\rvert/2}}{\lvert\alpha\rvert^{2\Theta}\rho_0^{\lvert\alpha\rvert}}
  \sum_{\beta\leqslant\alpha}
  \sum_{\gamma\leqslant\beta/2}
\nonumber
\\
& \quad
  \times
  \lvert{t}\rvert^{\lvert\gamma\rvert}
  \frac{\alpha!^{1-s}}{\beta!^{1-s}(\alpha-\beta)!^{1-s}}
  \frac{\beta!^{1-s}}{\gamma!(\beta-2\gamma)!}
  \frac{1}{(\alpha-\beta)!^s}
  \left\lVert
  \frac{x^{\beta-2\gamma}}{\langle{x}\rangle^{\lvert\alpha\rvert}}
  J^{\alpha-\beta}u
  \right\rVert_\theta
\nonumber
\\
& \leqslant
  \frac{2^{-\lvert\alpha\rvert/2}}{\lvert\alpha\rvert^{2\Theta}\rho_0^{\lvert\alpha\rvert}}
  \sum_{\beta\leqslant\alpha}
  \sum_{\gamma\leqslant\beta/2}
\nonumber
\\
& \quad
  \times
  (1+T)^{\lvert\beta\rvert/2}
  \frac{\beta!^{1/2}}{\gamma!(\beta-2\gamma)!}
  \frac{1}{(\alpha-\beta)!^s}
  \left\lVert
  \frac{x^{\beta-2\gamma}}{\langle{x}\rangle^{\lvert\alpha\rvert}}
  J^{\alpha-\beta}u
  \right\rVert_\theta.
\label{equation:71}
\end{align}
We remark that there exists a constant $C_0>0$ which is independent of 
$\alpha$, $\beta$ and $\gamma$, such that 
\begin{equation}
\lVert
x^{\beta-2\gamma}
\langle{x}\rangle^{-\lvert\alpha\rvert}
v
\rVert_\theta
\leqslant
C_0\lvert\alpha\rvert^{2\Theta}
\lVert{v}\rVert_\theta. 
\label{equation:73}
\end{equation}
On the other hand, 
$\beta!^{1/2}\leqslant2^{\lvert\beta\rvert/2}(\beta-[\beta/2])!$ 
since
$$
\frac{\beta!}{(\beta-[\beta/2])!^2}
\leqslant
\frac{[\beta/2]!}{[\beta/2]!(\beta-[\beta])!}
\leqslant
2^{\lvert\beta\rvert}.
$$
Hence, we have 
\begin{align}
  \sum_{\gamma\leqslant\beta/2}
  \frac{\beta!^{1/2}}{\gamma!(\beta-2\gamma)!}   
& \leqslant
  2^{\lvert\beta\rvert/2}
  \sum_{\gamma\leqslant\beta/2}
  \frac{(\beta-[\beta/2])!}{\gamma!(\beta-2\gamma)!}   
\nonumber
\\
& \leqslant
  2^{\lvert\beta\rvert/2}
  \sum_{\gamma\leqslant\beta/2}
  \frac{(\beta-[\beta/2])!}{\gamma!(\beta-[\beta/2]-\gamma)!}   
\nonumber
\\
& \leqslant
  2^{\lvert\beta\rvert/2}
  \sum_{\gamma\leqslant\beta-[\beta/2]}
  \frac{(\beta-[\beta/2])!}{\gamma!(\beta-[\beta/2]-\gamma)!}   
\nonumber
\\
& \leqslant
  2^{3\lvert\beta\rvert/2-[\beta/2]}
  \leqslant
  2^{\lvert\beta\rvert+n}
\label{equation:72}
\end{align}
Using 
\eqref{equation:gev71}, 
\eqref{equation:71}, 
\eqref{equation:73} 
and 
\eqref{equation:72}, 
we deduce 
\begin{align*}
& \frac{1}{(\lvert\alpha\rvert+2\Theta)!^s}
  \left(\frac{\lvert{t}\rvert}{\rho_0}\right)^{\lvert\alpha\rvert}
  \lVert\langle{x}\rangle^{-\lvert\alpha\rvert}\p^\alpha{u}\rVert_\theta
\\
& \leqslant
  \frac{C_02^{-\lvert\alpha\rvert/2}}{\rho_0^{\lvert\alpha\rvert}}
  \sum_{\beta\leqslant\alpha}
  (1+T)^{\lvert\beta\rvert/2}
  \frac{1}{(\alpha-\beta)!^s}
  \lVert{J^{\alpha-\beta}u}\rVert_\theta
  \sum_{\gamma\leqslant\beta/2}
  \frac{\beta!^{1/2}}{\gamma!(\beta-2\gamma)!}
\\
& \leqslant
  \frac{2^nC_0}{\rho_0^{\lvert\alpha\rvert}}
  \sum_{\beta\leqslant\alpha}
  \{2(1+T)\}^{\lvert\beta\rvert/2}
  \frac{1}{(\alpha-\beta)!^s}
  \lVert{J^{\alpha-\beta}u}\rVert_\theta
\\
& \leqslant
  2^nC_0\rho_0^{-\lvert\alpha\rvert/2}
  \sum_{\beta\leqslant\alpha}
  \left(
  \frac{2(1+T)}{\rho_0}
  \right)^{\lvert\beta\rvert/2}
  (\rho_0^{1/2}r)^{-\lvert\alpha-\beta\rvert}
  \frac{r^{\lvert\alpha-\beta\rvert}}{(\alpha-\beta)!^s}
  \lVert{J^{\alpha-\beta}u}\rVert_\theta
\\
& \leqslant
  2^nC_0\rho_0^{-\lvert\alpha\rvert/2}
  \sum_{\beta}
  \frac{r^{\lvert\beta\rvert}}{\beta!^s}
  \lVert{J^{\beta}u}\rVert_\theta
\\
& \leqslant
  2^nM_0\rho_0^{-\lvert\alpha\rvert/2}
  \leqslant
  2^{n-\lvert\alpha\rvert}C_0M_0.
\end{align*}
Thus
$$
\sum_\alpha
\frac{1}{(\lvert\alpha\rvert+2\Theta)!^s}
\left(\frac{\lvert{t}\rvert}{\rho_0}\right)^{\lvert\alpha\rvert}
\lVert\langle{x}\rangle^{-\lvert\alpha\rvert}\p^\alpha{u}\rVert_\theta
\leqslant
2^{2n}C_0M_0.
$$
If we set $\rho=2^s\rho_0$, we have 
\begin{align*}
& \sum_\alpha
  \frac{1}{\lvert\alpha\rvert!^s}
  \left(\frac{\lvert{t}\rvert}{\rho}\right)^{\lvert\alpha\rvert}
  \lVert\langle{x}\rangle^{-\lvert\alpha\rvert}\p^\alpha{u}\rVert_\theta
\\
& \leqslant
  \frac{2^{2s\Theta}}{(2\Theta)!^s}
  \sum_\alpha
  \frac{1}{(\lvert\alpha\rvert+2\Theta)!^s}
  \left(\frac{2^s\lvert{t}\rvert}{\rho}\right)^{\lvert\alpha\rvert}
  \lVert\langle{x}\rangle^{-\lvert\alpha\rvert}\p^\alpha{u}\rVert_\theta
\\
& =
  \frac{2^{2s\Theta}}{(2\Theta)!^s}
  \sum_\alpha
  \frac{1}{(\lvert\alpha\rvert+2\Theta)!^s}
  \left(\frac{\lvert{t}\rvert}{\rho_0}\right)^{\lvert\alpha\rvert}
  \lVert\langle{x}\rangle^{-\lvert\alpha\rvert}\p^\alpha{u}\rVert_\theta
\\
& \leqslant
  \frac{2^{2n+2s\Theta}C_0a^nM}{(2\Theta)!^s}.
\end{align*}
This completes the proof. 
\end{proof}
\begin{proof}[Proof of Lemma~\ref{theorem:weight}]
When $\lvert\alpha\rvert+2m=0$,  
\eqref{equation:weight1} and \eqref{equation:weight2} are obvious. 
When $\lvert\alpha\rvert+2m\ne0$, 
\eqref{equation:weight1} follows from
\begin{align*}
  \langle{x}\rangle^{-\lvert\alpha\rvert-2m}
  \p_t^m\p^{\alpha+e_j}u 
& =
  \p_j
  \Bigl(
  \langle{x}\rangle^{-\lvert\alpha\rvert-2m}
  \p_t^m\p^{\alpha}u 
  \Bigr)
\\
& -
  (\lvert\alpha\rvert+2m)
  \frac{2x_j}{\langle{x}\rangle^2}
  \langle{x}\rangle^{-\lvert\alpha\rvert+2m}
  \p_t^m\p^\alpha{u}, 
\end{align*}
and \eqref{equation:weight2} follows from 
\begin{align*}
  \langle{x}\rangle^{-2}\p_j
  \{\langle{x}\rangle^{-\lvert\alpha\rvert-2m}\p_t^m\p^{\alpha+e_k}u\}   
& =
  \langle{x}\rangle^{-\lvert\alpha\rvert-2m-2}\p_t^m\p^{\alpha+e_j+e_k}u
\\
& -(\lvert\alpha\rvert+2m)
  \frac{2x_j}{\langle{x}\rangle^3}
  \langle{x}\rangle^{-\lvert\alpha\rvert-2m-1}\p_t^m\p^{\alpha+e_k}u.
\end{align*}
This completes the proof.  
\end{proof}
Now we shall complete the proof of Theorem~\ref{theorem:main}. 
\begin{proof}[Proof of Theorem~\ref{theorem:main}] 
We shall obtain \eqref{equation:gevrey} 
from \eqref{equation:pde1} and \eqref{equation:gev71} 
by an induction argument on the order of the time derivatives. 
Suppose $s\geqslant1$ 
Set 
$$
Y^l(t)
=
\sum_{m=0}^l
\sum_\alpha
\frac{\lvert{t}\rvert^{\lvert\alpha\rvert+2m}\rho^{-\lvert\alpha\rvert}\kappa^{-2m}}{(\lvert\alpha\rvert+2m-2)_+!^s}
\lVert
\langle{x}\rangle^{-\lvert\alpha\rvert-2m}\p_t^m\p^\alpha{u(t)}
\rVert_\theta,
$$
with some $\kappa>0$ determined later. 
Lemma~\ref{theorem:hermite} shows that 
there exists a positive constant $M$ such that 
$Y^0(t)\leqslant{M/2}$ for $t\in[-T,T]\setminus\{0\}$. 
Suppose that $Y^l(t)\leqslant{M}$ for $t\ne0$. 
Since $u$ is a solution to \eqref{equation:pde1}, we deduce 
\begin{align}
  Y^{l+1}(t) 
& =
  Y^0(t)
  +
  \sum_{m=1}^{l+1}
  \sum_\alpha
  \frac{\lvert{t}\rvert^{\lvert\alpha\rvert+2m}\rho^{-\lvert\alpha\rvert}\kappa^{-2m}}{(\lvert\alpha\rvert+2m-2)!^s}
  \lVert
  \langle{x}\rangle^{-\lvert\alpha\rvert-2m}\p_t^m\p^\alpha{u(t)}
  \rVert_\theta
\nonumber
\\
& =
  Y^0(t)
\nonumber
\\
& +
  \sum_{m=0}^l
  \sum_\alpha
  \frac{\lvert{t}\rvert^{\lvert\alpha\rvert+2m+2}\rho^{-\lvert\alpha\rvert}\kappa^{-2m-2}}{(\lvert\alpha\rvert+2m)!^s}
  \lVert
  \langle{x}\rangle^{-\lvert\alpha\rvert-2m-2}\p_t^m\p^\alpha\p_tu(t)
  \rVert_\theta  
\nonumber
\\
& \leqslant
  Y^0(t)
\nonumber
\\
& +
  \sum_{j=1}^n
  \sum_{m=0}^l
  \sum_\alpha
  \frac{\lvert{t}\rvert^{\lvert\alpha\rvert+2m+2}\rho^{-\lvert\alpha\rvert}\kappa^{-2m-2}}{(\lvert\alpha\rvert+2m)!^s}
  \lVert
  \langle{x}\rangle^{-\lvert\alpha\rvert-2m-2}\p_t^m\p^{\alpha+2e_j}u(t)
  \rVert_\theta
\nonumber
\\
& +
  \sum_{m=0}^l
  \sum_\alpha
  \frac{\lvert{t}\rvert^{\lvert\alpha\rvert+2m+2}\rho^{-\lvert\alpha\rvert}\kappa^{-2m-2}}{(\lvert\alpha\rvert+2m)!^s}
  \lVert
  \langle{x}\rangle^{-\lvert\alpha\rvert-2m-2}\p_t^m\p^{\alpha}f(u,\p{u})(t)
  \rVert_\theta
\nonumber
\\
& \leqslant
  Y^0(t)
  +
  \frac{\rho^2n}{\kappa^2}Y^l(t)
  +
  Z^l(t)
\nonumber
\\
& \leqslant
  \left(
  \frac{1}{2}
  +
  \frac{\rho^2n}{\kappa^2}
  \right)M
  +
  Z^l(t)
\label{equation:nanako}
\end{align}
where 
$$
Z^l(t)
=
\sum_{m=0}^l
\sum_\alpha
\frac{\lvert{t}\rvert^{\lvert\alpha\rvert+2m+2}\rho^{-\lvert\alpha\rvert}\kappa^{-2m-2}}{(\lvert\alpha\rvert+2m)!^s}
\lVert
\langle{x}\rangle^{-\lvert\alpha\rvert-2m-2}\p_t^m\p^{\alpha}f(u,\p{u})(t)
\rVert_\theta.
$$
The chain rule shows that 
\begin{align*}
  \p_t^m\p^\alpha{f(u,\p{u})}
& =
  \sum_{p=1}^\infty
  \sum_{\substack{\lvert\gamma\rvert=p+1 \\ \lvert\bar{\gamma}\rvert=p}}
  f_{\gamma\bar{\gamma}}
  \sum_{\substack{m_0+\dotsb+m_{2p}=m \\
 \alpha^0+\dotsb+\alpha^{2p}=\alpha}}
\\
& \qquad
  \times
  \frac{\alpha!}{\alpha^0!\dotsm\alpha^{2p}!}
  \frac{m!}{m_0!\dotsm{m_{2p}!}}
  \prod_{q=0}^{2p}
  \p_t^{m_q}\p^{\alpha^q}\p_{q,\gamma\bar{\gamma}}\tilde{u},
\end{align*}
where $\tilde{u}=u$ or $\bar{u}$. 
Set $\p_0=1$ for short. 
Using this and 
$[\p_j,\langle{x}\rangle^{-2}]=\mathcal{O}(\langle{x}\rangle^{-2})$ 
($j=0,1,\dotsc,n$), we have 
\begin{align}
  Z^l(t)
& \leqslant
  \sum_{m=0}^l
  \sum_\alpha
  \frac{\lvert{t}\rvert^{\lvert\alpha\rvert+2m+2}\rho^{-\lvert\alpha\rvert}\kappa^{-2m-2}}{(\lvert\alpha\rvert+2m)!^s}
\nonumber  
\\ 
& \qquad
  \times
  \sum_{p=1}^\infty
  \sum_{\substack{\lvert\gamma\rvert=p+1 \\ \lvert\bar{\gamma}\rvert=p}}
  \lvert{f_{\gamma\bar{\gamma}}}\rvert
  \sum_{\substack{m_0+\dotsb+m_{2p}=m \\ \alpha^0+\dotsb+\alpha^{2p}=\alpha}}
  \frac{\alpha!}{\alpha^0!\dotsm\alpha^{2p}!}
  \frac{m!}{m_0!\dotsm{m_{2p}!}}
\nonumber
\\
& \qquad\qquad
  \times
  \sum_{j=0}^n
  \left\lVert
  \langle{x}\rangle^{-2}
  \p_j
  \left\{
  \prod_{q=0}^{2p}
  \langle{x}\rangle^{-\lvert\alpha^{q}\rvert-2m_q}
  \p_t^{m_q}\p^{\alpha^q}\p_{q,\gamma\bar{\gamma}}\tilde{u}
  \right\}
  \right\rVert_{\theta-1}.
\label{equation:76}
\end{align}
By using Lemmas~\ref{theorem:leibniz} and \ref{theorem:weight}, we deduce
\begin{align}
& \left\lVert
  \langle{x}\rangle^{-2}
  \p_0
  \left\{
  \prod_{q=0}^{2p}
  \langle{x}\rangle^{-\lvert\alpha^{q}\rvert-2m_q}
  \p_t^{m_q}\p^{\alpha^q}\p_{q,\gamma\bar{\gamma}}\tilde{u}
  \right\}
  \right\rVert_{\theta-1}
\nonumber
\\
& \leqslant
  \left\lVert
  \prod_{q=0}^{2p}
  \langle{x}\rangle^{-\lvert\alpha^{q}\rvert-2m_q}
  \p_t^{m_q}\p^{\alpha^q}\p_{q,\gamma\bar{\gamma}}\tilde{u}
  \right\rVert_{\theta-1}  
\nonumber
\\
& \leqslant
  (2p+1)C_0^{2p+1}
  \prod_{q=0}^{2p}
  \lVert
  \langle{x}\rangle^{-\lvert\alpha^{q}\rvert-2m_q}
  \p_t^{m_q}\p^{\alpha^q}\p_{q,\gamma\bar{\gamma}}\tilde{u}
  \rVert_{\theta-1}
\nonumber
\\
& \leqslant
  (2p+1)C_0^{4p+2}
  \prod_{q=0}^{2p}
  (\lvert\alpha^{q}\rvert+2m_q)_+
  \lVert
  \langle{x}\rangle^{-\lvert\alpha^{q}\rvert-2m_q}
  \p_t^{m_q}\p^{\alpha^q}u
  \rVert_{\theta}
\\
& =
  (2p+1)C_0^{4p+2}
  \lvert{t}\rvert^{-\lvert\alpha\rvert-2m}
  \rho^{\lvert\alpha\rvert}
  \kappa^{2m}
  \prod_{q=0}^{2p}
  (\lvert\alpha^{q}\rvert+2m_q)!^s
\nonumber
\\
& \qquad
  \times
  \prod_{q=0}^{2p}
  \frac{\lvert{t}\rvert^{\lvert\alpha^{q}\rvert+2m_q}\rho^{-\lvert\alpha^q\rvert}\kappa^{-2m_q}}{(\lvert\alpha^{q}\rvert+2m_q-2)_+!^s}
  \lVert
  \langle{x}\rangle^{-\lvert\alpha^{q}\rvert-2m_q}
  \p_t^{m_q}\p^{\alpha^q}u
  \rVert_{\theta},
\label{equation:74}
\end{align}
and for $j\ne0$ 
\begin{align}
& \left\lVert
  \langle{x}\rangle^{-2}
  \p_j
  \left\{
  \prod_{q=0}^{2p}
  \langle{x}\rangle^{-\lvert\alpha^{q}\rvert-2m_q}
  \p_t^{m_q}\p^{\alpha^q}\p_{q,\gamma\bar{\gamma}}\tilde{u}
  \right\}
  \right\rVert_{\theta-1}
\nonumber
\\
& \leqslant
  \sum_{r=0}^{2p}
  \Biggl\lVert
  \prod_{\substack{q=0 \\ q\ne{r}}}^{2p}
  \langle{x}\rangle^{-\lvert\alpha^{q}\rvert-2m_q}
  \p_t^{m_q}\p^{\alpha^q}\p_{q,\gamma\bar{\gamma}}\tilde{u}
\nonumber
\\
& \qquad
  \times
  \langle{x}\rangle^{-2}
  \p_j
  \{
  \langle{x}\rangle^{-\lvert\alpha^{r}\rvert-2m_r}
  \p_t^{m_r}\p^{\alpha^r}\p_{r,\gamma\bar{\gamma}}\tilde{u}
  \}  
  \Biggr\rVert_{\theta-1}
\nonumber
\\
& \leqslant
  (2p+1)C_0^{2p+1}
  \sum_{r=0}^{2p}
  \prod_{\substack{q=0 \\ q\ne{r}}}^{2p}
  \lVert
  \langle{x}\rangle^{-\lvert\alpha^{q}\rvert-2m_q}
  \p_t^{m_q}\p^{\alpha^q}\p_{q,\gamma\bar{\gamma}}\tilde{u}
  \rVert_{\theta-1}
\nonumber
\\
& \qquad
  \times
  \lVert
  \langle{x}\rangle^{-2}
  \p_j
  \{
  \langle{x}\rangle^{-\lvert\alpha^{r}\rvert-2m_r}
  \p_t^{m_r}\p^{\alpha^r}\p_{r,\gamma\bar{\gamma}}\tilde{u}
  \}  
  \rVert_{\theta-1}
\nonumber
\\
& \leqslant
  (2p+1)C_0^{4p+2}
  \sum_{r=0}^{2p}
  \prod_{\substack{q=0 \\ q\ne{r}}}^{2p}
  (\lvert\alpha^{q}\rvert+2m_q)_+
  \lVert
  \langle{x}\rangle^{-\lvert\alpha^{q}\rvert-2m_q}
  \p_t^{m_q}\p^{\alpha^q}u
  \rVert_{\theta}
\nonumber
\\
& \qquad
  \times
  \Bigl\{
  \lVert
  \langle{x}\rangle^{-\lvert\alpha^{r}+e_j+e_r^\prime\rvert-2m_r}
  \p_t^{m_r}\p^{\alpha^r+e_j+e_r^\prime}u
  \rVert_{\theta-1}
\nonumber
\\
& \qquad\qquad
  +
  (\lvert\alpha^r\rvert+2m_r)_+
  \lVert
  \langle{x}\rangle^{-\lvert\alpha^{r}+e_r^\prime\rvert-2m_r}
  \p_t^{m_r}\p^{\alpha^r+e_r^\prime}u
  \rVert_{\theta-1}
  \Bigr\}
\nonumber
\\
& =
  (2p+1)C_0^{4p+2}
  \lvert{t}\rvert^{-\lvert\alpha\rvert-2m-1}
  \rho^{\lvert\alpha\rvert+1}
  \kappa^{2m}
  \prod_{q=0}^{2p}
  (\lvert\alpha^{q}\rvert+2m_q)!^s
  \sum_{r=0}^{2p}
\nonumber
\\
& \qquad
  \times
  \prod_{\substack{q=0 \\ q\ne{r}}}^{2p}
  \frac{\lvert{t}\rvert^{\lvert\alpha^q\rvert}\rho^{-\lvert\alpha^q\rvert}\kappa^{-2m_q}}{(\lvert\alpha^{q}\rvert+2m_q-2)_+!^s}
  \lVert
  \langle{x}\rangle^{-\lvert\alpha^{q}\rvert-2m_q}
  \p_t^{m_q}\p^{\alpha^q}u
  \rVert_{\theta}
\nonumber
\\
& \qquad\qquad
  \frac{\lvert{t}\rvert^{\lvert\alpha^r+e_j\rvert}\rho^{-\lvert\alpha^r+e_j\rvert}\kappa^{-2m_r}}{(\lvert\alpha^r+e_j\rvert+2m_r-2)_+!^s}
  \lVert
  \langle{x}\rangle^{-\lvert\alpha^r+e_j\rvert-2m_r}
  \p_t^{m_r}\p^{\alpha^r+e^j}u
  \rVert_{\theta}
\nonumber
\\
& +
  (2p+1)C_0^{4p+2}
  \lvert{t}\rvert^{-\lvert\alpha\rvert-2m}
  \rho^{\lvert\alpha\rvert}
  \kappa^{2m}
  \prod_{q=0}^{2p}
  (\lvert\alpha^{q}\rvert+2m_q)!^s
\nonumber
\\
& \qquad
  \times
  \prod_{q=0}^{2p}
  \frac{\lvert{t}\rvert^{\lvert\alpha^q\rvert+2m_q}\rho^{-\lvert\alpha^q\rvert}\kappa^{-2m_q}}{(\lvert\alpha^{q}\rvert+2m_q-2)_+!^s}
  \lVert
  \langle{x}\rangle^{-\lvert\alpha^{q}\rvert-2m_q}
  \p_t^{m_q}\p^{\alpha^q}u
  \rVert_{\theta},
\label{equation:75}
\end{align}
where $e_r^\prime=e_j$ with some $j=0,1,\dotsc,n$. 
Substituting 
\eqref{equation:74} and \eqref{equation:75} into \eqref{equation:76}, 
we have 
\begin{align}
  Z^l(t)  
& \leqslant
  \frac{CT^2}{\kappa^2}
  \sum_{p=1}^\infty
  \sum_{\substack{\lvert\gamma\rvert=p+1 \\ \lvert\bar{\gamma}\rvert=p}}
  \lvert{f_{\gamma\bar{\gamma}}}\rvert
  (1+C_0^2)^{2p+1}(2p+1)^2
\nonumber
\\
& \times
  \sum_{m=0}^l
  \sum_\alpha
  \sum_{\substack{m_0+\dotsb+m_{2p}=m \\
  \alpha^0+\dotsb+\alpha^{2p}=\alpha}}
  \frac{\displaystyle\prod_{q=0}^{2p}(\lvert\alpha^q\rvert+2m_q)!^s}{(\lvert\alpha\rvert+2m)!^s}
  \frac{\alpha!}{\alpha^0!\dotsm\alpha^{2p}!}
  \frac{m!}{m_0!\dotsm{m_{2p}!}}
\nonumber
\\
& \times
  \prod_{q=0}^{2p}
  \frac{\lvert{t}\rvert^{\lvert\alpha^q\rvert+2m_q}\rho^{-\lvert\alpha^q\rvert}\kappa^{-2m_q}}{(\lvert\alpha^{q}\rvert+2m_q-2)_+!^s}
  \lVert
  \langle{x}\rangle^{-\lvert\alpha^{q}\rvert-2m_q}
  \p_t^{m_q}\p^{\alpha^q}u
  \rVert_{\theta}
\nonumber
\\
& +
  \sum_{j=1}^n\frac{CT\rho}{\kappa^2}
  \sum_{p=1}^\infty
  \sum_{\substack{\lvert\gamma\rvert=p+1 \\ \lvert\bar{\gamma}\rvert=p}}
  \lvert{f_{\gamma\bar{\gamma}}}\rvert
  (1+C_0^2)^{2p+1}(2p+1)
  \sum_{r=0}^{2p}
\nonumber
\\
& \times
  \sum_{m=0}^l
  \sum_\alpha
  \sum_{\substack{m_0+\dotsb+m_{2p}=m \\
  \alpha^0+\dotsb+\alpha^{2p}=\alpha}}
  \frac{\displaystyle\prod_{q=0}^{2p}(\lvert\alpha^q\rvert+2m_q)!^s}{(\lvert\alpha\rvert+2m)!^s}
  \frac{\alpha!}{\alpha^0!\dotsm\alpha^{2p}!}
  \frac{m!}{m_0!\dotsm{m_{2p}!}}
\nonumber
\\
& \times  
  \prod_{\substack{q=0 \\ q{\ne}r}}^{2p}
  \frac{\lvert{t}\rvert^{\lvert\alpha^q\rvert+2m_q}\rho^{-\lvert\alpha^q\rvert}\kappa^{-2m_q}}{(\lvert\alpha^{q}\rvert+2m_q-2)_+!^s}
  \lVert
  \langle{x}\rangle^{-\lvert\alpha^{q}\rvert-2m_q}
  \p_t^{m_q}\p^{\alpha^q}u
  \rVert_{\theta}
\nonumber
\\
& \times
  \frac{\lvert{t}\rvert^{\lvert\alpha^r+e_j\rvert+2m_r}\rho^{-\lvert\alpha^r+e_j\rvert}\kappa^{-2m_r}}{(\lvert\alpha^r\rvert+2m_r-2)_+!^s}
  \lVert
  \langle{x}\rangle^{-\lvert\alpha^r+e_j\rvert-2m_r}
  \p_t^{m_r}\p^{\alpha^r+e_j}u
  \rVert_{\theta}.
\label{equation:77}
\end{align}
In view of Lemma~\ref{theorem:factorial}, we have 
\begin{align*}
& \frac{\displaystyle\prod_{q=0}^{2p}(\lvert\alpha^q\rvert+2m_q)!^s}{(\lvert\alpha\rvert+2m)!^s}
  \frac{\alpha!}{\alpha^0!\dotsm\alpha^{2p}!}
  \frac{m!}{m_0!\dotsm{m_{2p}!}}
\\
& \leqslant
  \frac{\displaystyle\prod_{q=0}^{2p}(\lvert\alpha^q\rvert+2m_q)!^s}{(\lvert\alpha\rvert+2m)!^s}
  \frac{\alpha!}{\alpha^0!\dotsm\alpha^{2p}!}
  \frac{m!}{m_0!\dotsm{m_{2p}!}}
\\
& \leqslant
  \frac{\displaystyle\prod_{q=0}^{2p}(\lvert\alpha^q\rvert+2m_q)!^s}{(\lvert\alpha\rvert+2m)!^s}
  \frac{\alpha!}{\alpha^0!\dotsm\alpha^{2p}!}
  \frac{(2m)!}{(2m_0)!\dotsm{(2m_{2p})!}}
\\
& =
  \frac{\displaystyle\prod_{q=0}^{2p}\lvert(\alpha^q,2m_q)\rvert!^s}{\lvert(\alpha,2m)\rvert!^s}
  \frac{(\alpha,2m)!}{\displaystyle\prod_{q=0}^{2p}(\alpha^q,2m_q)!}
  \leqslant
  1.
\end{align*}
Applying this to \eqref{equation:77}, we deduce 
\begin{align}
  Z^l(t)
& \leqslant
  \frac{CT(T+\rho)}{\kappa^2}
  \sum_{p=1}^\infty
  A_p
  (1+C_0^2)^{2p+1}(2p+1)^2
  Y^l(t)^{2p+1}
\nonumber
\\
& \leqslant
  \frac{2CT(T+\rho)}{\kappa^2}
  \sum_{p=1}^\infty
  A_p
  (1+C_0^2)^{2p+1}(2p+1)^2
  \Bigl\{
  e(1+C_0)^2M
  \Bigr\}^{2p+1}.
\nonumber
\end{align}
Set $E=C_R$ with $R=2e(1+C_0)^2M$ for short. 
Then we have
\begin{equation}
Z^l(t)
\leqslant
\frac{2CT(T+\rho)E}{\kappa^2} 
\label{equation:78}
\end{equation}
for $t\in[-T,T]\setminus\{0\}$. 
Combining \eqref{equation:nanako} and \eqref{equation:78}, we have 
$$
Y^{l+1}(t)
\leqslant
\frac{M}{2}
+
\frac{\rho^2nM+2CT(T+\rho)E}{\kappa^2}
$$
for $t\in[-T,T]\setminus\{0\}$. 
If we choose $\kappa$ satisfying 
$$
\kappa
\geqslant
\sqrt{\dfrac{2\rho^2nM+4CT(T+\rho)E}{M}}, 
$$
then $Y^{l+1}(t)\leqslant{M}$ for $t\in[-T,T]\setminus\{0\}$. 
This completes the proof.   
\end{proof}
\section{Concluding remarks}
\label{section:remark}
Finally we state some remarks concerned with the results of \cite{hk}. 
We shall present some examples for which the Gevrey estimate  
\eqref{equation:gevrey} holds for $s\geqslant1/2$. 
First we remark that $e^{it\Delta}u_0$ gains analyticity in space-time
variables if $u_0$ decays faster than the Gaussian functions. 
\begin{theorem}
\label{theorem:free}
Let $s\geqslant1/2$ and $\theta\in\mathbb{R}$. 
Suppose that 
$\exp(\ep\langle{x}\rangle^{1/s})u_0{\in}H^\theta$ with some $\ep>0$. 
Then, For any $T>0$ there exist positive constants $M$ and $\rho$ such that 
for $t\in[-T,T]\setminus\{0\}$ 
$$
\lVert
\langle{x}\rangle^{-\lvert\alpha\rvert-2m}
\p_t^m\p^{\alpha}e^{it\Delta}u_0
\rVert_\theta
\leqslant
M\rho^{\lvert\alpha\rvert+2m}
t^{-\lvert\alpha\rvert-2m}
m!^{2s}\alpha!^s.
$$
\end{theorem}
\begin{proof}
Fix arbitrary $T>0$. 
Lemma~\ref{theorem:decay} shows that 
$$
\lVert{x^\alpha}u_0\rVert_\theta
\leqslant
M_0\rho_0^{\lvert\alpha\rvert}
\alpha!^s
$$
with some $M_0>0$ and $\rho_0>0$. 
Applying $J^\alpha$ to $(\p_t-i\Delta)e^{it\Delta}u_0=0$, we have 
$(\p_t-i\Delta)J^{\alpha}u=0$. 
It is easy to see that  
$$
\lVert{J^\alpha}e^{it\Delta}u_0\rVert_\theta
=
\lVert{x^\alpha}u_0\rVert_\theta
\leqslant
M_0\rho_0^{\lvert\alpha\rvert}
\alpha!^s. 
$$
Lemma~\ref{theorem:hermite} shows that for $t\in[-T,T]\setminus\{0\}$ 
$$
\lVert
\langle{x}\rangle^{-\lvert\alpha\rvert}
\p^{\alpha}e^{it\Delta}u_0
\rVert_\theta
\leqslant
M_1\rho_1^{\lvert\alpha\rvert}
\lvert{t}\rvert^{-\lvert\alpha\rvert}
\alpha!^s
$$
with some $M_1>0$ and $\rho_1>0$. 
Using the equation $(\p_t-i\Delta)e^{it\Delta}u_0=0$ again, we deduce 
\begin{align*}
& \lVert
  \langle{x}\rangle^{-\lvert\alpha\rvert-2m}
  \p_t^m\p^{\alpha}e^{it\Delta}u_0
  \rVert_\theta
\\
& =
  \lVert
  \langle{x}\rangle^{-\lvert\alpha\rvert-2m}
  \Delta^m\p^{\alpha}e^{it\Delta}u_0
  \rVert_\theta
\\
& \leqslant
  \sum_{j(1)=1}^n
  \dotsm
  \sum_{j(m)=1}^n
  \lVert
  \langle{x}\rangle^{-\lvert\alpha\rvert-2m}
  \p^{\alpha+2(e_{j(1)}+\dotsb+e_{j(m)})}e^{it\Delta}u_0
  \rVert_\theta
\\
& \leqslant
   n^mM_1\rho_1^{\lvert\alpha\rvert+2m}
   t^{-\lvert\alpha\rvert-2m}
   (\lvert\alpha\rvert+2m)!^s  
\end{align*}
for $t\in[-T,T]\setminus\{0\}$. 
This completes the proof. 
\end{proof}
Next we apply Theorem~\ref{theorem:free} to 
the initial value problem for 
one-dimensional nonlinear equations of the form 
\begin{alignat}{2}
  u_t-iu_{xx}
& =
  2a(\lvert{u}\rvert^2)_xu
  +
  ia^2\lvert{u}\rvert^4u
& 
  \quad\text{in}\quad
& \mathbb{R}^2,
\label{equation:special}
\\
  u(0,x)
& =
  u_0(x)
&
  \quad\text{in}\quad
& \mathbb{R},
\label{equation:speciald}
\end{alignat}
where $u_t=\p{u}/\p{t}$, $u_{xx}=\p^2{u}\p/{x^2}$, 
and $a$ is a real constant. 
The equation \eqref{equation:special} has very special nonlinearity. 
In fact, if $u$ is a smooth solutions to \eqref{equation:special}, 
then 
\begin{equation}
v(t,x)
=
\exp\left(-ia\int_{-\infty}^x\lvert{u(t,y)}\rvert^2dy\right)
u(t,x)
\label{equation:gauge0}
\end{equation}
formally solves the equation $v_t-iv_{xx}=0$. 
The mapping  
\begin{equation}
u 
\longmapsto 
v(x)
=
\exp\left(-ia\int_{-\infty}^x\lvert{u(y)}\rvert^2dy\right)
u(x)
\label{equation:gauge1} 
\end{equation}
defined for functions of $x$ is said to be a gauge transform. 
We remark that 
$\lvert{u(x)}\rvert=\lvert{v(x)}\rvert$ 
and 
$\lVert{u}\rVert=\lVert{v}\rVert$ for $u{\in}L^2(\mathbb{R})$, 
and the inverse of the gauge transform is given by 
\begin{equation}
u(x)
=
\exp\left(ia\int_{-\infty}^x\lvert{v(y)}\rvert^2dy\right)
v(x).  
\label{equation:gauge2}
\end{equation}
More properties of the gauge transform needed in this section are the following. 
\begin{lemma}
\label{theorem:gauge-transform}
Let $\theta>1/2$ and $s>0$. 
\begin{description}
\item[{\rm (i)}\ ] 
The gauge transform is a homeomorphic mapping of $H^\theta(\mathbb{R})$ onto itself. 
\item[{\rm (ii)}\ ] 
If $u{\in}H^\theta(\mathbb{R})$ satisfies 
$\exp(\ep\langle{x}\rangle^{1/s})u{\in}H^\theta(\mathbb{R})$, 
the gauge transformation of $u$ has the same property. 
\item[{\rm (iii)}\ ] 
If $u{\in}C(\mathbb{R};H^1(\mathbb{R}))$ solves {\rm \eqref{equation:special}}, 
then $v(t,x)$ defined by \eqref{equation:gauge0} solves $v_t-iv_{xx}=0$. 
\item[{\rm (iv)}\ ] 
If the sequence $\{u_n\}_{n=1}^\infty{\subset}H^\theta(\mathbb{R})$ satisfies 
$$
u_n \longrightarrow u
\quad\text{in}\quad
H^\theta(\mathbb{R}) 
\quad\text{as}\quad
n \rightarrow \infty,
$$
then 
$$
(\lvert{u_n}\rvert^2)_xu_n \longrightarrow (\lvert{u}\rvert^2)_xu
\quad\text{in}\quad
\mathscr{D}^\prime(\mathbb{R}) 
\quad\text{as}\quad
n \rightarrow \infty,  
$$
where $\mathscr{D}^\prime(\mathbb{R})$ 
is the space of distributions on $\mathbb{R}$. 
\end{description}
\end{lemma}
\begin{proof}
Set 
$$
\phi(x)
=
\int_{-\infty}^x\lvert{u(y)}\rvert^2dy
$$
for short. 
First we show (i). 
Suppose that $u{\in}H^\theta(\mathbb{R})$ with some $\theta>1/2$. 
We can check 
$\phi{\in}\mathscr{B}^{\theta+1/2}(\mathbb{R})$ 
and 
$\phi^\prime{\in}H^\theta$ 
since 
$\phi^\prime(x)=\lvert{u(x)}\rvert^2$ 
and 
$H^\theta(\mathbb{R})$ is an algebra for $\theta>1/2$. 
For any integer $j=0,1,2,\dotsb,[\theta]$, the chain rule shows that 
\begin{align}
  \left(\frac{d}{dx}\right)^j(e^{-ia\phi}u)
& =
  \sum_{k=0}^j
  \frac{j!}{k!(j-k)!}
  \left(\frac{d}{dx}\right)^ke^{-ia\phi}
  \left(\frac{d}{dx}\right)^{j-k}u
\nonumber
\\
& =
  \sum_{k=0}^j
  \frac{j!}{k!(j-k)!}
  e^{-ia\phi}
  \left\{
  e^{ia\phi}
  \left(\frac{d}{dx}\right)^ke^{-ia\phi}
  \right\}
  \left(\frac{d}{dx}\right)^{j-k}u.
\label{equation:80}
\end{align}
In view of \eqref{equation:chain}, we deduce 
\begin{equation}
e^{-ia\phi}
\left\{
e^{ia\phi}
\left(\frac{d}{dx}\right)^ke^{-ia\phi}\right\}
\left(\frac{d}{dx}\right)^{j-k}u
\in
H^{\theta-j}(\mathbb{R}). 
\label{equation:81}
\end{equation}
Lemma~\ref{theorem:algebra} shows that 
$$
\langle{D}\rangle^{\theta-[\theta]}
\left(\frac{d}{dx}\right)^j
(e^{-ia\phi}u)
\in
H^{[\theta]-j}(\mathbb{R})
\quad\text{for}\quad
j=0,1,\dotsc,[\theta]. 
$$
This asserts that 
$u{\in}H^\theta(\mathbb{R}){\mapsto}e^{-ia\phi}u{\in}H^\theta(\mathbb{R})$ 
is continuous. 
In the same way, the inverse \eqref{equation:gauge2} is also continuous. 
Hence the gauge transform \eqref{equation:gauge1} 
is homeomorphic on $H^\theta(\mathbb{R})$. 
\par
Next we show (ii). 
Replacing $u$ by $\exp(\ep\langle{x}\rangle^{1/s})u$ in \eqref{equation:80}, 
we have 
\begin{align*}
& \left(\frac{d}{dx}\right)^j
  \Bigl(\exp(\ep\langle{x}\rangle^{1/s})e^{-ia\phi}u\Bigr)
\\
& =
  \sum_{k=0}^j
  \frac{j!}{k!(j-k)!}
  e^{-ia\phi} 
  \left\{
  e^{ia\phi}
  \left(\frac{d}{dx}\right)^ke^{-ia\phi}
  \right\}
  \left(\frac{d}{dx}\right)^{j-k}\exp(\ep\langle{x}\rangle^{1/s})u.
\end{align*}
In the same way, we can check 
$\exp(\ep\langle{x}\rangle^{1/s})e^{-ia\phi}u{\in}H^\theta(\mathbb{R})$. 
\par
Next we show (iii). 
Suppose that $u{\in}C(\mathbb{R};H^1(\mathbb{R}))$ solves \eqref{equation:special}.  
It follows that 
$\p_tu{\in}C(\mathbb{R};H^{-1}(\mathbb{R}))$, 
$\phi{\in}C(\mathbb{R};\mathscr{B}^{3/2}(\mathbb{R}))$, 
$\phi_x{\in}C(\mathbb{R};H^1(\mathbb{R}))$ 
and 
$v=e^{-ia\phi}u$ 
also belongs to $C(\mathbb{R};H^1(\mathbb{R}))$. 
Thus, the following computations 
\begin{align*}
  e^{-ia\phi}u_t
& =
  v_t+iav\phi_t
\\
& =
  v_t
  +
  iav
  \int_{-\infty}^x
  (u_t\bar{u}+u\bar{u}_t
   +
   4a(\lvert{u}\rvert^2)_x\lvert{u}\rvert^2
   +
   ia^2\lvert{u}\rvert^6
   -
   ia^2\lvert{u}\rvert^r)
  dy
\\
& =
  v_t
  +
  iav
  \int_{-\infty}^x
  (iu_{xx}\bar{u}-iu\bar{u}_{xx}
   +2a(\lvert{u}\rvert^4)_x)
  dy
\\
& =
  v_t
  -
  a(u_x\bar{u}-u\bar{u}_x)v
  +
  2ia^2\lvert{u}\rvert^4v, 
\\
  v_x
& =
  -ia\lvert{u}\rvert^2v
  +
  e^{-ia\phi}u_x,
\\
  -ie^{-ia\phi}u_{xx}
& =
  -iv_{xx}
  +
  2a\lvert{u}\rvert^2v_x
  +
  a(\lvert{u}\rvert^2)_x
  +
  ia^2\lvert{u}\rvert^4v
\\
& =
  -iv_{xx}
  +
  a(\lvert{u}\rvert^2)_x
  +
  2au_x\bar{u}v
  -
  ia\lvert{u}\rvert^4v,   
\end{align*}
are justified, and it is easy to see that $v$ solves $v_t-iv_{xx}=0$. 
\par
Lastly, we check (iv). 
Fix arbitrary $\psi\in\mathscr{D}(\mathbb{R})$. 
Since $\theta>1/2$, one can easily verify 
$$
(\lvert{u_n}\rvert^2)_x
\rightarrow
(\lvert{u}\rvert^2)_x
\quad\text{in}\quad
H^{-1/2}(\mathbb{R}), 
\qquad
u_n\psi
\rightarrow
u\psi
\quad\text{in}\quad
H^{1/2}(\mathbb{R}). 
$$
This shows 
$$
(\lvert{u_n}\rvert^2)_xu_n \longrightarrow (\lvert{u}\rvert^2)_xu
\quad\text{in}\quad
\mathscr{D}^\prime(\mathbb{R}) 
\quad\text{as}\quad
n \rightarrow \infty.   
$$
This completes the proof. 
\end{proof}
Theorem~\ref{theorem:free} and Lemma~\ref{theorem:gauge-transform} 
prove the following. 
\begin{theorem}
\label{theorem:special}
Let 
$\theta\geqslant1$, 
$s\geqslant1/2$ 
and 
$\ep>0$. 
Set $\sigma=\max\{1,s\}$. 
\begin{description}
\item[{\rm Existence}\ ] 
Suppose that $u_0{\in}H^\theta(\mathbb{R})$. 
Then, the initial value problem 
\eqref{equation:special}-\eqref{equation:speciald} 
posseses a unique solution $u{\in}C(\mathbb{R};H^\theta(\mathbb{R}))$. 
\item[{\rm Analyticity}\ ] 
Moreover, 
if $\exp(\ep\langle{x}\rangle^{1/s})u_0{\in}H^\theta(\mathbb{R})$, 
then for any $T>0$ there exist positive constants $M$ and $\rho$ such that 
for $t\in[-T,T]\setminus\{0\}$
$$
\lVert
\langle{x}\rangle^{-\alpha-2m}
\p_t^m\p^\alpha
u(t)
\rVert_\theta
\leqslant
M\rho^{\alpha+2m}\lvert{t}\rvert^{-\alpha-2m}
m!^{2s}\alpha!^\sigma.
$$
\end{description}
\end{theorem}
\begin{proof}[Proof of Theorem~\ref{theorem:special}, Existence]
Set 
$v=e^{it\p^2}(e^{-ia\phi_0}u_0)$, 
$u=e^{ia\phi}v$, 
$$
\phi_0(x)
=
\int_{-\infty}^x
\lvert{u_0(y)}\rvert^2
dy,
\quad
\phi(t,x)
=
\int_{-\infty}^x
\lvert{v(t,y)}\rvert^2
dy. 
$$
Here we remark that 
$\lvert{u(t,x)}\rvert=\lvert{v(t,x)}\rvert$, 
$\lVert{u(t)}\rVert=\lVert{v(t)}\rVert$. 
Lemma~\ref{theorem:gauge-transform} shows that 
$v\mapsto{u}$ is homeomorphic on $C(\mathbb{R};H^\theta(\mathbb{R}))$. 
Pick up a sequence 
$\{v_0^{(n)}\}_{n=1}^\infty\subset\mathscr{S}(\mathbb{R})$ 
satisfying 
$$
\lVert
v_0^{(n)}
-
e^{-ia\phi_0}u_0
\rVert_\theta
\longrightarrow
0
\quad\text{as}\quad
n\rightarrow\infty.
$$
Set 
$v^{(n)}=e^{iy\p^2}v_0^{(n)}$, 
$u^{(n)}=e^{ia\phi^{(n)}}v^{(n)}$, 
$$
\phi^{(n)}(t,x)
=
\int_{-\infty}^x
\lvert{v^{(n)}(t,y)}\rvert^2
dy,
\quad
\phi_0^{(n)}(x)
=
\int_{-\infty}^x
\lvert{v_0^{(n)}(y)}\rvert^2
dy.
$$
Since $v^{(n)}{\in}C^\infty(\mathbb{R};\mathscr{S}(\mathbb{R}))$, 
$u^{(n)}$ solves 
\begin{align*}
  u^{(n)}_t-iu^{(n)}_{xx}
& =
  2a(\lvert{u^{(n)}}\rvert^2)_xu^{(n)}
  +
  ia^2\lvert{u^{(n)}}\rvert^4u^{(n)},
\\
  u^{(n)}(0,x)
& =
  e^{ia\phi_0^{(n)}(x)}v_0^{(n)}(x). 
\end{align*}
Obviously, 
$$
\lVert
v^{(n)}(t)-v(t)
\rVert_\theta
=
\lVert
v_0^{(n)}-e^{-ia\phi_0}u_0
\rVert_\theta
\longrightarrow
0
\quad\text{as}\quad
n\rightarrow\infty
$$
for any $t\in\mathbb{R}$. 
Using Lemma~\ref{theorem:gauge-transform} again, we deduce that 
$$
\lVert
u^{(n)}(t)-u(t)
\rVert_\theta
\longrightarrow
0
\quad\text{as}\quad
n\rightarrow\infty
$$
for any $t\in\mathbb{R}$, 
and that $u$ is a solution to 
\eqref{equation:special}-\eqref{equation:speciald}. 
The uniquness of $v$ implies the uniquness of $u$. 
\end{proof}
\begin{proof}[Proof of Theorem~\ref{theorem:special}, Analyticity]
Fix arbitrary $T>0$. 
We remark that the solution $u$ is represented by 
$u=e^{ia\phi}e^{it\p^2}(e^{-ia\phi_0}u_0)$. 
By using Theorem~\ref{theorem:free} and Lemma~\ref{theorem:gauge-transform}, 
we heve
\begin{equation}
\lVert
\langle{x}\rangle^{-\alpha-2m}
\p_t^m\p^\alpha
v(t)
\rVert_\theta
\leqslant
M\rho^{\alpha+2m}\lvert{t}\rvert^{-\alpha-2m}
(\alpha+2m)!^s
\label{equation}
\end{equation}
for $t\in[-T,T]\setminus\{0\}$. 
Note that 
$$
\lVert{u_0}\rVert
=
\lVert{v_0}\rVert
\leqslant
\lVert{v_0}\rVert_\theta
=
\lVert{v(t)}\rVert_\theta
\leqslant
M.
$$
First we shall show that 
for $\alpha+2m\geqslant1$ and $t\in[-T,T]\setminus\{0\}$ 
\begin{equation}
\lVert
\langle{x}\rangle^{-\alpha-2m}
\p_t^m\p^\alpha
\phi(t)
\rVert_\theta
\leqslant
C_\theta
M^2
(2\rho)^{\alpha+2m-1}
\lvert{t}\rvert^{-\alpha-2m+1}
(\alpha+2m)!^s.
\label{equation:84} 
\end{equation}
When $m=0$ and $\alpha\geqslant1$, 
$$
\p^\alpha\phi
=
\p^{\alpha-1}\lvert{v}\rvert^2
=
\sum_{\beta=0}^{\alpha-1}
\frac{(\alpha-1)!}{\beta!(\alpha-\beta-1)!}
\p^{\beta}v
\p^{\alpha-\beta-1}\bar{v}. 
$$
By using this formula and the fact that $H^\theta(\mathbb{R})$ is an algebra, 
we have 
\begin{align}
  \lVert
  \langle{x}\rangle^{-\alpha}
  \p^\alpha\phi(t)
  \rVert_\theta
& \leqslant
  \sum_{\beta=0}^{\alpha-1}
  \frac{(\alpha-1)!}{\beta!(\alpha-\beta-1)!}
  \lVert
  \langle{x}\rangle^{-\alpha}
  \p^{\beta}v
  \p^{\alpha-\beta-1}\bar{v}
  \rVert_\theta
\nonumber
\\
& \leqslant
  C_\theta
  \sum_{\beta=0}^{\alpha-1}
  \frac{(\alpha-1)!}{\beta!(\alpha-\beta-1)!}
  \lVert
  \langle{x}\rangle^{-\alpha+1}
  \p^{\beta}v
  \p^{\alpha-\beta-1}\bar{v}
  \rVert_\theta
\nonumber
\\
& \leqslant
  C_\theta
  \sum_{\beta=0}^{\alpha-1}
  \frac{(\alpha-1)!}{\beta!(\alpha-\beta-1)!}
  \lVert
  \langle{x}\rangle^{-\beta}
  \p^{\beta}v
  \rVert_\theta
  \lVert
  \langle{x}\rangle^{-\alpha+\beta+1}
  \p^{\alpha-\beta-1}\bar{v}
  \rVert_\theta
\nonumber
\\
& \leqslant
  C_\theta
  M^2
  \rho^{\alpha-1}\lvert{t}\rvert^{-\alpha+1}
  \sum_{\beta=0}^{\alpha-1}
  \frac{(\alpha-1)!}{\beta!(\alpha-\beta-1)!}
  \beta!^s
  (\alpha-\beta-1)!^s
\nonumber
\\
& \leqslant 
  C_\theta
  M^2
  \rho^{\alpha-1}\lvert{t}\rvert^{-\alpha+1}
  \alpha!^s
  \sum_{\beta=0}^{\alpha-1}
  \frac{(\alpha-1)!}{\beta!(\alpha-\beta-1)!}
\nonumber
\\
& =
  C_\theta
  M^2
  (2\rho)^{\alpha-1}\lvert{t}\rvert^{-\alpha+1}
  \alpha!^s
\label{equation:85}  
\end{align}
for $\alpha\geqslant1$ and $t\in\mathbb{R}\setminus\{0\}$. 
On the other hand, when $m\geqslant1$, 
\begin{align*}
  \p_t^m\p^\alpha\phi 
& =
  \p_t^{m-1}\p^\alpha
  \int_{-\infty}^x
  \p_t\lvert{v}\rvert^2
  dy
\\
& =
  i\p_t^{m-1}\p^\alpha
  \int_{-\infty}^x
  (v_{yy}\bar{v}-v\bar{v}_{yy})
  dy
\\
& =
  i\p_t^{m-1}\p^\alpha
  (v_{x}\bar{v}-v\bar{v}_{x})
\\
& =
  2\im
  \sum_{l=0}^{m-1}
  \sum_{\beta=0}^\alpha
  \frac{(m-1)!}{l!(m-l-1)!}
  \frac{\alpha!}{\beta!(\alpha-\beta)!}
  \p_t^l\p^{\beta+1}v
  \p_t^{m-l-1}\p^{\alpha-\beta}\bar{v}.
\end{align*}
Then, we have 
\begin{align}
& \lVert
  \langle{x}\rangle^{-\alpha-2m}
  \p_t^m\p^\alpha
  \phi(t)
  \rVert_\theta
\nonumber
\\
& \leqslant
  2
  \sum_{l=0}^{m-1}
  \sum_{\beta=0}^\alpha
  \frac{(m-1)!}{l!(m-l-1)!}
  \frac{\alpha!}{\beta!(\alpha-\beta)!}
  \lVert
  \langle{x}\rangle^{-\alpha-2m}
  \p_t^l\p^{\beta+1}v
  \p_t^{m-l-1}\p^{\alpha-\beta}\bar{v}
  \rVert_\theta
\nonumber
\\
& \leqslant
  2C_\theta
  \sum_{l=0}^{m-1}
  \sum_{\beta=0}^\alpha
  \frac{(m-1)!}{l!(m-l-1)!}
  \frac{\alpha!}{\beta!(\alpha-\beta)!}
  \lVert
  \langle{x}\rangle^{-\alpha-2m+1}
  \p_t^l\p^{\beta+1}v
  \p_t^{m-l-1}\p^{\alpha-\beta}\bar{v}
  \rVert_\theta
\nonumber
\\
& \leqslant
  2C_\theta
  \sum_{l=0}^{m-1}
  \sum_{\beta=0}^\alpha
  \frac{(m-1)!}{l!(m-l-1)!}
  \frac{\alpha!}{\beta!(\alpha-\beta)!}
\nonumber
\\
& \qquad
  \times
  \lVert
  \langle{x}\rangle^{-\beta-2l-1}
  \p_t^l\p^{\beta+1}v
  \rVert_\theta
  \lVert
  \langle{x}\rangle^{-\alpha+\beta-2m+2l+2}
  \p_t^{m-l-1}\p^{\alpha-\beta}\bar{v}
  \rVert_\theta
\nonumber
\\
& \leqslant
  2C_\theta
  M^2\rho^{\alpha+2m-1}
  \lvert{t}\rvert^{-\alpha-2m+1}
  \sum_{l=0}^{m-1}
  \sum_{\beta=0}^\alpha
  \frac{(m-1)!}{l!(m-l-1)!}
  \frac{\alpha!}{\beta!(\alpha-\beta)!}
\nonumber
\\
& \qquad
  \times
  (\beta+1+2l)!^s
  (\alpha-\beta-2m+2l-2)!^s
\nonumber
\\
& \leqslant
  2C_\theta
  M^2\rho^{\alpha+2m-1}
  \lvert{t}\rvert^{-\alpha-2m+1}
  (\alpha+2m-1)!^s
  \sum_{l=0}^{m-1}
  \sum_{\beta=0}^\alpha
  \frac{(m-1)!}{l!(m-l-1)!}
  \frac{\alpha!}{\beta!(\alpha-\beta)!}
\nonumber
\\
& =
  2^{\alpha+m}C_\theta
  M^2\rho^{\alpha+2m-1}
  \lvert{t}\rvert^{-\alpha-2m+1}
  (\alpha+2m-1)!^s
\nonumber
\\
& \leqslant
  C_\theta
  M^2(2\rho)^{\alpha+2m-1}
  \lvert{t}\rvert^{-\alpha-2m+1}
  (\alpha+2m-1)!^s
\label{equation:86a}
\end{align}
for $m\geqslant1$ and $t\in[-T,T]\setminus\{0\}$. 
Combining \eqref{equation:85} and \eqref{equation:86a}, 
we obtain \eqref{equation:84}. 
\par
Set 
$$
A=
M
+
\left\{
1+
\frac{T}{2\rho}
\right\}
{C_\theta}M^2. 
$$
If we replace $2\rho$ by $\rho$, we have 
\begin{align}
  \lVert
  \langle{x}\rangle^{-\alpha-2m}
  \p_t^m\p^\alpha
  \phi(t)
  \rVert_\theta
& \leqslant
  A
  \rho^{\alpha+2m}
  \lvert{t}\rvert^{-\alpha-2m}
  m!^{2s}\alpha!^\sigma,
\label{equation:87}
\\
  \lVert\phi(t)\rVert_{\mathscr{B}^{\theta+1/2}}
& \leqslant
  A,
\label{equation:88}
\end{align}
for $t\in[-T,T]\setminus\{0\}$ and $\alpha+2m\ne0$. 
We compute the regularity of $u$. 
The Taylor series of the exponential function gives 
\begin{align*}
  \langle{x}\rangle^{-\alpha-2m}
  \p_t^m\p^\alpha
  u 
& =
  \langle{x}\rangle^{-\alpha-2m}
  \p_t^m\p^\alpha
  (e^{ia\phi}v)
\\
& =
  \sum_{k=0}^\infty
  \frac{(ia)^k}{k!}
  \langle{x}\rangle^{-\alpha-2m}
  \p_t^m\p^\alpha
  \Bigl(\phi^ku\Bigr)
\\
& =
  \sum_{k=0}^\infty
  \frac{(ia)^k}{k!}
  \sum_{\substack{m_0+\dotsb+m_k=m \\ \alpha_0+\dotsb+\alpha_k=\alpha}}
  \frac{m!}{m_0!{\dotsm}m_k!}
  \frac{\alpha!}{\alpha_0!{\dotsm}\alpha_k!}
\\
& \qquad
  \times
  \langle{x}\rangle^{-\alpha_0-2m_0}
  \p_t^{m_0}\p^{\alpha_0}
  v
  \prod_{j=1}^k
  \langle{x}\rangle^{-\alpha_j-2m_j}
  \p_t^{m_j}\p^{\alpha_j}\phi.
\end{align*}
Applying 
Lemma~\ref{theorem:free} to $v$, 
\eqref{equation:87} to $\phi$ for $\alpha_j+2m_j\ne0$, 
and 
\eqref{equation:88} to $\phi$ fot $\alpha_j+2m_j=0$ 
respectively, 
we deduce 
\begin{align*}
& \lVert
  \langle{x}\rangle^{-\alpha-2m}
  \p_t^m\p^\alpha
  u (t)
  \rVert_\theta 
\\
& \leqslant
  \sum_{k=0}^\infty
  \frac{\lvert{a}\rvert^k}{k!}
  \sum_{\substack{m_0+\dotsb+m_k=m \\ \alpha_0+\dotsb+\alpha_k=\alpha}}
  \frac{m!}{m_0!{\dotsm}m_k!}
  \frac{\alpha!}{\alpha_0!{\dotsm}\alpha_k!}
\\
& \qquad
  \times
  \left\lVert
  \langle{x}\rangle^{-\alpha_0-2m_0}
  \p_t^{m_0}\p^{\alpha_0}
  v(t)
  \prod_{j=1}^k
  \langle{x}\rangle^{-\alpha_j-2m_j}
  \p_t^{m_j}\p^{\alpha_j}\phi(t)
  \right\rVert_\theta
\\
& \leqslant
  \sum_{k=0}^\infty
  \frac{\lvert{a}\rvert^k}{k!}
  \sum_{\substack{m_0+\dotsb+m_k=m \\ \alpha_0+\dotsb+\alpha_k=\alpha}}
  \frac{m!}{m_0!{\dotsm}m_k!}
  \frac{\alpha!}{\alpha_0!{\dotsm}\alpha_k!}
\\
& \qquad
  \times
  C_\theta^{k+1}A^{k+1}\rho^{\alpha+2m}\lvert{t}\rvert^{-\alpha+2m}
  \prod_{j=0}^k\alpha!^\sigma{m_j!^{2s}}
\\
& =
  C_\theta{A}
  \rho^{\alpha+2m}\lvert{t}\rvert^{-\alpha-2m}
  \alpha!^\sigma{m!^{2s}}
  \sum_{k=0}^\infty
  \frac{(C_\theta\lvert{a}\rvert{A})^k}{k!}
\\
& \qquad
  \times
  \sum_{\substack{m_0+\dotsb+m_k=m \\ \alpha_0+\dotsb+\alpha_k=\alpha}}
  \left\{\frac{m!}{m_0!{\dotsm}m_k!}\right\}^{1-2s}
  \left\{\frac{\alpha!}{\alpha_0!{\dotsm}\alpha_k!}\right\}^{1-\sigma}
\\
& \leqslant
  C_\theta{A}
  (2\rho)^{\alpha+2m}\lvert{t}\rvert^{-\alpha-2m}
  \alpha!^\sigma{m!^{2s}}
  \sum_{k=0}^\infty
  \frac{(C_\theta\lvert{a}\rvert{A})^k}{k!}
  \sum_{\substack{m_0+\dotsb+m_k=m \\ \alpha_0+\dotsb+\alpha_k=\alpha}}
  2^{-\alpha-m}
\\
& \leqslant
  C_\theta{A}
  (2\rho)^{\alpha+2m}\lvert{t}\rvert^{-\alpha-2m}
  \alpha!^\sigma{m!^{2s}}
  \sum_{k=0}^\infty
  \frac{(C_\theta\lvert{a}\rvert{A})^k}{k!}
  \left\{\sum_{p=0}^\infty2^{-p}\right\}^{2k+2}
\\
& \leqslant
  4C_\theta{A}
  (2\rho)^{\alpha+2m}\lvert{t}\rvert^{-\alpha-2m}
  \alpha!^\sigma{m!^{2s}}
  \sum_{k=0}^\infty
  \frac{(4C_\theta\lvert{a}\rvert{A})^k}{k!}
\\
& =
  \Bigl(4C_\theta{A}\exp(4C_\theta\lvert{a}\rvert{A})\Bigr)
  (2\rho)^{\alpha+2m}\lvert{t}\rvert^{-\alpha2-m}
  \alpha!^\sigma{m!^{2s}}, 
\end{align*}
which is desired. 
This completes the proof.
\end{proof}
{\bf Acknowledgment.}\quad
The author would like to express to the referee his deepest gratitude for
reading such a long manuscript carefully. 
%
%


\begin{thebibliography}{99}
\bibitem{chihara}H.~Chihara, 
  {\it Gain of regularity for semilinear Schr\"odinger equations}, 
  Math.\ Ann.\ {\bf 315} (1999), 529--567.
\bibitem{cm}R.~Coifman and Y. Meyer, 
  ``Au del\`a des op\'erateurs pseudo--diff\'erentiels'', 
  Ast\'erisque {\bf 57}, 1979.
\bibitem{doi1}S.-I.~Doi, 
  {\it On the Cauchy problem for Schr\"odinger type equations 
       and the regularity of the solutions}, 
  J. Math.\ Kyoto Univ.\ {\bf 34} (1994), 319--328.
\bibitem{doi2}S.-I.~Doi, 
  {\it Commutator algebra and abstract smoothing effect}, 
  J. Funct.\ Anal.\ {\bf 168} (1999), 428--469.
\bibitem{doi3}S.-I.~Doi, 
  {\it Smoothing effects for Schr\"odinger evolution equation
       and global behavior of geodesic flow}, 
  Math.\ Ann.\ {\bf 318} (2000), 355--389.
\bibitem{hnp}N.~Hayashi, P.~I.~Naumkin and P.-N.~Pipolo, 
  {\it Analytic smoothing effects for some 
       derivative nonlinear Schr\"odinger equations}, 
  Tsukuba J. Math.\  {\bf 24}  (2000),  21--34. 
\bibitem{hk}N.~Hayashi and K.~Kato,
  {\it Analyticity in time and smoothing effect of solutions to 
       nonlinear Schr\"odinger equations}, 
  Comm.\ Math.\ Phys.\  {\bf 184}  (1997), 273--300.
\bibitem{hoermander}L.~H\"ormander, 
  "The analysis of linear partial differential operators III", 
  Springer-Verlag, 1985.
\bibitem{kajitani}K.~Kajitani, 
  {\it Smoothing effect in Gevrey classes for Schr\"odinger equations, II}, 
  Ann.\ Univ.\ Ferrara Sez.\ VII, 
  Sc.\ Mat.\ Suppl.\ {\bf 45}  (2000), 173--186.
\bibitem{kw}K.~Kajitani and S.~Wakabayashi, 
  {\it Analytically smoothing effect for Schr\"odinger type equations with variable coefficients}, 
  Int.\ Soc.\ Anal.\ Appl.\ Comput.\ {\bf 5} (2000), 185--219.
\bibitem{kp}T.~Kato and G.~Ponce, 
  {\it Commutator estimates and the Euler and Navier-Stokes equations}, 
  Comm.\ Pure Appl.\ Math.\  {\bf 41} (1988), 891--907.
\bibitem{kpv}C.~E.~Kenig, G.~Ponce and L.~Vega, 
  {\it Smoothing effects and local existence theory for 
       the generalized nonlinear Schr\"odinger equations}, 
  Invent.\ math.\ {\bf 134} (1998), 489--545.
\bibitem{kumanogo}H.~Kumano-go, 
  ``Pseudo-Differential Operators'', 
  The MIT Press, 1981.
\bibitem{mns}A.~Martinez, S.~Nakamura, V.~Sordoni, 
  {\it Analytic smoothing effect for the Schr\"odinger equation 
       with long-range perturbation}, 
  Comm.\ Pure Appl.\ Math.\ {\bf 59} (2006), 1330--1351.
\bibitem{mrz}Y.~Moromoto, L.~Robbiano and C. Zuily, 
  {\it Remarks on the analytic smoothing effect for 
       the Schr\"odinger equation}, 
  Indiana Univ.\ Math.\ J. (to appear).
\bibitem{nagase}M.~Nagase, 
  {\it The $L^p$-boundedness of pseudo-differential operators 
       with non-regular symbols}, 
  Comm.\  Partial Differential Equations {\bf 2} (1977), 1045--1061.
\bibitem{rz1}L.~Robbiano and C. Zuily, 
  {\it Microlocal analytic smoothing effect for the Schr\"odinger equation}, 
  Duke Math.\ J. {\bf 100} (1999), 93--129.
\bibitem{rz2}L.~Robbiano and C. Zuily, 
  {\it Effet r\'egularisant microlocal analytique pour 
       l'\'equation de Schr\"odinger: le cas des donn\'ees oscillantes}, 
  Comm.\ Partial Differential Equations {\bf 25}, (2000), 1891--1906. 
\bibitem{rolvung}C.~Rolvung, 
  {\it Nonisotropic Schr\"odinger equations}, 
  Thesis, The University of Chicago, 1998.
\bibitem{szeftel}J.~Szeftel, 
  {\it Microlocal dispersive smoothing for the nonlinear Schr\"odinger equation}, 
  SIAM J. Math.\ Anal.\ {\bf 37} (2005), 549--597.
\bibitem{takuwa}H.~Takuwa, 
  {\it Analytic smoothing effects for a class of dispersive equations}, 
  Tsukuba J. Math.\ {\bf 28}(2004), 1--34.
\bibitem{taylor}M.~E.~Taylor, 
  ``Pseudodifferential Operators'', 
  Princeton University Press, 1981.
\bibitem{taylor2}M.~E.~Taylor, 
  ``Pseudodifferential operators and nonlinear PDE'', 
  Progress in Math.\ {\bf 100}, 
  Birkh\"auser, 1991. 
\bibitem{t'joen}L.~T'Joen, 
  {\it Effects r\'egularisants et 
       existence locale pour l'equation de 
       Schr\"odinger non-lin\'eaire 
       \`a coefficients variable}, 
  Comm.\ Partial Differential Equations 
  {\bf 27} (2002), 527--564. 
\end{thebibliography}
\end{document}